\documentclass[a4paper,10pt, reqno, oneside]{amsart}

\evensidemargin
\oddsidemargin
\usepackage{amsmath, amsthm, latexsym, amssymb}
\usepackage[dvips]{color}

\usepackage{calc,subfigure, bm}
\usepackage{graphicx}
\usepackage{tikz}

\newcommand{\e}{\varepsilon}

\renewcommand{\phi}{\varphi}

\newcommand{\id}{\one}

\newcommand{\ssup}[1] {{\scriptscriptstyle{({#1}})}}

\newcommand{\one}{{\mathbf 1}}

\newcommand{\Prob}{{\mathrm{Prob}}}

\newcommand{\R}{\mathbb R}
\newcommand{\Z}{\mathbb Z}

\newcommand{\N}{\mathbb N}

\newcommand{\E}{\mathbb E}
\renewcommand{\P}{\mathbb P}

\newcommand{\Za}{Z^{\ssup 1}_t}
\newcommand{\Zb}{Z^{\ssup 2}_t}
\newcommand{\Zc}{Z^{\ssup i}_t}
\newcommand{\Zaa}{\hat Z^{\ssup 1}_t}
\newcommand{\Zbb}{\hat Z^{\ssup 2}_t}
\newcommand{\Zcc}{\hat Z^{\ssup i}_t}

\newtheorem{theorem}{Theorem}[section]

\newtheorem{lemma}[theorem]{Lemma}

\newtheorem{prop}[theorem]{Proposition}
\newcommand{\heap}[2]  {\genfrac{}{}{0pt}{}{#1}{#2}}

\newcounter{remnr}
\newenvironment{remark}{\refstepcounter{remnr}
{\sf Remark~\arabic{remnr}.\ }\nopagebreak  }
{\nopagebreak {\hfill{$\diamond$}}\\ }

\renewcommand{\phi}{\varphi}

\renewcommand{\P}{\mathbb{P}}
\renewcommand{\E}{\mathbb{E}}

\usepackage[backend=bibtex,doi=false,isbn=false,url=false,firstinits=true]{biblatex}
\renewbibmacro{in:}{}
\bibliography{twosite}
\usepackage{hyperref}
\usepackage{enumerate}
\numberwithin{equation}{section}

 \usepackage[foot]{amsaddr}

\begin{document}

\title[Delocalising the PAM through partial duplication]{Delocalising the parabolic Anderson model \\ through partial duplication of the potential}
\author{Stephen Muirhead$^1$}
\address{$^1$Mathematical Institute, University of Oxford}
\email{muirhead@maths.ox.ac.uk}
\author{Richard Pymar$^2$}
\address{$^2$Department of Statistics, University College London (Currently: Department of Economics, Mathematics and Statistics, Birkbeck, University of London)}
\email{r.pymar@bbk.ac.uk}
\author{Nadia Sidorova$^3$}
\address{$^3$Department of Mathematics, University College London}
\email{n.sidorova@ucl.ac.uk}
\subjclass[2010]{60H25 (Primary) 82C44, 60F10 (Secondary)}
\keywords{Parabolic Anderson model, localisation}
\thanks{The first author was supported by the Engineering \& Physical Sciences Research Council (EPSRC) Fellowship EP/M002896/1 held by Dmitry Belyaev. The second author was supported by the EPSRC Grant EP/M027694/1 held by Codina Cotar. The authors thank an anonymous referee for helpful comments and corrections.}
\date{\today}

\begin{abstract}
The parabolic Anderson model on $\Z^d$ with i.i.d.\ potential is known to completely localise if the distribution of the potential is sufficiently heavy-tailed at infinity. In this paper we investigate a modification of the model in which the potential is partially duplicated in a symmetric way across a plane through the origin. In the case of potential distribution with polynomial tail decay, we exhibit a surprising phase transition in the model as the decay exponent varies. For large values of the exponent the model completely localises as in the i.i.d.\ case. By contrast, for small values of the exponent we show that the model may delocalise. More precisely, we show that there is an event of non-negligible probability on which the solution has non-negligible mass on two sites.
\end{abstract}
\maketitle


\section{Introduction}

\subsection{Delocalising the parabolic Anderson model}
Given a potential field $\xi : \Z^d \to \R$, the parabolic Anderson model (PAM) is the solution to the Cauchy problem with localised initial condition 
\begin{align}
\label{eq:PAM} \partial_t u(t,z) &=\Delta u(t,z)+\xi(z)u(t,z), & (t,z)\in (0,\infty)\times \Z^d,\\
\nonumber u(0,z)&=\one_{\{0\}}(z), & z\in\Z^d,
\end{align} 
where $\Delta$ is the discrete Laplacian acting on functions $f:\Z^d\to\R$ by
\begin{align*}
(\Delta f)(z)=\sum_{|y-z|=1}(f(y)-f(z)), \qquad z\in\Z^d, 
\end{align*}
with $|\cdot|$ the standard $\ell_1$ distance. The PAM models the competition between smoothing effects, generated by the Laplacian, and roughening effects, generated by the potential. It is well known that if the potential $\xi$ is sufficiently inhomogeneous, the PAM may undergo a process of \textit{localisation} in which its solution is eventually concentrated, at typical large times, on a small number of spatially disjoint clusters of sites. Indeed, if $\xi$ is an i.i.d.\ random field with the law of $\xi(\cdot)$ sufficiently heavy-tailed at infinity, the solution is known to eventually concentrate on a single site with overwhelming probability, i.e.\ there exists a $\Z^d$-valued process $Z_t$ such that, as $t \to \infty$,
\[   \frac{u(t, Z_t)}{\sum_{z \in \Z^d} u(t, z)} \to 1 \quad \text{in probability.} \]
In this case we say that the PAM \textit{completely localises}.

\smallskip

While there are many results in the literature establishing localisation in the PAM in various settings (see Section~\ref{s:PAM} for an overview), our understanding of the \textit{absence} of localisation is much less well-developed. In the case that the potential $\xi$ is a random field, there are at least two features of $\xi$ which may prevent complete localisation in the PAM. First, the potential may be too homogeneous on large scales -- too close to a constant potential -- for sharp peaks in the solution to form. Second, even if the potential is sufficiently inhomogeneous, complete localisation may be prevented by the presence of `duplicated' regions in which the potential is very similar; in this case, the solution may have no reason to favour one such region over another.

\smallskip

This paper is motivated by the following question:
\begin{quote}
\textit{Given a random potential for which the PAM completely localises, what kind of `duplication' of the potential will cause complete localisation to fail?}
\end{quote}

Of course, there are trivial ways to prevent complete localisation by introducing duplication. For instance, if the potential is symmetric about some plane through the origin then $u(t,\cdot)$ is also symmetric about this plane, and so complete localisation cannot occur. This paper considers a model of \textit{partial duplication} in which we pick a fraction $p \in (0, 1)$ of the sites to duplicate across the plane of symmetry. It turns out that this model exhibits a rich phenomenon of delocalisation; indeed, if the potential is i.i.d.\ with Pareto distribution (i.e.\ with polynomial tail decay), we show that the model exhibits a phase transition in the Pareto parameter. 

\subsection{Localisation in the PAM}
\label{s:PAM}
The study of localisation in the PAM has received much attention in recent years. This began with the seminal paper \cite{GM90} and is by now well-understood, see \cite{GW05,Konig16,Moerters11} for surveys. In the i.i.d.\ case, for a wide class of potentials with unbounded tails it is known that the solution to the PAM is concentrated at typical large times on a small number of spatially disjoint clusters of sites, known as \textit{islands}. The shape of the potential and the solution $u(t,\cdot)$ on these islands was first studied in \cite{GKM07} for the case of double-exponential potentials. More recently, it has been shown that for sufficiently heavy-tailed potentials (Pareto \cite{KLMS}, exponential \cite{LM12}, Weibull \cite{FM,ST}), the solution exhibits the strongest possible form of localisation: complete localisation. In \cite{MP} this was shown to also be the case for a model that replaced the Laplacian with the generator of a trapped random walk. By contrast, in very recent work \cite{BKS16} it has been shown that in the double-exponential case the PAM localises on a single connected island, rather than on a single site. This has confirmed the long standing conjecture that, in the i.i.d.\ case, potentials with double-exponential tail decay form the boundary of the complete localisation universality class.

\smallskip

The model we consider is an example of the PAM in a random potential that has spatial correlation. To the best of our knowledge, the only previous work that has considered the PAM with correlated potential in a discrete setting is \cite{GM00}, in which the motivation was to more accurately model a physical system by introducing long-range correlations. The main result in that paper is an asymptotic formula for moments of the total solution; this shows that the solution is intermittent in a certain weak sense, but is not precise enough to determine the localisation/delocalisation properties of the model.

\subsection{The PAM with partially duplicated potential}
In this section we formally introduce the \textit{PAM with partially duplicated potential} that is the object of our study. For the remainder of the paper we fix $d=1$. This avoids certain additional complications that arise in higher dimensions, while preserving the phenomena that we seek to investigate; we comment on the nature of these complications in Section \ref{s:heu}.

\smallskip

We begin by introducing the partially duplicated potential~$\xi$. Define an auxiliary random field $\xi_0: \Z \to [1, \infty)$ consisting of independent Pareto random variables with parameter $\alpha>0$, that is, with distribution function 
\begin{align*}
F(x)=1-x^{-\alpha}, \quad x \ge 1.
\end{align*}
Fix a parameter $p\in (0,1)$ that controls the density of duplicated sites. Abbreviate $\N_0 = \N\cup\{0\}$, and define a random field $\xi : \Z \to [1, \infty)$ by setting $\xi(n) =\xi_0(n)$ for each $n\in\N_0$ and, for each $n\in \N$, independently setting
\[
\xi(-n) = \begin{cases}
\xi_0(n)&\mbox{with probability $p$},\\
\xi_0(-n)&\mbox{otherwise}.
\end{cases}
\]
We henceforth refer to $\xi$ as the potential, and denote its corresponding probability and expectation by $\Prob$ and $\mathrm{E}$ respectively. 

\smallskip

The model that we consider is the parabolic Anderson model on $\Z$ -- i.e.\ the solution of equation \eqref{eq:PAM} -- with the partially duplicated potential $\xi$. It follows from~\cite{GM90} by the same argument as in the i.i.d.\! case that the solution exists provided that $\alpha>1$, and is given by the Feynman-Kac formula
\begin{align*}
u(t,z)=\E\Big[\exp\Big\{\int_0^t\xi(X_s)ds\Big\}\one\{X_t=z\}\Big], \qquad (t,z)\in (0,\infty)\times\Z,
\end{align*}
where $(X_t)_{t\ge 0}$ is a continuous-time random walk on $\Z$ with generator $\Delta$ started at the origin and $\P$ and $\E$ are its corresponding probability and expectation. We denote by 
\begin{align*}
U(t) = \sum_{z\in\Z}u(t,z)
\end{align*}
the total mass of the solution.

\smallskip

\subsection{The phase transition in the model}
We are now ready to introduce our results. Let $D =\{z \in \Z:\xi(z)=\xi(-z)\}$ denote the set of integers whose potential values are \emph{duplicated}, and $E = \Z \setminus D$ the set of positive integers whose potential values are unique (or \emph{exclusive}) to them. For each $t>0$ and $z\in \Z$, define the functional
\begin{align*}
\Psi_t(z) =\xi(z)-\frac{|z|}{t}\log\xi(z).
\end{align*}
Notice that $\Psi_t$ represents a balance between the local potential value and a `penalty term' which increases in the distance to the origin; it turns out that $\Psi_t$ is a good approximation for the asymptotic growth rate of the high peaks of the solution of the PAM, in the sense that, for a high peak centred at $z \in \Z^d$,
\[ \frac{1}{t}  \log  u(t, z) \approx  \Psi_t(z)   , \]
see \cite{KLMS} for example. For each $t > 0$, let $\Omega_t$ be the set of maximisers of $\Psi_t$; in Lemma~\ref{exi} we prove that either $\Omega_t = \{z\}$ for some $z \in E$, or $\Omega_t = \{-z, z\}$ for some $z \in D$. Define $\mathfrak{D}_t =\{|\Omega_t| = 2\}$ to be the event that the maximisers of $\Psi_t$ are duplicated; an example of this event and its complement are depicted in Figure \ref{fig:d}. 

\smallskip 

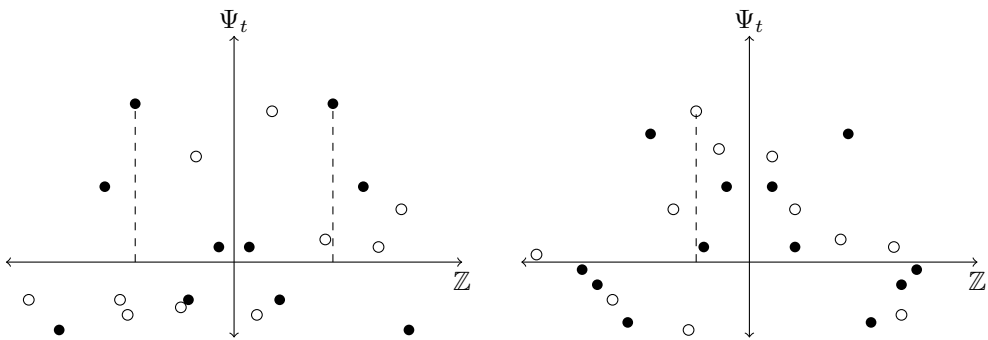
\begin{figure}[ht]
\begin{tikzpicture}
\draw [<->] (-3,0) -- (3,0);
\draw [<->] (0,-1) -- (0,3);
\node [below] at (3, 0) {$\Z$};
\node at (0, 3.2) {$\Psi_t$};
\foreach \Point in {(0.2, 0.2), (-0.2, 0.2), (0.6, -0.5), (-0.6, -0.5),(1.3, 2.1),(-1.3, 2.1), (1.7, 1), (-1.7, 1), (2.3, -0.9),(-2.3, -0.9)}{
    \fill \Point circle[radius=2pt];}
 \foreach \Point in {(0.3, -0.7), (-0.5, 1.4), (0.5, 2), (-0.7, -0.6),(1.2, 0.3),(-1.4, -0.7), (1.9, 0.2), (-1.5, -0.5), (2.2, 0.7),(-2.7, -0.5)}{
    \draw \Point circle[radius=2pt];}
   \draw[dashed] (1.3, 0) -- (1.3, 2.1);
   \draw[dashed] (-1.3, 0) -- (-1.3, 2.1);   
\end{tikzpicture}
\hspace{0.3cm}
\begin{tikzpicture}
\draw [<->] (-3,0) -- (3,0);
\draw [<->] (0,-1) -- (0,3);
\node [below] at (3, 0) {$\Z$};
\node at (0, 3.2) {$\Psi_t$};
\foreach \Point in {(0.3, 1), (-0.3, 1), (0.6, 0.2), (-0.6, 0.2),(1.6, -0.8),(-1.6, -0.8), (1.3, 1.7), (-1.3, 1.7), (2, -0.3),(-2, -0.3), (2.2, -0.1),(-2.2, -0.1)}{
    \fill \Point circle[radius=2pt];}
 \foreach \Point in {(0.6, 0.7), (0.3, 1.4), (-0.4, 1.5), (-0.7, 2),(1.2, 0.3),(-1, 0.7), (1.9, 0.2), (-1.8, -0.5), (2, -0.7),(-2.8, 0.1), (-0.8, -0.9)}{
    \draw \Point circle[radius=2pt];}
      \draw[dashed] (-0.7, 0) -- (-0.7, 1.96);
\end{tikzpicture}
\caption{An example of the event $\frak{D_t}$ (left) and its complement (right). The filled and empty circles represent the values of $\Psi_t$ for points in $D$ and and $E$ respectively; we have only plotted the top order statistics of $\Psi_t$. The dashed lines mark out the sites in $\Omega_t$.}
\label{fig:d}
\end{figure}

Our first result is to show that, for all values of the Pareto parameter $\alpha > 1$, the model always localises on the set $\Omega_t$. We also show that the event $\mathfrak{D}_t$ has non-negligible probability. Of course, outside the event $\mathfrak{D}_t$ this is already enough to conclude that the model completely localises. 

\begin{theorem}[Localisation of the model]
\label{main0}
Let $\alpha>1$. As $t\to\infty$, 
\begin{align}
\label{twopoints}
\frac{1}{U(t)}\sum_{z\in\Omega_t}u(t,z)\to 1
\qquad\text{in probability}.
\end{align}
Moreover, as $t\to\infty$, 
\begin{align*}
\Prob(\mathfrak{D}_t)\to \frac{p}{2-p}.
\end{align*}
\end{theorem}

Our next two results establish the following phase transition in the model. If $\alpha \in (1, 2)$, then on the event $\mathfrak{D}_t$ the two sites in $\Omega_t$ both have a non-negligible proportion of the solution; in other words the model \textit{delocalises}. By contrast, if $\alpha \ge 2$ only one site in $\Omega_t$ has a non-negligible proportion of the solution; in other words, the model completely localises whether the event $\mathfrak{D}_t$ holds or not. Surprisingly, the critical value of $\alpha = 2$ does not depend on the value of~$p$. To state these result, let $\Za \in \Omega_t$, with $\Za$ chosen to be positive on the event $\mathfrak{D}_t$. 

\begin{theorem}[Delocalisation in the case $\alpha \in (1, 2)$]
\label{main1}
Let $\alpha \in (1,2)$. As $t\to\infty$,
\begin{align*}
\mathcal{L}\Big(\frac{u(t,\Za)}{u(t,-\Za)} \: \Big| \: \mathfrak{D}_t\Big)\Rightarrow \mathcal{L}(\Upsilon),
\end{align*}
where $\Upsilon$ is a random variable with positive density on $\R_+$, $\mathcal{L}(\cdot)$ denotes the law of a random variable, and $\Rightarrow$ denotes weak convergence. In Section \ref{s:up} we give an explicit construction of the random variable $\Upsilon$.
\end{theorem}

\begin{theorem}[Complete localisation in the case $\alpha \ge 2$]
\label{main2}
Let $\alpha\ge 2$. As $t\to\infty$,
\begin{align*}
\Big|\log \frac{u(t,\Za)}{u(t,-\Za)}\Big|\to\infty \qquad \text{in probability.} 
\end{align*} 
\end{theorem}

\begin{remark}
At first glance it may seem counter-intuitive that delocalisation occurs for small, rather than large, values of~$\alpha$, since by analogy with the i.i.d.\ case we might expect that the heavier the tails of the potential, the stronger the localisation. However, in our model it is precisely the strengthening of the concentration effect for small $\alpha$ which results in delocalisation. 

\smallskip

To explain this, consider that if $\alpha$ is smaller, the advantage of the sites in $\Omega_t$ relative to other sites is increased. We show that, if $\alpha$ is small enough, this advantage is so great that the impact of the other potential values (at sites closer to the origin than $\Za$) is minimal, and the solution cannot readily distinguish between the sites in $\Omega_t$. On the other hand, for large values of $\alpha$ the advantage is less pronounced, and the fluctuations in the other potential values eventually force one of the sites in~$\Omega_t$ to be significantly more beneficial than the other. In the next subsection, we give some heuristics for why the transition occurs at $\alpha = 2$.
\end{remark}

\begin{remark}
One surprising aspect of the phase transition in the model is that it is not \textit{sharp}. In particular, the random variable $\Upsilon$ in Theorem \ref{main1} does not, as might be expected, degenerate for small $\alpha$. As will be further explained in the next subsection, this is ultimately due to two different scales, arising from distinct sources, exactly cancelling each other out.
\end{remark}

\begin{remark}
The proof of the Theorem \ref{main0} is relatively straightforward, and is similar to analogous results in the i.i.d.\ case, see \cite{HMS, KLMS, ST}. The proof of Theorems \ref{main1} and \ref{main2} are much more involved, and require us to analyse the model, and indeed the PAM with i.i.d.\ potential, in much finer detail than has been done in previous work.
\end{remark}

\begin{remark}
Our main results can be recast as a demonstration of the robustness, or lack thereof, of the total mass of the solution of the PAM with i.i.d.\ potential under a resampling of some of the potential values.  More precisely, suppose $u(t, z)$ denotes the solution of the PAM on $\mathbb{Z}$ with the i.i.d.\ potential~$\xi_0$, with $U(t) = \sum_z u(t, z)$ the total mass of the solution. Now consider resampling each potential value independently with probability $q \in(0,1)$, and let $\tilde u(t, z)$ be the solution of the PAM with this resampled potential, with $\tilde{U}(t) = \sum_z \tilde{u}(t, z)$ the total mass of the solution. Then our results, suitably translated, demonstrate the following phase transition. If $\alpha \in (1,2)$, then there exists an event of non-negligible probability on which $U(t)/\tilde U(t)$ converges in distribution to a random variable with positive density on $\R_+$. By contrast, if $\alpha\ge2$, then $|\log  U(t)/\tilde U(t)| \to\infty$ in probability. 
\end{remark}

\subsection{Heuristics for the phase transition}
\label{s:heu}
We start by recalling (see above) that the high peaks of the solution have the first-order approximation
\[  \log  u(t, z) \approx t \Psi_t(z)  =    t \xi(z)- |z| \log\xi(z)  . \]
The main challenge with analysing our model, compared to i.i.d.\! case, is that on the event $\mathfrak{D}_t$ there are two maximisers of the functional $\Psi_t$. Hence the height of the peak at these sites is identical up to the first-order approximation. As a result, and in contrast to the i.i.d.\ case, in order to understand the localisation phenomena we must turn to \textit{second-order} contributions. Notice that the first-order approximation, captured by the functional $\Psi_t$, has the nice feature that it is \textit{local}, depending on the value of~$\xi$ at the site $z$ only. By contrast, the second-order contributions depend on \textit{all} the potential values \textit{along entire paths} to $\Za$ and $-\Za$. This makes them much more challenging to study.

\smallskip

To explain the phase transition in the model at $\alpha = 2$, we show that the second-order contributions undergo two distinct transitions as $\alpha$ increases, both of which, seemingly coincidentally, occur at $\alpha = 2$. The first transition is the negligibility or otherwise of non-direct paths which end at the sites in $\Omega_t$; this transition serves mainly as a extra technical difficulty in our proofs, rather than a determining factor in the phase transition of the model. The second transition is a shift in the fluctuations of the second-order contributions from the Gaussian universality class ($\alpha \ge 2$) to the $\alpha$-stable universality class ($\alpha \in (1, 2)$), and it is this which turns out to cause the phase transition of the model. 

\smallskip

These transitions are also relevant for the PAM with i.i.d.\ potential, and give a more nuanced understanding of localisation phenomena in the i.i.d.\ case than has previously been available. For example, in the case $\alpha \in (1, 2)$, our proof of the first transition establishes that the \textit{PAM path measure}, given by
\[  d \mathbb{Q}( (X_s)_{s \le t} ) = \frac{1}{U(t)} \exp\Big\{\int_0^t\xi_0(X_s)ds \Big\} \, d \P( (X_s)_{s \le t} )  ,\]
concentrates on a \textit{single geometric path} (i.e.\ the direct path to the localisation site), which is much stronger result than the complete localisation of the solution. In the case $\alpha \ge 2$, we strongly suspect that the path measure instead concentrates on a class of paths that end at the localisation site but which also contain small loops. 

\subsubsection{The first transition: Direct/non-direct paths} 
Recall that the Feynman-Kac formula allows us to consider the contribution to $U(t)$ coming from different geometric paths which start at the origin. Assuming the localisation result in Theorem \ref{main0}, we know that, for all $\alpha > 1$, the only significant contribution to $U(t)$ comes from paths which end in $\Omega_t$. In Proposition~\ref{u3} we show that, if $\alpha \in (1,2)$, the only significant contribution to $U(t)$ actually come from the \textit{direct} paths to $\Omega_t$; here we give some heuristics for why this should be true. On the other hand, if $\alpha \ge 2$, then we strongly believe that certain sets of non-direct paths \textit{do} make a non-negligible contribution to $U(t)$; since we do not need this for our main results, we do not formally prove this.

\smallskip

Assume that $\alpha \in (1,2)$ and let $y^{(t)}$ denote the direct path from the origin to $\Za$. For the purposes of keeping the calculations simple, we will show only that the contribution to $U(t)$ from paths $\Pi^{(t, +)}$ from the origin to $\Za$ obtained by adding a single loop of length two to $y^{(t)}$, anywhere along the path except at the end, are negligible with respect to the contribution to $U(t$) from the path $y^{(t)}$ itself. The same argument can be extended, with minor adaption, to cover all non-direct paths to~$\Omega_t$.

\smallskip

We can assume without loss of generality that $\Za \in \N$. For any path $y=(y_0,\ldots,y_n)$ of length $n$ we can write the contribution from $y$ at time $t$ as 
\[ U(t,y)=e^{-2t} I_{n}(t; \xi(y_0),\dots,\xi(y_{n})), \]
for a function $I_n(t; a_0,\dots,a_n)$ with a rather nice structure; see equations \eqref{uy1} and \eqref{ihat}. In Lemma~\ref{L:comp} we prove a bound on $I$ which enables us to compare $U(t, y)$ for various paths. This lemma implies that for any path $y^{(t,+)} \in \Pi^{(t, +)}$ we have
\[
\frac{U(t,y^{(t,+)})}{U(t,y^{(t)})}\le \max_{0\le j<\Za}(\xi(\Za)-\xi(j))^{-2},
\]
which reflects the fact that each extra step induces in a `penalty' of order $(\xi(\Za)-\xi(j))^{-1}$. In Proposition \ref{e0} we prove that (up to a small correction) $\xi(\Za)-\xi(j)\ge (t/\log t)^{1/(\alpha-1)}$ for any $0\le j<\Za$, and also that $|\Za|$ is asymptotically  $(t/\log t)^{\alpha/(\alpha-1)}$. Since there are no more than $2|\Za|$ such paths (there are $|\Za|$ places to add the loop and two directions the loop can go in), their total contribution is at most
\[
2|\Za|\Big(\frac{\log t}{t}\Big)^{2/(\alpha-1)}U(t,y^{(t)})\le 2\Big(\frac{t}{\log t}\Big)^{(\alpha -2)/(\alpha-1)}U(t,y^{(t)}).
\]
Notice that the exponent is negative if $\alpha < 2$, which confirms that such paths are negligible with respect to the direct path. As mentioned, we can readily extend this argument to all non-direct paths.

\subsubsection{The second transition: The universality class of fluctuations}
To keep things simple, and since the intuition is correct, we shall for now assume that, for all $\alpha > 1$, it is sufficient to consider only direct paths (even though we strongly believe that this is only true in the case $\alpha \in (1, 2)$).

\smallskip

Assume that the event $\mathfrak{D}_t$ holds and denote by $y^{(t,1)}$ the direct path to $\Za$ and $y^{(t,-1)}$ the direct path to $-\Za$. We derive in Lemma \ref{bounds} that, provided $a_n\neq a_i$ for $i\neq n$, the function $I$ satisfies
\[
I_n(t; a_0,\dots,a_n)=e^{ta_n}\prod\limits_{j=0}^{n-1}\frac{1}{a_n-a_j}
-\sum_{i=0}^{n-1}I_i(t,a_0,\dots,a_i)\prod_{j=i}^{n-1}
\frac{1}{a_n-a_j}.
\]
 In Proposition \ref{shortest} we show that the second term in the form of $I$ above can be essentially discarded when considering the direct path, thus giving
\[ U(t,y^{(t, 1)})\approx e^{t\xi(\Za)-2t}
\prod_{j=0}^{|\Za|-1}\frac{1}{\xi(\Za)-\xi( j)}, \]
and similarly for $U(t,y^{(t, -1)})$. Using the assumption that only direct paths are significant, we obtain
\begin{align}
\label{e:heur}
\log \frac{u(t,\Za)}{u(t,-\Za)}& \approx - \!\!\sum_{0\le j<|\Za|}\!\!
\Big[\log\Big(1-\frac{\xi(j)}{\xi(\Za)}\Big)-\log\Big(1-\frac{\xi(-j)}{\xi(\Za)}\Big)\Big] \\
& \nonumber \approx  \xi(\Za)^{-1}\sum_{0\le j<|\Za|}(\xi(j)-\xi(-j)) + \ldots .
\end{align}
where we have used a Taylor expansion for the logarithm in the last step (this Taylor expansion does not actually converge if $\alpha < 2$, but it does give a good insight into the scale of the fluctuations; see Section \ref{s:main1} for precise statements). Note that the summand is zero for each $j\in D$ and so in expectation there are $q|\Za|$ non-zero terms, where $q = 1-p$. At this point we have reduced the study of the ratio $u(t,\Za)/u(t,-\Za)$ to the study of fluctuations in the sum of independent (although not identically distributed) random variables, and so we may appeal to the well-developed theory of such fluctuations. 

\smallskip

In the case $\alpha \in (1,2)$, these fluctuations belong to the $\alpha$-stable universality class, and so we obtain 
\begin{align}
\label{eq:2}
\log \frac{u(t,\Za)}{u(t,-\Za)}\approx  \xi(\Za)^{-1}(q|\Za|)^{\frac1{\alpha}}Y,
\end{align}
where $Y$ is a certain non-degenerate random variable. Since we prove in Proposition \ref{e0} that
\begin{align}
\label{eq:approx}
 \xi(\Za) \approx (t/\log t)^{1/(\alpha-1)} \quad \text{and} \quad |\Za| \approx  (t/\log t)^{\alpha/(\alpha-1)}, 
 \end{align}
 the growing scales in \eqref{eq:2} exactly cancel out. Hence the ratio $u(t, \Za)/u(t, \Zb)$ remains of constant order as $t \to \infty$, and so there is a non-negligible proportion of the solution at both sites in $\Omega_t$.
 
\smallskip

In the case $\alpha > 2$, the fluctuations are instead in the Gaussian universality class, and so we obtain 
\[ \log \frac{u(t,\Za)}{u(t,-\Za)}\approx  \xi(\Za)^{-1}(q|\Za|)^{\frac1{2}}Y. \]
Using \eqref{eq:approx}, this gives that
\[ \left| \log \frac{u(t,\Za)}{u(t,-\Za)} \right| \approx   (t/\log t)^{\frac{1}{\alpha-1}(-1 +\alpha/2)}  \to \infty  . \]
The case $\alpha = 2$ is slightly more delicate, but using the extra logarithmic factor that appears in the fluctuations, we can prove that $| \log u(t,\Za)/u(t,-\Za) | \to \infty$ also in this case.

\smallskip

The above analysis also gives an indication why the model is harder to study in $d \ge 2$. Indeed, even assuming only the shortest paths to $\Omega_t$ make a non-negligible contribution to $U(t)$, since there are in general many such shortest paths we must replace \eqref{e:heur} with
\[ \log  \frac{u(t,\Za)}{u(t,-\Za)} \approx  \log  \frac{ \sum_{p \in \{\text{shortest paths to $\Za$}\}}  \prod_{0\le j<|\Za|} (\xi(\Za) - \xi(p_j))^{-1} }{\sum_{p \in \{ \text{shortest paths to $-\Za$}\}}  \prod_{0\le j<|\Za|}  (\xi(\Za) - \xi(p_j))^{-1} }  ,  \]
where $(p_j)$ denote the sites along $p$. The fluctuation theory for this expression is significantly more complicated than for \eqref{e:heur}, as it does not reduce to the study of sums of independent random variables.

\subsection{Future work}
Intuitively, the closer $p$ is to $1$, the more symmetric the model becomes and the more likely that the model delocalises for a wider class of potentials. Our results show that if $p$ is uniformly bounded away from $1$ then this intuition is not realised, since the threshold $\alpha = 2$ is the same for all values of $p \in (0 ,1)$. This leads us to wonder what happens if $p$ is not uniformly bounded away from $1$. One way to investigate this is to let $\xi(z)=\xi(-z)$ with probability $p=p(|z|)$ that depends on the distance of $z$ from the origin. We can then ask the question: how fast should $p(n)\to1$ so that, for a given value of $\alpha>2$, complete localisation fails? We conjecture that there is a critical scale for $p(n)$ such that if and only if $p\to1$ slower than this scale then complete localisation holds. We will investigate this model in a future paper. 

\bigskip

\section{Outline of proof}
\label{s:outline}
In this section we give an outline of the proof of our main results, and an overview of the rest of the paper. We assume henceforth that $\alpha > 1$.

\smallskip

\textbf{Step $1$: Trimming the path set.}
As already remarked, the Feynman-Kac formula allows us to consider contributions to $u(t,z)$ coming from various geometric paths which start at the origin and are at site $z$ at time $t$. The first step is to eliminate paths that \textit{a priori} make a negligible contribution to the solution, either because they fail to hit the sites in $\Omega_t$ or because they make too many jumps. This step is rather standard, and is similar to in \cite{HMS, KLMS, ST}.

\smallskip

We now define the \textit{a priori} negligible paths. Introduce the scales 
\begin{align*}
r_t=\Big(\frac{t}{\log t}\Big)^{\frac{\alpha}{\alpha-1}}
\qquad\text{and}\qquad
a_t=\Big(\frac{t}{\log t}\Big)^{\frac{1}{\alpha-1}},
\end{align*}
which, as suggested in \eqref{eq:approx}, are the asymptotic scales for $|\Za|$ and $\xi(\Za)$ respectively. For technical reasons, we also introduce some auxiliary positive scaling functions $f_t \to 0$ and $g_t \to \infty$ which can be thought of as being arbitrarily slowly decaying or growing. We shall need these scales to satisfy
\begin{align}
\label{fg}
g_t,1/f_t=O(\log\log\log t). 
\end{align} 
Let $R_t=|\Za|(1+f_t)$. For any set $A\subseteq \Z$ denote by $\tau_A=\inf\{t>0:X_t\in A\}$ its hitting time by the continuous-time random walk $(X_s)$. Let $J_t$ be the number of jumps of $(X_s)$ by time $t$. We decompose the total mass $U(t)$ into a significant component
\[U_0(t) =\mathrm{E}\Big[\exp\Big\{\int_0^t \xi(X_s)ds\Big\} \one\{J_t \le R_t, \tau_{\Omega_t}  < t \} \Big] \]
and a negligible component $U_1(t) = U(t) - U_0(t)$. \smallskip

In Section~\ref{s:neg} we use standard methods to prove that $U_1$ is negligible with respect to~$U$ as long as certain typical properties of $\xi$ hold. To define these properties, denote, for each $n\in\N_0$, 
\begin{align*}
\xi_n^{\ssup 1}=\max\{\xi(z):|z|\le n\}
\qquad\text{and}\qquad
\xi_n^{\ssup 2}=\max\{\xi(z):|z|\le n, \xi(z)<\xi_n^{\ssup 1}\}.
\end{align*}
Let $\Zb$ be a maximiser of $\Psi_t$ on the set $\Z \setminus \Omega_t$; we prove that $\Zb$ exists in Lemma \ref{exi}. The typical properties are contained in the event
\begin{align*}
\mathcal{E}_t
&=\Big\{r_tf_t<|\Za|<r_tg_t, \ a_tf_t<\xi(\Za)<a_tg_t, \ \Psi_t(\Za)-\Psi_t(\Zb)>a_tf_t , \\
&\phantom{aaaaaa}\Psi_t(\Za)>f_t\xi(\Za), \ \xi(\Za)=\xi_{R_t}^{\ssup 1}, \ \xi_{R_t}^{\ssup 1}-\xi_{R_t}^{\ssup 2}> a_tf_t, \ \xi(z)<\frac{|z|}{t}\log\frac{|z|}{2et} \,\,\forall\,|z|>r_tg_t
\Big\},
\end{align*}
which in particular guarantees a large gap between the value of $\Psi_t$ at sites in $\Omega_t$ and all other sites.

\smallskip

\textbf{Step $2$: Reduction to subsets of paths that end at $\Omega$.}
At this point understanding $U_0$ becomes the main goal, and we aim to find out which paths make a non-negligible contribution to it; here we make a distinction between the cases $\alpha \in (1, 2)$ and $\alpha \ge 2$ (see the heuristics in Section \ref{s:heu}).

\smallskip

The main input is a careful analysis of the properties of the function $I$ that defines the contribution to $U(t)$ for any path. To define this function precisely, denote by 
\begin{align*}
\mathcal{P}_{all}=\{y=(y_0,\dots,y_{\ell})\in \Z^{\ell+1}:\ell\in\N_0,|y_i-y_{i-1}|=1\text{ for all }1\le i\le \ell\}
\end{align*}
the set of all geometric paths on $\Z$.  For each path $y\in\mathcal{P}_{all}$, denote by $\ell(y)$ its length (counted as the number of edges). Denote by $(\tau_i)_{i\in \N_0}$
the sequence of the jump times of the continuous-time random walk $(X_t)$ and by 
\begin{align*}
P(t,y)=\{X_0=y_0, \, X_{\tau_0+\cdots+\tau_{i-1}}=y_i\text{ for all }1\le i\le \ell(y),  \, t-\tau_{\ell(y)}\le\tau_0+\cdots+\tau_{\ell(y)-1}<t\} 
\end{align*}
the event that the random walk has the trajectory $y$ up to time $t$. Let
\begin{align*}
U(t,y)=\E\Big[\exp\Big\{\int_0^t\xi(X_s)ds\Big\}\one_{P(t,y)}\Big]
\end{align*}
be the contribution of the event $P(t,y)$ to $U(t)$. By direct computation, we have 
\begin{align}
U(t,y)
&=2^{-\ell(y)}\E\Big[\exp\Big\{\sum_{i=0}^{\ell(y)-1}\tau_i\xi(y_i)+\Big(t-\sum_{i=0}^{\ell(y)-1}\tau_i\Big)\xi(y_{\ell(y)})\Big\}
\one\Big\{\sum_{i=0}^{\ell(y)-1}\tau_i<t,\sum_{i=0}^{\ell(y)}\tau_i>t\Big\}\Big]\notag\\
&=2\int_{\R_{+}^{\ell(y)+1}}\exp\Big\{ \sum_{i=0}^{\ell(y)-1}x_i \xi(y_i)+\Big(t-\sum_{i=0}^{\ell(y)-1}x_i\Big)\xi(y_{\ell(y)})
-2\sum_{i=0}^{\ell(y)}x_i\Big\}\notag\\
&\phantom{aaaaaaaaaaaaaa}\times
\one\Big\{\sum_{i=0}^{\ell(y)-1}x_i<t,\sum_{i=0}^{\ell(y)}x_i> t\Big\}\Big]dx_0\cdots dx_{\ell(y)}\notag\\
&=e^{-2t} I_{\ell(y)}(t; \xi(y_0),\dots,\xi(y_{\ell(y)})).
\label{uy1}
\end{align}  
where the function $I$ is defined by
\begin{align}
\label{ihat}
I_n(t; a_0,\dots,a_n)=e^{ta_n}\int_{\R_{+}^{n}}\exp\Big\{ \sum_{i=0}^{n-1}x_i (a_i-a_n)\Big\}
\one\Big\{\sum_{i=0}^{n-1}x_i<t\Big\}dx_0\cdots dx_{n-1},
\end{align}
for each $t>0$, $n\in\N$, and $a_0,\dots,a_n\in\R$. In particular, $I_0(t; a_0)=e^{ta_0}$. 

\smallskip

In Section \ref{s:i} we show that $I$ has a rather neat symmetric structure and study its properties. Using this understanding, in Section~\ref{s:paths} we identify the paths making non-negligible contribution to $U_0$. For $\alpha \in (1,2)$ the situation is relatively simple: in Propositions~\ref{shortest} and~\ref{u3} we show that only the direct paths to $\Omega_t$ are significant, and approximate their contribution to $U_0(t)$ by a certain product over the path. This is useful because, since each site is visited at most once, we can invoke standard fluctuation theory to analyse this product.

\smallskip

The situation is more complicated for $\alpha\ge 2$ since we strongly suspect that non-direct paths are significant. Instead we show in Proposition~\ref{prop:U0null} that, as long as certain additional typical properties of~$\xi$ hold, we can limit the significant paths to those that end at $\Omega_t$ and visit each site in $\{0\}\cup\mathcal{N}_t$ at most once, where $\mathcal{N}_t$ is a set of \textit{non-duplicated sites of high potential}. The advantage is that, after careful conditioning, it will be sufficient to study the fluctuations of the contribution from sites in~$\mathcal{N}_t$. Since these sites are visited at most once, we can again apply standard fluctuation theory. 

\smallskip

To define the set $\mathcal{N}_t$ precisely, we first introduce an additional auxiliary scaling function
\begin{align*}
\delta_t=(\log t)^{-\frac{1}{2\alpha}}
\end{align*}
which is chosen in such a way that, on the one hand, $1/\delta_t$ grows slower than $(\log t)^{\frac{1}{\alpha}}$, but on the other hand, $\log (1/\delta_t)$ grows faster than any power of $g_t$ and $1/f_t$. For each $t>0$, we then let
\begin{align}
\label{n}
\mathcal{N}_t=
\big\{z\in\Z: 0<|z|<|\Za|, z \in E, \xi(z)>\delta_t\xi(\Za)\big\}.
\end{align}
The additional typical properties we need are
\[ \mathcal{E}_t^{[2,\infty)} =\Big\{\delta_t^{-\alpha}/\log\log(1/\delta_t)<|\mathcal{N}_t|
<\delta_t^{-\alpha}\log(1/\delta_t), \inf_{z\in \mathcal{N}_t,x\in \Omega_t}|z-x|>g_t\Big\} , \]
which guarantees the set $\mathcal{N}_t$ is large enough and well-separated from $\Omega_t$; in Proposition~\ref{e2} we prove that this event holds eventually with overwhelming probability, assuming the event $\mathcal{E}_t$ also holds.

\smallskip

This analysis is already enough to finish the proof of Theorem~\ref{main0} assuming the event $\mathcal{E}_t$ holds; we complete the proof at the end of Section \ref{s:paths}. 

\smallskip

\textbf{Step $3$: Point process techniques.} In Section~\ref{s:pp} we build up a point process approach to study the high exceedences of $\xi$ and the top order statistics of the penalisation functional~$\Psi_t$. We start by proving that the potential $\xi$, properly rescaled, converges to a Poisson point process. We then use this convergence to pass certain functionals of $\xi$, including properties of $\Psi_t$, to the limit. Since this analysis involves several lengthy computations, some of the proofs are deferred to Appendix \ref{A:B}.

\smallskip

To end the section, we draw two main consequences from our point process analysis. First, we establish that the event $\mathcal{E}_t$ holds eventually with overwhelming probability. Second, we give an explicit construction for the limit random variable~$\Upsilon$ appearing in Theorem \ref{main1}; this is done via identifying it as the law of a certain time-inhomogeneous L\'{e}vy process stopped at a random time.

\smallskip

\textbf{Step $4$: Fluctuation theory for the ratio $u(t, \Za)/u(t, -\Za)$.} At this point we have assembled all the main ingredients, and all that is left is to apply fluctuation theory to analyse the ratio $u(t, \Za)/u(t, -\Za)$; here we again distinguish between the cases $\alpha \in (1, 2)$ and $\alpha \ge 2$ (see the heuristics in Section \ref{s:heu}).

 \smallskip

In Section~\ref{s:main1} we study the case $\alpha \in (1,2)$ and complete the proof of Theorem~\ref{main1}. In particular, since only direct paths contribute significantly to $U_0(t)$, and since the contribution from these paths can be approximated by a product over the path, we can use standard theory to study these fluctuations. With the aid of our point process analysis, we prove that the ratio $u(t,\Za)/u(t,-\Za)$ converges to the limit random variable we identify in Section \ref{s:pp}.
 
\smallskip

In Section~\ref{s:main2} we study the case $\alpha \ge 2$ and complete the proof of Theorem~\ref{main2}. Here we apply a central limit theorem to establish that the fluctuations in $u(t, \Za)/u(t, -\Za)$ due to the sites $\mathcal{N}_t$ (which are visited at most once) are in the Gaussian universality class; the proof of the central limit theorem is deferred to Appendix \ref{clt}. These fluctuations turn out to already be sufficient to prove that $| \log u(t,\Za)/u(t,-\Za)| \to \infty$, irrespective of the contribution due to the other sites.

\bigskip


\section{Preliminaries}
\label{s:prelim}

In this section we establish some preliminary results. First, we prove asymptotic properties of the potential $\xi$. Second, we establish the negligibility of $U_1(t)$. Lastly, we study the structure of the function $I$ introduced in \eqref{ihat}.

\subsection{Asymptotic properties of the potential}
\label{s:pot}

To begin, we establish asymptotic properties of the potential. This allows us to deduce properties of the maximisers $\Za$ and $\Zb$, and also to establish that $\mathcal{E}_t^{[2,\infty)}$ holds eventually with overwhelming probability. 

\begin{lemma} 
\label{assup}
Recall that $\xi_n^{\ssup 1} = \max_{|z| \le n} \xi(z)$. For every $\e>0$, almost surely 
\begin{align*}
n^{1/\alpha-\e}<\xi_n^{\ssup 1}<n^{1/\alpha}(\log n)^{1/\alpha+\e}
\end{align*}
eventually.
\end{lemma}

\begin{proof} According to~\cite[Lemma 3.5]{HMS}, almost surely the sequence $(\xi_0(z))_{z \in \N}$ of independent Pareto($\alpha$) random variables satisfies 
\[ \max\{\xi_0(z):|z|\le n\}<n^{1/\alpha}(\log n)^{1/\alpha+\e} \]
and
\[ \min \left\{  \max\{\xi_0(z):0\le z\le n\}, \max\{\xi_0(z):-n\le z\le 0\} \right\} >n^{1/\alpha-\e} \]
eventually for all $n$, and the result follows.
\end{proof}

\begin{lemma} 
\label{exi}
For fixed $t$, almost surely either $\Omega_t = \{z\}$ for some $z \in E$ or $\Omega_t = \{-z, z\}$ for some $z \in D$, and the same conclusion holds for the maximisers of $\Psi_t$ on the set $\Z \setminus \Omega_t$. Moreover, almost surely $\Psi_t(\Za) > \Psi_t(\Zb)>1$ eventually for all $t$.
\end{lemma}

\begin{proof} By Lemma~\ref{assup} with $0<\e<\min\{1-1/\alpha,1/\alpha\}$, for all $z$ with $|z|$ sufficiently large
\begin{align*}
\Psi_t(z)\le |z|^{1/\alpha+\e}-\frac{|z|}{t}(1/\alpha-\e)\log |z|,
\end{align*}
which is a bounded function of $|z|$. Hence $\Psi_t$ is bounded for each $t>0$. Since $\Psi_t(z)$ is a continuous random variable with no point mass, this implies the first statement.

\smallskip

For the second statement, let $z_1,\,z_2\in \mathbb{Z}_+$ be fixed sites satisfying $\xi(z_1)\wedge\xi(z_2)>1$ (such sites exist almost surely). Then $\Psi_t(z_1)\wedge\Psi_t(z_2)>1$ for all $t$ sufficiently large and so in particular
$\Psi_t(\Za)$ and~$\Psi_t(\Zb)$ are both larger than one eventually.  Again since $\Psi_t(z)$ is a continuous random variable with no point mass, this implies the second statement.

\end{proof}

\begin{prop}
\label{e2}
$\text{\rm Prob}\big(\mathcal{E}_t^{[2,\infty)} \: | \: \mathcal{E}_t \big)\to 1$ as $t\to\infty$.
\end{prop}

\begin{proof} 
Let
\begin{align*}
\mathcal{E}'_t=\Big\{r_t f_t<|\Za|<r_t g_t, \xi(\Za)>a_tf_t,\Psi_t(\Za)>f_t\xi(\Za)\Big\}.
\end{align*}
and 
\begin{align*}
\mathcal{E}''_t=\Big\{f_t|\Za|<\big|\big\{|y|<|\Za|:y \in E \big\}\big|<g_t|\Za|\Big\},
\end{align*}
and observe that we may work on the event $\mathcal{E}'_t \cap \mathcal{E}''_t$ since $\mathcal{E}'_t$ is implied by $\mathcal{E}_t$ and $\Prob(\mathcal{E}''_t | \mathcal{E}'_t)\to 1$ by the law of large numbers.

\smallskip

For each $t>0$, denote by $\mathcal{G}_t$ the $\sigma$-algebra generated by $D$, $\Za$ and $\xi(\Za)$, and denote the conditional probability with respect to $\mathcal{G}_t$ by $\Prob_{\mathcal{G}_t}$. It is easy to see that, conditionally on $\mathcal{G}_t$, 
the events 
\begin{align*}
\big\{z\in \mathcal{N}_t\big\}_{z \in E,|z|<|\Za|}
\end{align*}
are independent.  Hence we can stochastically dominate the desired properties of $\mathcal{N}_t$ by equivalent
properties of Bernoulli trials, and use standard properties of such trials to complete the proof. For each $z \in E$, $|z|<|\Za|$, the conditional distribution of $\xi(z)$ with respect to $\mathcal{G}_t$ is the Pareto distribution with parameter $\alpha$ conditioned on $\Psi_t(z)<\Psi_t(\Za)$. Observe that
\begin{align}
\label{c1}
1\ge \int_{1}^{\infty}\frac{\alpha dy}{y^{\alpha+1}}
\one_{\{y-\frac{|z|}{t}\log y<\Psi_t(\Za)\}}
\ge \int_{1}^{\Psi_t(\Za)}\frac{\alpha dy}{y^{\alpha+1}}
=1-\Psi_t(\Za)^{-\alpha}>1/2
\end{align}
uniformly for all $z$ for all sufficiently large $t$ almost surely. Further, using $\delta_t\xi(\Za)>\delta_t f_ta_t>1$
and $\Psi_t(\Za)>f_t\xi(\Za)>\delta_t \xi(\Za)$ on 
$\mathcal{E}'_t$ eventually, we have
\begin{align*}
\int_{1}^{\infty}\frac{\alpha dy}{y^{\alpha+1}}
\one_{\{y>\delta_t\xi(\Za),y-\frac{|z|}{t}\log y<\Psi_t(\Za)\}}
&\ge \int_{\delta_t\xi(\Za)}^{f_t\xi(\Za)}\frac{\alpha dy}{y^{\alpha+1}}\\
&=\xi(\Za)^{-\alpha}\big(\delta_t^{-\alpha}-f_t^{-\alpha}\big)
>(a_tg_t\delta_t)^{-\alpha}/2
\end{align*}
and 
\begin{align*}
\int_{1}^{\infty}\frac{\alpha dy}{y^{\alpha+1}}
\one_{\{y>\delta_t\xi(\Za),y-\frac{|z|}{t}\log y<\Psi_t(\Za)\}}
\le \int_{\delta_t\xi(\Za)}^{\infty}\frac{\alpha dy}{y^{\alpha+1}}
=(\xi(\Za)\delta_t)^{-\alpha}
<(a_tf_t\delta_t)^{-\alpha}.
\end{align*}
Combining two above inequalities with~\eqref{c1} we get 
\begin{align}
\label{bin1}
(a_tg_t\delta_t)^{-\alpha}/2<\Prob_{\mathcal{G}_t}\big(z\in \mathcal{N}_t\big)< 2(a_tf_t\delta_t)^{-\alpha}
\end{align}
uniformly for all $z$ for all sufficiently large $t$ almost surely.
Using this together with the conditional independence and the properties guaranteed by $\mathcal{E}'_t$ and $\mathcal{E}''_t$ we infer that eventually
\begin{align}
\label{bin}
\text{Bin}(f_t^2 r_t, (a_tg_t\delta_t)^{-\alpha}/2)
\prec |\mathcal{N}_t|\prec
\text{Bin}(g_t^2 r_t, 2(a_tf_t\delta_t)^{-\alpha}),
\end{align}
where $\text{Bin}(n,\tau)$ denotes  a binomial random variable with parameters $n\in\N$ and $\tau\in [0,1]$, 
and $\prec$ denotes stochastic domination. By looking at the characteristic function of the binomial distribution we see that 
\begin{align*}
\frac{\text{Bin}(n_t\tau_t)}{n_t\tau_t}\Rightarrow 1
\end{align*}
as $t\to\infty$ if $n_t\tau_t\to \infty$. 
This condition is clearly satisfied by both binomial random variables in~\eqref{bin} by the choice of $\delta_t, f_t$, and $g_t$. 
To complete the proof of the inequalities on $|\mathcal{N}_t|$, it remains to notice that $\delta_t^{-\alpha}/\log\log(1/\delta_t)\ll f_t^2 r_t (a_t g_t\delta_t)^{-\alpha}/2$ and $2g_t^2 r_t (a_tf_t\delta_t)^{-\alpha}\ll \delta_t^{-\alpha}\log(1/\delta_t)$
since $r_ta_t^{-\alpha}=1$ and by the choice of $\delta_t, f_t$, and $g_t$. 

\smallskip

Similarly, the upper bound in~\eqref{bin1} also implies that, conditionally on $\mathcal{G}_t$, 
\begin{align*}
\inf_{z\in \mathcal{N}_t,x\in \Omega_t}|z-x|\succ 
\min\big\{\text{Geo}_1\big(2(a_tf_t\delta_t)^{-\alpha}\big), \text{Geo}_2\big(2(a_tf_t\delta_t)^{-\alpha}\big)\big\},
\end{align*}
where $\text{Geo}_1(\tau)$ and $\text{Geo}_2(\tau)$ denote  two independent geometric random variables with parameter $\tau\in [0,1]$ supported on $\N$. Observe that
\begin{align*}
\mathbb{P}(\text{Geo}(\tau_t)>g_t)=
(1-\tau_t)^{\lfloor g_t\rfloor}\to 1
\end{align*}
as $t\to\infty$ if $\tau_tg_t\to 0$. It remains to notice that 
$2(a_tf_t\delta_t)^{-\alpha}g_t\to 0$ as $t\to\infty$.
\end{proof}

\subsection{Eliminating the \textit{a priori} negligible paths}
\label{s:neg}

We begin by decomposing $U_1(t)$ into
\begin{align}
\label{fc1}
U'_1(t)&=\mathrm{E}\Big[\exp\Big\{\int_0^t \xi(X_s)ds\Big\}\one\{J_t>R_t\}\Big],\\
\label{fc2}
U''_1(t)&=\mathrm{E}\Big[\exp\Big\{\int_0^t \xi(X_s)ds\Big\}\one\{J_t\le R_t,\tau_{\Omega_t}\ge t\}\Big].
\end{align}
We first find a lower bound for $U$ in Lemma~\ref{lowerbound} and upper bounds for $U_1'$ and $U_1''$ in Lemmas~\ref{upperbound1} and~\ref{upperbound2} respectively, before combining these to prove the negligibility of $U_1$. This approach is standard and similar to~\cite{HMS, KLMS, ST}. 

\begin{lemma} 
\label{lowerbound}
Almost surely,
\begin{align*}
\log U(t)>t\Psi_t(\Za) - 2t + O(\log t)
\end{align*}
on the event $\mathcal{E}_t$ as $t\to\infty$.
\end{lemma}

\begin{proof} The idea of the proof as the same as of~\cite[Prop.~4.2]{KLMS}. 
Let $\rho\in (0,1]$ and $z\in\Z$, $z\neq 0$.  Following the lines of~\cite[Prop.~4.2]{KLMS},
we obtain
\begin{align}
\label{lb1}
\log U(t) > \exp\Big\{t(1-\rho)\xi(z)-|z|\log\frac{|z|}{e\rho t}-2t+O(\log|z|)\Big\}.
\end{align}
Take $z=\Za$ and $\rho=\frac{|\Za|}{t\xi(\Za)}$. Observe that on the event $\mathcal{E}_t$ this 
$\rho$ eventually belongs to $(0,1]$ since 
\begin{align*}
\frac{|\Za|}{t\xi(\Za)} < \frac{g_t r_t}{t f_t a_t}=\frac{g_t }{f_t \log t}\to 0
\end{align*}
by~\eqref{fg}. Substituting this into~\eqref{lb1} we obtain 
\begin{align*}
\log U(t) > \exp\Big\{t\xi(\Za)-|\Za|\log\xi(\Za)-2t+O(\log t)\Big\}.
\end{align*}
as required. 
\end{proof}

\begin{lemma} 
\label{upperbound1}
Almost surely, 
\begin{align*}
\log U_1'(t) < \max\Big\{t\Psi(\Zb) +o(ta_tf_t),  \xi(\Za)-R_t\log\frac{R_t}{2et} + O(t) \Big\}  
\end{align*}
on the event $\mathcal{E}_t$ as $t\to\infty$. 
\end{lemma}

\begin{proof} Observe that the number of jumps $J_t$ of the continuous-time random walk
by the time $t$ has Poisson distribution with parameter $2t$.  
Fix some $0<\varepsilon<1-1/\alpha$. We can estimate the integral in~\eqref{fc1} by $t\xi^{\ssup 1}_n$ on the event $\{J_t=n\}$ and then use Lemma~\ref{assup} to obtain the almost-sure bound
\begin{align}\label{eq:u1'}
U_1'(t)\le \sum_{n>R_t}e^{t\xi_n^{(1)}-2t}\frac{(2t)^n}{n!}\le \sum_{n>R_t}\exp\{tn^{\frac1{\alpha}+\varepsilon}-2t\}\frac{(2t)^n}{n!}.
\end{align}We now present an upper-bound for the tail of this series. Fix some $\theta>1$ and $\beta>(1-\varepsilon-1/\alpha)^{-1}$. Define $\gamma:=\beta(1-\varepsilon-1/\alpha)-1$ and note that $\gamma>0$. By Stirling's formula,
\[
n!=\sqrt{2\pi n}\left(\frac{n}{e}\right)e^{\delta(n)},\qquad\mbox{with }\lim_{n\to\infty}\delta(n)=0,
\]and so for all $n>t^\beta$ and sufficiently large $t$,
\begin{align*}
&t n^{\frac1{\alpha}+\varepsilon}+n\log(2t)-\log(n!)\le t n^{\frac1{\alpha}+\varepsilon}-n\log\frac{n}{2et}-\delta(n)\\&\le tn^{\frac1{\alpha}+\varepsilon}\left(1-\frac{n^{1-\frac1{\alpha}-\varepsilon}}{t}\log\frac{n}{2et}-\frac{\delta(n)}{tn^{\frac1{\alpha}+\varepsilon}}\right)\le
tn^{\frac1{\alpha}+\varepsilon}\left(1-t^\gamma\log\frac{t^{\beta-1}}{2e}-\frac{\delta(n)}{tn^{\frac1{\alpha}+\varepsilon}}\right)\le -\theta\log n.
\end{align*}Splitting the sum on the right of \eqref{eq:u1'} at $n=\lceil t^\beta\rceil$, and using $\sum_{n>\lceil t^\beta\rceil}n^{-\theta}=o(1)$, we have
\begin{align*}
\log U_1'(t)\le \log\Big[\sum_{n> R_t}e^{t\xi_n^{\ssup 1}-2t}\frac{(2t)^n}{n!}\Big]
< \max_{n> R_t}\Big[t\xi_n^{\ssup 1}-n\log\frac{n}{2te}\Big]+O(t).
\end{align*} 
Denote by $n_t> R_t$ the maximiser of the expression on the right-hand side, and by $z_t\in \Z$  
a point such that $\xi(z_t)=\xi_{n_t}^{\ssup 1}$. If $|z_t|\le R_t$ then $\xi(z_t)=\xi(\Za)$ on the event $\mathcal{E}_t$
and so 
\begin{align*}
\log U_1'(t) < t\xi(\Za)-n_t\log\frac{n_t}{2te}+O(t) < t\xi(\Za)-R_t\log\frac{R_t}{2te}+O(t).
\end{align*}
If $|z_t|>R_t$ then by monotonicity $n_t=|z_t|$. If in fact $|z_t|>r_tg_t$, then on the event $\mathcal{E}_t$,
\[
\log U'_t(t)<O(t),
\] whereas if $R_t<|z_t|\le r_tg_t$ then almost surely
\begin{align*}
\log U_1'(t) &< t\Psi_t(z_t)+|z_t|\log\xi(z_t)-|z_t|\log\frac{|z_t|}{2te}+O(t) \\&< t\Psi_t(z_t)+O\left(ta_t\frac{g_t\log\log t}{\log t}\right)\le t\Psi(\Zb)+o(ta_tf_t)
\end{align*}
by Lemma~\ref{assup} and \eqref{fg}.
\end{proof}

\begin{lemma} 
\label{upperbound2}
Almost surely, 
\begin{align*}
\log U_1''(t) < t\Psi(\Zb)+o(t a_t f_t)
\end{align*}
on the event $\mathcal{E}_t$ as $t\to\infty$.
\end{lemma}

\begin{proof} 
For any $n\in\N_0$, let
\begin{align*}
\zeta_n=\max\{\xi(z):|z|\le n, z \notin \Omega_t\} .
\end{align*}
Similarly to Lemma~\ref{upperbound1} we have 
\begin{align*}
\log U_1''(t)\le \log\Big[\sum_{n\le R_t}e^{t\zeta_n-2t}\frac{(2t)^n}{n!}\Big]
< \max_{n\le R_t}\Big[t\zeta_n-n\log\frac{n}{2te}\Big]+O(t).
\end{align*} 
Denote by $n_t\le R_t$ the maximiser of the expression on the right-hand side, and by $z_t\in \Z$  a point such that $\xi(z_t)=\zeta_{n_t}$. 

\smallskip

If $|z_t|<r_t (\log t)^{-2}$ then 
by monotonicity and Lemma~\ref{assup} with small 
$\e>0$
\begin{align*}
t\zeta_{n_t}-n_t\log\frac{n_t}{2te} <
t\xi_{r_t(\log t)^{-2}}^{\ssup 1}+2t
< t r_t^{1/\alpha}(\log t)^{-1/\alpha+\e}
=o(t a_t f_t).
\end{align*}
If $r_t (\log t)^{-2}\le |z_t|\le R_t$ then by monotonicity 
\begin{align*}
t\zeta_{n_t}-n_t\log\frac{n_t}{2te}
=t\xi(z_t)-|z_t|\log\frac{|z_t|}{2te}
=t\Psi(z_t)+|z_t|\log\frac{2te\xi(z_t)}{|z_t|}.
\end{align*}
Clearly $\Psi(z_t)\le \Psi(\Zb)$ since $z_t\notin\Omega_t$. 
By Lemma~\ref{assup} on the event $\mathcal{E}_t$
\begin{align*}
|z_t|\log\frac{2te\xi(z_t)}{|z_t|}
< |z_t|\log\frac{2te(\log |z_t|)^{1/\alpha+\e}}{|z_t|^{1-1/\alpha}}
= R_t O(\log\log t)=O(r_tg_t\log\log t)=o(t a_t f_t)
\end{align*}
by~\eqref{fg} as required.
\end{proof}

\begin{prop} 
\label{u12}
Almost surely, 
\begin{align*}
\frac{U_1(t)}{U(t)}\one_{\mathcal{E}_t} \to 0
\end{align*}
as $t\to\infty$.
\end{prop}

\begin{proof} 
We first claim that $U_1'(t)/U(t) \to 0$ on the event $\mathcal{E}_t$. Combining Lemmas~\ref{lowerbound} and~\ref{upperbound1} we have on $\mathcal{E}_t$
\begin{align*}
\log U_1'(t)-\log U(t) < \max\Big\{t\Psi(\Zb)+o(ta_tf_t),t\xi(\Za)-R_t\log\frac{R_t}{2et}+O(t)\Big\}-t\Psi_t(\Za).
\end{align*}
First, on the event $\mathcal{E}_t$
\begin{align}
\label{firstmax}
t\Psi_t(\Zb)-t\Psi_t(\Za)+o(ta_tf_t)<-t a_t f_t+o(ta_tf_t)\to -\infty.
\end{align}
Second, 
\begin{align*}
t\xi(\Za)-R_t\log\frac{R_t}{2et}-t\Psi_t(\Za)+O(t)
&=|\Za|\log\xi(\Za)-R_t\log\frac{R_t}{2et}+O(t)\\
&< |\Za|\log\frac{2et\xi(\Za)}{|\Za|}-f_t|\Za|\log\frac{|\Za|}{2et}+O(t).
\end{align*}
On the event $\mathcal{E}_t$, as $t\to\infty$, we have for the first term 
\begin{align*}
|\Za|\log\frac{2et\xi(\Za)}{|\Za|} < r_tg_t\log\frac{2etg_ta_t}{f_tr_t}
=r_tg_t\log\frac{2eg_t\log t}{f_t}\sim r_tg_t\log\log t
\end{align*}
and for the second term 
\begin{align*}
f_t|\Za|\log\frac{|\Za|}{2et} > f_tr_tf_t\log\frac{r_tf_t}{2et}\sim \frac{1}{\alpha-1}f_t^2r_t\log t.
\end{align*}
By~\eqref{fg} the first term is negligible with respect to the second term 
and so 
\begin{align*}
t\xi(\Za)-R_t\log\frac{R_t}{2et}-t\Psi_t(\Za)+O(t)\to -\infty,
\end{align*}
which proves the claim.

\smallskip

It remains to show that $U_1''(t)/U(t) \to 0$ on the event $\mathcal{E}_t$. Combining Lemmas~\ref{lowerbound} and~\ref{upperbound2} we have on the event $\mathcal{E}_t$
\begin{equation*}
\log U_1''(t)-\log U(t) < t\Psi_t(\Zb)-t\Psi_t(\Za)+o(ta_tf_t)
< -ta_tf_t+o(ta_tf_t)\to-\infty . \qedhere
\end{equation*}
\end{proof}

\subsection{Structure of the function $I$}
\label{s:i}

In this section we study the structure of the function $I$ introduced in \eqref{ihat}. Our point of departure is the recursion 
\begin{align}
\label{rec2}
I_n(t;a_0,\dots,a_n)=\frac{1}{a_{n}-a_{n-1}}
\Big[I_{n-1}(t; a_0,\dots,a_{n-2},a_n)-I_{n-1}(t; a_0,\dots,a_{n-2},a_{n-1})\Big]
\end{align}
whenever $a_n\neq a_{n-1}$, obtained by evaluating the integral over $x_{n-1}$ in \eqref{ihat}. By iterating this recursion we establish the following.

\begin{lemma} 
\label{bounds}
The following hold:
\begin{enumerate}
\item If $a_n \neq a_i$ for $i \neq n$ then 
\begin{align*}
I_n(t; a_0,\dots,a_n)=e^{ta_n}\prod\limits_{j=0}^{n-1}\frac{1}{a_n-a_j}
-\sum_{i=0}^{n-1}I_i(t,a_0,\dots,a_i)\prod_{j=i}^{n-1}
\frac{1}{a_n-a_j} .
\end{align*}
Moreover, if $a_0,\dots,a_n$ are pairwise distinct then
\begin{align}
\label{iii}
I_n(t; a_0,\dots,a_n)=\sum_{i=0}^{n}e^{ta_i}\prod\limits_{\heap{j=0}{j\neq i}}^{n}\frac{1}{a_i-a_j};
\end{align}
\item $I_n$ is symmetric with respect to the variables $a_0,\dots,a_n$.  
\end{enumerate}
\end{lemma}

\begin{proof}
The first statement in $(1)$ follows by induction from~\eqref{rec2}, where we apply induction to the first term in the recursion  and keep the second term. The second statement in $(1)$ also follows by induction once we notice that it is true for $n=0$ and the expression on the right hand side satisfies the recursion~\eqref{rec2}. Finally, the symmetry of $I_n$ for pairwise distinct variables follows from the symmetry of the expression on the right hand side of~\eqref{iii}. Then it extends by continuity to all variables. 
\end{proof}

We now establish two upper bounds on the function $I$. The first bounds the effect of adding additional steps onto a base path. The second bounds the effect of changing the largest value of $a_i$ along a path; for this we shall need an additional lemma that establishes `negative dependence' in the effect on $I$ due to changes in the $a_i$.

\begin{lemma}
\label{L:comp}
Let $m,n\in\mathbb{N}_0$ and suppose $a_j<a_n$ for all $0\le j< n$ and $a_j=a_n$ for all $n\le j\le n+m$. Then for any $t>0$, $0\le k\le n$, and $0\le i\le m$,
\begin{align*}
I_{n+m}(t;a_0,\ldots,a_{n+m})
&\le \frac{t^i}{i!}I_{n+m-i-k}(t;a_k,\ldots,a_{n+m-i})\prod_{j=1}^k\frac1{a_n-a_{j-1}}
\end{align*}
\end{lemma}
\begin{proof}Integrating with respect to the last $i$ variables 
we obtain 
\begin{align*}
I_{n+m}
&(t; a_0,\dots,a_{n+m})\\
&=e^{ta_n}\int_{\R_{+}^{n}}\exp\Big\{ \sum_{s=0}^{n-1}x_s (a_s-a_n)\Big\}
\one\Big\{\sum_{s=0}^{n+m-i-1}x_s+\!\!\!\!\sum_{s=n+m-i}^{n+m-1}\!\!\!\!
x_s<t\Big\}\Big]dx_0\cdots dx_{n+m-1}\\
&=e^{ta_n}\frac{1}{i!}\int_{\R_{+}^{n}}\exp\Big\{ \sum_{s=0}^{n-1}x_s (a_s-a_n)\Big\}
\one\Big\{\sum_{s=0}^{n+m-i-1}x_s<t\Big\}\Big(t-\sum_{s=0}^{n+m-i-1}x_s\Big)^idx_0\cdots dx_{n+m-i-1}\\
&\le \frac{t^i}{i!}I_{n+m-i}(t; a_0,\dots,a_{n+m-i}).
\end{align*}
Further, it follows from~\eqref{rec2} and symmetry of $I$ proved in Lemma~\ref{bounds} that 
\begin{align*}
I_{n+m-i}(t; a_0,\dots,a_{n+m-i})
&\le 
I_{n+m-i-1}(t; a_1,\dots,a_{n+m-i})\frac{1}{a_n-a_0}
\le\cdots\\
&\le 
I_{n+m-i-k}(t;a_k,\ldots,a_{n+m-i})\prod_{j=1}^k\frac1{a_n-a_{j-1}} . \qedhere
\end{align*} 
\end{proof}

Our `negative dependence' lemma requires the application of a result of \cite{bss}, which we state below. 
\begin{theorem}[{\cite[Theorem 4.1]{bss}}]
\label{T:bss}
Fix $s\in\mathbb{R}$ and $n \in \mathbb{N}$. Let $X_0,\ldots, X_{n+1}$ be independent random variables each with a log-concave density, and let $(Y_0,\ldots,Y_n)$ be a random vector satisfying
\[ \mathcal{L} \left( (Y_0,\ldots,Y_n) \right) = \mathcal{L} \left( (X_0,\ldots,X_n) \big|\, X_0+\cdots+X_{n+1} =s \right) . \]
Then for each $0\le i<j\le n$, 
\[
\E(Y_iY_j)\le \E(Y_i)\E(Y_j).
\]
\end{theorem}
\begin{proof}
First we remark that a density being log-concave is equivalent to the density being a \emph{Polya frequency function of order 2} (or PF$_2$, using the terminology from \cite{bss}). Then \cite[Theorem 4.1]{bss} implies that $(Y_1,\ldots,Y_n)$ is \emph{reverse regular of order 2 in pairs} (again using the terminology of \cite{bss}). Then the discussion following Definition 2.2 in \cite{bss} demonstrates that this implies the result.
\end{proof}
\begin{lemma}
\label{L:logmono} 
Let $n\ge 2$. For any $k\neq j$
\begin{align}
\label{jk}
\frac{\partial^2}{\partial a_k\partial a_j}\log I_n(t;a_0,\dots,a_n)\le 0.
\end{align}
\end{lemma}

\begin{proof}
By symmetry of $I$ proved in Lemma~\ref{bounds} it suffices to prove the statement for 
$j,k\neq n$. Denote ${\bf a}=(a_0,\dots,a_n)$. It is easy to see that~\eqref{jk} is equivalent to showing  that
\begin{align}
\label{multi}
I_n(t; {\bf a})^{-1}\frac{\partial^2 }{\partial a_k\partial a_j}I_n(t; {\bf a})\le 
\Big[I_n(t; {\bf a})^{-1}\frac{\partial}{\partial a_k}I_n(t; {\bf a})\Big]
\cdot\Big[I_n(t; {\bf a})^{-1}\frac{\partial }{\partial a_j}I_n(t; {\bf a})\Big]
\end{align}
Fix $t>0$. Let $W_i$, $0\le i\le n$, be independent random variables with density $c_i e^{(a_i-a_n)x}$ on $[0,t]$ and zero otherwise, 
where $c_i$ is a normalising constant, and let $W_{n+1}$ be uniform on $[0,t]$. We remark that each of $W_i$, $0\le i\le n+1$, has a log-concave density. Further, let 
$\hat W_i$, $0\le i\le n$, be defined by 
\begin{align*}
(\hat W_0,\dots,\hat W_n)\stackrel{d}{=}\Big(W_0,\dots,W_n\Big|\sum_{i=0}^{n+1} W_i=t\Big).
\end{align*}
Since the densities of $W_i$, $0\le i\le n+1$,
are log-concave, by Theorem \ref{T:bss} we have 
\begin{align*}
\mathrm{E} \big(\hat W_k\hat W_j\big)\le 
\mathrm{E} (\hat W_k)\mathrm{E} (\hat W_j). 
\end{align*}
To prove~\eqref{multi}, it suffices now to show that 
\begin{align*}
\mathrm{E} (\hat W_k)
=I_n(t; {\bf a})^{-1}\frac{\partial}{\partial a_k}I_n(t; {\bf a})
\qquad\text{and}\qquad
\mathrm{E} \big(\hat W_k\hat W_j\big)
=I_n(t; {\bf a})^{-1}\frac{\partial^2 }{\partial a_k\partial a_j}I_n(t; {\bf a}).
\end{align*}
For this note that
\begin{align*}
\mathrm{E}(\hat W_k)&=c\int_{\mathbb{R}_+^{n+1}}x_k\exp\Big\{\sum_{i=0}^{n}(a_i-a_n)x_i\Big\}\one\Big\{\sum_{i=0}^n x_i\le t\Big\}dx_0\cdots dx_n=ce^{-ta_n}\frac{\partial}{\partial a_k}I_n(t;{\bf a})
,\\
\mathrm{E}(\hat W_k \hat W_j)&=c\int_{\mathbb{R}_+^{n+1}}x_kx_j\exp\Big\{\sum_{i=0}^{n}(a_i-a_n)x_i\Big\}\one\Big\{\sum_{i=0}^n x_i\le t\Big\}dx_0\cdots dx_n=ce^{-ta_n}\frac{\partial^2}{\partial a_k\partial a_j}I_n(t;{\bf a})
,\end{align*}where $c$ is a normalising constant and thus satisfies
\[
1=c\int_{\mathbb{R}_+^{n+1}}\exp\Big\{\sum_{i=0}^{n}(a_i-a_n)x_i\Big\}\one\Big\{\sum_{i=0}^n x_i\le t\Big\}dx_0\cdots dx_n=ce^{-ta_n}I_n(t;{\bf a}).
\]Plugging this value of $c$ into the above equations gives the required identities.
\end{proof}

\begin{lemma}
\label{L:Ixtoy}
Let $n\ge 2$, $x<y$  and $a_0,\dots,a_{n-1}\in \R$ be such that 
$a_i\le y$ for all $0\le i\le n-1$. Then
\[
I_n(t;{\bf a},x)\le \frac{n}{(y-x)t}I_n(t;{\bf a},y),
\]
where ${\bf a}=(a_0,\dots,a_{n-1})$.
\end{lemma}

\begin{proof}
Since the function $s\mapsto \log I_n(t;{\bf a},s)$ is continuous we can write
\begin{align}
\label{ff1}
\frac{I_n(t;{\bf a},x)}{I_n(t;{\bf a},y)}=\exp\Big\{-\int_x^y\frac{\partial}{\partial s}\log I_n(t;{\bf a},s)ds\Big\}.
\end{align}
It follows from the definition~\eqref{ihat} of $I_n$ that 
$$\log I_n(t;{\bf a},s)=\log I_n(t;{\bf a}-{\bf y},s-y)+ty,$$ 
where ${\bf y}=(y,\dots,y)$
and hence 
\begin{align*}
\frac{\partial}{\partial s}\log I_n(t;{\bf a},s)=\frac{\partial}{\partial s}\log I_n(t;{\bf a}-
{\bf y},s-y).
\end{align*}
Since all $a_i\le y$ we can use monotonicity proved in Lemma~\ref{L:logmono} to obtain 
\begin{align}
\label{ff2}
\frac{\partial}{\partial s}\log I_n(t;{\bf a}-
{\bf y},s-y)\ge \frac{\partial}{\partial s}\log I_n(t;{\bf 0},s-y),
\end{align}
where ${\bf 0}=(0,\dots,0)$. This implies 
\[
\frac{I_n(t;{\bf a},x)}{I_n(t;{\bf a},y)}\le \exp\Big\{-\int_x^y\frac{\partial}{\partial s}\log I_n(t;{\bf 0},s-y)\,ds\Big\}=\frac{I_n(t;{\bf 0},x-y)}{I_n(t;{\bf 0},0)}.
\]
It is easy to see that 
\begin{align}
\label{ff3}
I_n(t;{\bf 0},0)=\frac{t^n}{n!}.
\end{align}
Using $y>x$ and the substitution $u_i=x_i$, $0\le i\le n-2$ and 
$u_{n-1}=x_0+\cdots+x_{n-1}$ in the definition~\eqref{ihat}
of $I_n$, we also obtain integrating over $u_{n-1}$ that
\begin{align}
\label{ff4}
I_n(t;{\bf 0},x-y)
\le e^{t(x-y)}\int_{\R_{+}^{n-1}}
\Big[\int_{-\infty}^t e^{u_{n-1}(y-x)}du_{n-1}\Big]du_0\dots du_{n-2}
=\frac{t^{n-1}}{(y-x)(n-1)!}.
\end{align}
Combining~\eqref{ff1}, \eqref{ff2}, \eqref{ff3} and~\eqref{ff4} gives the stated result.
\end{proof}

\bigskip


\section{Significant paths}
\label{s:paths}

The aim of this section is to determine which paths make a non-negligible contribution to $U_0(t)$. As described in Section \ref{s:heu}, in the case $\alpha \in (1, 2)$ we prove that only the direct paths to $\Omega_t$ are significant. In the case $\alpha \ge 2$, we can only prove the much weaker result that the significant paths are those which end at $\Omega_t$ and visit the set $\mathcal{N}_t \cup \{0\}$ at most once, where $\mathcal{N}_t$ is the set of non-duplicated sites of high potential defined in \eqref{n} (actually this is true for all $\alpha > 1$, but is not as strong as what we prove for $\alpha \in (1,2)$).

\smallskip

Assuming the event $\mathcal{E}_t$ holds, this is already enough to prove the localisation statement in Theorem~\ref{main0}; we complete this proof at the end of the section.

\subsection{The case $\alpha \in (1, 2)$: Direct paths to $\Omega_t$}
\label{s:paths1}

To prove that only direct paths are significant, we first give an approximation for the contribution made by the direct paths, and then use this approximation to show the negligibility of all other paths. Denote by $y^{\ssup{t,1}}\in\mathcal{P}_{all}$, $y^{\ssup{t,-1}}\in\mathcal{P}_{all}$ the shortest geometric paths from $0$ to $|\Za|$ and to $-|\Za|$, respectively.  

\smallskip

Before we begin, we state a small combinatorial lemma that will be used in Proposition \ref{u3} below.

\begin{lemma} 
\label{binom}
For any $n\ge 4$ and any $w\in\N_0 $,
\begin{align*}
{{n+2w}\choose{w}} < 16 n^w.
\end{align*}
\end{lemma}

\begin{proof} For $n=4$ we have ${{n+2w}\choose{w}} < 2^{n+2w}=16n^w$ for all $w$. By induction 
\begin{align*}
{{n+1+2w}\choose{w}}={{n+2w}\choose{w}}\cdot \frac{n+1+2w}{n+1+w} < 16n^w\Big(1+\frac{w}{n+1+w}\Big)
< 16(n+1)^w
\end{align*}
since 
\begin{align*}
\Big(\frac{n+1}{n}\Big)^w\ge 1+\frac w n > 1+\frac{w}{n+1+w}
\end{align*}
by Bernoulli's inequality. 
\end{proof}

\begin{prop} 
\label{shortest}
Almost surely
\begin{align*}
U(t,y^{\ssup{t,\iota}})=
\Big\{e^{t\xi(\Za)-2t}
\prod_{j=0}^{|\Za|-1}\frac{1}{\xi(\Za)-\xi(j\iota)}\Big\}
+o(1)U(t)
\end{align*}
for $\iota=\text{\rm sgn}(\Za)$ on the event $\mathcal{E}_t$ and for each $\iota\in\{-1,1\}$ on the event $\mathcal{E}_t\cap \mathfrak{D}_t$, as $t\to\infty$.
\end{prop}

\begin{proof} 
Fix $t>0$, $\iota\in\{-1,1\}$, and assume the corresponding event $\mathcal{E}_t$ or $\mathcal{E}_t\cap \mathfrak{D}_t$ holds. Denote $n=|\Za|$, $a_i=\xi(i\iota)$, $0\le i\le n$. According to~\eqref{uy1} and Lemma~\ref{bounds} we have 
\begin{align*}
U(t,y^{\ssup{t,\iota}})
&=e^{-2t} I_n(t; a_0,\dots,a_n)\\
&=e^{t\xi(\Za)-2t}\prod\limits_{j=0}^{|\Za|-1}\frac{1}{\xi(\Za)-\xi(j\iota)}
-\sum_{i=0}^{n-1}e^{-2t} I_i(t;a_0,\dots,a_i)\prod_{j=i}^{n-1}
\frac{1}{a_n-a_j}.  
\end{align*}
Observe that on $\mathcal{E}_t$, $a_n-a_j>a_tf_t>1$ eventually for all $1\le j<n$. Further, again by~\eqref{uy1} we have 
\begin{align*}
e^{-2t} I_i(t;a_0,\dots,a_i)
=U(t,w^{\ssup i})
\end{align*} 
for all $0\le i<n$, where $w^{\ssup i}$ is the shortest path to $i\iota$.  Since $\sum_{i=0}^{n-1}U(t,w^{(i)})\le U_1''(t)$, by Lemma~\ref{upperbound2} we have
\begin{align*}
\sum_{i=0}^{n-1}e^{-2t} I_i(t;a_0,\dots,a_i)\prod_{j=i}^{n-1}
\frac{1}{a_n-a_j}
\le U_1''(t) < \exp\Big\{t\Psi_t(\Zb)+o(ta_tf_t)\Big\}.
\end{align*}
Combining this with the lower bound for $U(t)$ from Lemma~\ref{lowerbound} and also taking into account that $t\Psi(\Za)-t\Psi_t(\Zb)>ta_tf_t$  we obtain 
\begin{align*}
\sum_{i=0}^{n-1}e^{-2t} I_i(t;a_0,\dots,a_i)\prod_{j=i}^{n-1}
\frac{1}{a_n-a_j}=o(1)U(t),
\end{align*}
which completes the proof.
\end{proof}

\begin{prop}
\label{u3}
Let $\alpha \in (1,2)$. Almost surely,
\begin{align*}
U_0(t)=(1+o(1))\sum_{\iota\in\{-1,1\}}U(t,y^{\ssup{t,\iota}})
\end{align*}
on the event $\mathcal{E}_t$, as $t\to\infty$.
\end{prop}

\begin{proof}
For a path $y\in \mathcal{P}_{all}$ that hits $\Omega_t$, 
let $z_t(y)\in \Omega_t$ be the first point where $y$ hits $\Omega_t$, i.e.,
\begin{align*}
z_t(y)=y_i, \quad\text{ where }i=\min\{j: y_j\in\Omega_t\}.
\end{align*}

Denote by $m_t(y)$ the number of times $\Omega_t$
is visited minus one, i.e., 
\begin{align*}
m_t(y)=|\{0\le i\le \ell(y):
y_i\in\Omega_t\}|-1. 
\end{align*}
Denote by $2w_t(y)$ the difference between 
the hitting time of $z_t(y)$ and $|\Za|$, i.e., 
\begin{align*}
w_t(y)=\frac{\min\{i:y_i=z_t(y)\}-|\Za|}{2}.
\end{align*} 
Finally, denote by $s_t(y)$ the number of points on the path 
after the first visit to $\Omega_t$ that do not belong to 
$\Omega_t$, i.e., 
\begin{align*}
s_t(y)=|\{|\Za|+2w_t(y)<i\le \ell(y):y_i\notin \Omega_t\}|.
\end{align*}
Observe that $s_t(y)\ge m_t(y)$.

\smallskip

For each $t>0$ and $m\in\N\cup\{0\}$, $w\in\N\cup\{0\}$, $s\ge m$, denote
\begin{align*}
\mathcal{P}^{t}_{m,w,s}
=\big\{y\in\mathcal{P}_{all}:
&y_0=0, m_t(y)=m, w_t(y)=w,s_t(y)=s\big\}.
\end{align*} 
Using Lemma~\ref{binom} we have 
\begin{align*}
\big|\mathcal{P}_{m,w,s}^{t}\big|
\le 2^s{|\Za|+2w\choose w}
< 16 \cdot 2^s|\Za|^w < 16 \cdot 2^s(r_tg_t)^w.
\end{align*}
For any 
$y\in\mathcal{P}^t_{m,w,s}$ we use~\eqref{uy1} 
and Lemma~\ref{L:comp}  
with $n+m$ being the length of $y$, $i=m$, $k=n$, 
$a_0,\dots,a_{n-1}$ being the values of $\xi$ along $y$ except when it visits $\Omega_t$, and $a_n,\dots,a_{n+m}=\xi(\Za)$
and obtain 
\begin{align*}
U(t,y)\le e^{t\xi(\Za)-2t}\frac{t^m}{m!}
\prod_{\heap{j=0}{y_j\notin\Omega_t}}^{\ell(y)}\frac{1}{\xi(\Za)-\xi(y_j)},
\end{align*}
on the event $\mathcal{E}_t$. We will keep $|\Za|$ terms in the product corresponding 
to one visit to each of the points $i\iota$, $0\le i\le |\Za|-1$, where $\iota=\mathrm{sgn}(z_t(y))$, 
and estimate the rest by 
\begin{align*}
\xi(\Za)-\xi(y_j)\ge 
\xi_{R_t}^{\ssup 1}-\xi_{R_t}^{\ssup 2}
> a_tf_t.
\end{align*}
This implies 
\begin{align*}
U(t,y) < \Big\{e^{t\xi(\Za)-2t}
\prod_{j=0}^{|\Za|-1}\frac{1}{\xi(\Za)-\xi(i\iota)}\Big\}
\frac{t^m}{m!}
(a_tf_t)^{-2w-s}.
\end{align*}
By Proposition~\ref{shortest} we obtain on $\mathcal{E}_t$
\begin{align}
\label{bbb}
U(t,y) < \big[U(t,y^{\ssup{t,\iota}})+o(1)U(t)\big]
\frac{t^m}{m!}
(a_tf_t)^{-2w-s}.
\end{align}
Let us show that the total mass corresponding to all paths from $\mathcal{P}^{t}_{m,w,s}$
except those corresponding to 
$(m,w,s)=(0,0,0)$ is negligible. 
Indeed, 
\begin{align*}
\sum_{w=0}^{\infty}\sum_{m=0}^{\infty}\sum_{s=m}^{\infty}
&|\mathcal{P}^t_{m,w,s}|\frac{t^m}{m!}
(a_tf_t)^{-2w-s}\one\{(m,k,s)\neq (0,0,0)\}\\
& <16\sum_{w=0}^{\infty}\sum_{m=0}^{\infty}\sum_{s=m}^{\infty}
2^s(r_tg_t)^w\frac{t^m}{m!}
(a_tf_t)^{-2w-s}\one\{(m,w,s)\neq (0,0,0)\}\\
&=16\Big[\Big(\sum_{w=0}^{\infty}\Big(\frac{r_tg_t}{a_t^2f_t^2}\Big)^{w}\Big)\Big(\sum_{m=0}^{\infty}\frac{t^m}{m!}
\Big(\sum_{s=m}^{\infty}\Big(\frac{2}{a_tf_t}\Big)^s\Big)\Big)-1\Big]\\
&=16\Big[
\Big(1-\frac{r_tg_t}{a_t^2f_t^2}\Big)^{-1}
\Big(1-\frac{2}{a_tf_t}\Big)^{-1}
\sum_{m=0}^{\infty}\frac{1}{m!}\big(\frac{2t}{a_tf_t}\Big)^m
-1\Big]\\
&=16\Big[\Big(1-\frac{r_tg_t}{a_t^2f_t^2}\Big)
\Big(1-\frac{2}{a_tf_t}\Big)^{-1}
\exp\Big\{\frac{2t}{a_tf_t}\Big\}
-1\Big]=o(1)
\end{align*}
since $\frac{r_tg_t}{a_t^2f_t^2}=o(1)$, $\frac{2}{a_tf_t}=o(1)$,
and $\frac{2t}{a_tf_t}=o(1)$ as $\alpha \in (1,2)$.  
Combining this with~\eqref{bbb} we obtain on the event $\mathcal{E}_t$
\begin{align*}
U_0(t) < \sum_{\iota\in\{-1,1\}}\big[U(t,y^{\ssup{t,\iota}})+o(1)U(t)\big](1+o(1)),
\end{align*}
which gives the required result by Proposition \ref{u12}.
\end{proof}

\subsection{The case $\alpha \ge 2$: Paths to $\Omega_t$ visiting sites in $\mathcal{N}_t$ at most once}
\label{s:paths2}

Our proof proceeds in two stages. First, we analyse the portion of the part up until the first visit to $\Omega_t$ and after the last visit to $\Omega_t$, and show that, in this portion of the path, it is never beneficial to visit sites in $\mathcal{N}_t \cup \{0\}$ more than once. Second, we analyse the portion of the path consisting of the loops that occur between first and last visit to $\Omega_t$, showing that it is never beneficial for these loops to return to sites in $\mathcal{N}_t \cup \{0\}$; in fact, we show the stronger result that these loops have length at most $\lfloor 2\alpha \rfloor$ (although we suspect that the optimal bound is actually $\lfloor \alpha \rfloor$).

\smallskip

Denote by $\mathcal{P}^t$ the set of all geometric paths contributing to $U_0(t)$, that is, those visiting $\Omega_t$ and having length at most $R_t$.  Fix $t>0$ and let $y\in \mathcal{P}^t$. The \emph{skeleton}
of $y$, denoted $\text{skel}(y)$, is the geometric path from the origin to a site in $\Omega_t$ constructed by chronologically removing all loops in $y$ which start and end at any site belonging to $\{0\}\cup \mathcal{N}_t$ up until the first visit of $\Omega_t$ as well as removing any part of the path after the final visit of $y$ to $\Omega_t$. 

\smallskip

We can now partition $\mathcal{P}^t$ into equivalence classes by saying that paths $y$ and $\hat y$ are in the same class if and only if $\text{skel}(y)=\text{skel}(\hat y)$. We write $\mathfrak{P}^t$ for the set of all such equivalence classes.
Note that any such equivalence class 
$\mathcal{P}\in\mathfrak{P}^t$ contains the \emph{null path}, $y_{\mathrm{null}}^\mathcal{P}\in\mathcal{P}$, defined as $y_{\mathrm{null}}^\mathcal{P}=\mathrm{skel}(y_{\mathrm{null}}^\mathcal{P})$. Observe that every null path, prior to visiting $\Omega_t$ for the first time, either (i) visits each site in $\{0\} \cup (\mathcal{N}_t \cap \mathbb{N})$ exactly once, or (ii) visits each site in $\{0\} \cup (\mathcal{N}_t \cap -\mathbb{N})$ exactly once. In particular, until the first visit of $\Omega_t$ each null path visits either only positive integers, or only negative integers. 

\smallskip

The importance of the null path is through the following lemma, which states that the contribution to the solution coming from an equivalence class is dominated by that coming from the null path.

\begin{lemma}
\label{L:nullpaths}
Almost surely,
\[
\sum_{y\in\mathcal{P}}U(t,y)< (1+o(1))U(t,y_{\mathrm{null}}^\mathcal{P})
\]
uniformly for all $\mathcal{P}\in\mathfrak{P}^t$ 
on the event $\mathcal{E}_t\cap \mathcal{E}_t^{[2,\infty)}$, as $t\to\infty$.
\end{lemma}  
 
\begin{proof} 
For $k\in\mathbb{N}$, write $\mathcal{P}^k$ for the subset of $\mathcal{P}$ consisting of the paths with additional length $k$ compared to $y_{\mathrm{null}}^\mathcal{P}$. 
We have on $\mathcal{E}_t^{[2,\infty)}$
$$|\mathcal{P}^k|\le \big(2(|\mathcal{N}_t|+2)\big)^{k}
< \delta_t^{-2\alpha k}$$
since each of the additional $k$ pieces will be added to a loop
at a site in $\{0\}\cup\mathcal{N}_t$ or at the end in at most two ways. 
Note that no null path can visit both sites in $\Omega_t$ since each null path is in $\mathcal{P}_t$ and has length at most $R_t<2|\Za|$. 
Using~\eqref{uy1} and 
Lemma \ref{L:comp} 
with $m+1$ being the number of visits of $y$ to $\Omega_t$, $n+m$ the length of $y$, $i=0$,
$a_0,\dots, a_{k-1}$ the values of $\xi$ at the additional points of $y$, $a_k,\dots,a_{n-1}$ the values of $\xi$ along $y_{\mathrm{null}}^\mathcal{P}$ except when it visits $\Omega_t$, and $a_n=\dots=a_{n+m}$ the value of $\xi$ on $\Omega_t$, we obtain 
\begin{align*}
U(t,y)\le U(t,y_{\mathrm{null}}^\mathcal{P})\prod_{j=1}^k
\frac{1}{a_n-a_{j-1}}.
\end{align*}
on $\mathcal{E}_t$. Since none of the additional sites visited by any path in $\mathcal{P}$ are in $\Omega_t$, we have on 
$\mathcal{E}_t$ 
\[
U(t,y) < U(t,y_{\mathrm{null}}^\mathcal{P})(a_tf_t)^{-k},
\]
and thus 
\begin{align*}
\sum_{y\in\mathcal{P}}U(t, y)&=\sum_{k=0}^\infty\sum_{ y\in\mathcal{P}^k}U(t,y) < U(t,y_{\mathrm{null}}^\mathcal{P})\sum_{k=0}^\infty \big(a_tf_t\delta_t^{2\alpha}\big)^{-k}=(1+o(1))U(t,y_{\mathrm{null}}^\mathcal{P})
\end{align*}
on $\mathcal{E}_t\cap \mathcal{E}_t^{[2,\infty)}$ as
$a_tf_t\delta_t^{2\alpha}\to \infty$.
\end{proof} 

We now eliminate paths that make loops from $\Omega_t$ that return to sites in $\mathcal{N}_t$. Denote by $\mathrm{Null}^t_1$ the set of all null paths in $\mathcal{P}^t$  which visit each site in $\{0\}\cup\mathcal{N}_t$ at most once,  $\mathrm{Null}^t_2$ for all other null paths in $\mathcal{P}^t$ and $\mathrm{Null}^t$ for their union.

\begin{lemma}
\label{L:null1}
Almost surely,
\[
\sum_{y\in \mathrm{Null}^t_2}U(t,y)=o(1)\sum_{y\in \mathrm{Null}^t_1}U(t,y)
\]
on the event $\mathcal{E}_t\cap\mathcal{E}^{[2,\infty)}_t$, 
as $t\to\infty$.
\end{lemma}
\begin{proof}
Note that by the construction of null paths, the only way for a null path to visit a site in $\mathcal{N}_t$ more than once is by having a loop from $\Omega_t$. On the event $\mathcal{E}^{[2,\infty)}_t$ this loop must have length at least $g_t$. We shall show a stronger result than is needed: that all null paths with loops from $\Omega_t$ of length more than $k_0$, where $k_0>2\alpha$, have negligible contribution to the solution compared to the contribution from all other null paths. 
\smallskip

To do this we partition $\mathrm{Null}^t$ into equivalence classes by saying two null paths 
are in the same class if and only if they are identical after removing all loops from $\Omega_t$ of length at least  $k_0$. For any such equivalence class $\mathcal{P}$, write $y^{\mathcal{P}}_\mathrm{min}$ for the path in $\mathcal{P}$ of minimum length (i.e.\ the path without any loops from $\Omega_t$ of length at least  $k_0$). Further, for any 
$k\ge k_0$, write $\mathcal{P}^k$ for the set of paths in 
$\mathcal{P}$ with additional length $k$ compared to $y^{\mathcal{P}}_\mathrm{min}$. Finally we write $\mathfrak{N}^t$ for the set of all such equivalence classes. 

\smallskip

Observe that for all $k\ge k_0$ and $\mathcal{P}\in \mathfrak{N}^t$, any path $y\in \mathcal{P}^k$ can make no more than $\lfloor k/k_0\rfloor$ extra visits to $\Omega_t$ compared to $y^{\mathcal{P}}_\mathrm{min}$.
Using~\eqref{uy1} and 
Lemma \ref{L:comp} 
with $m+1$ being the number of visits of $y$ to $\Omega_t$, $n+m$ the length of $y$, $i$ the number of additional visits to $\Omega_t$ compared to $y^{\mathcal{P}}_\mathrm{min}$,  
$a_0,\dots, a_{k-1-i}$ the values of $\xi$ at the additional points of $y$ except when it visits $\Omega_t$, $a_{k-i}=\cdots=a_{k-1}$ the value of $\xi$ on $\Omega_t$, $a_k,\dots,a_{n-1}$ the values of $\xi$ along $y_{\mathrm{min}}^\mathcal{P}$ except when it visits $\Omega_t$, and $a_n=\dots=a_{n+m}$ the value of $\xi$ on $\Omega_t$, we obtain 
\[
U(t,y)\le U(t,y^{\mathcal{P}}_\mathrm{min}) 
\frac{t^i}{i!}\prod_{j=0}^{k-1-i}\frac{1}{a_n-a_j}
< U(t,y^{\mathcal{P}}_\mathrm{min}) t^{ k/k_0}(a_tf_t)^{-k+ k/k_0}
\]
on $\mathcal{E}_t$.
Further, on $\mathcal{E}_t$
\begin{align*}
|\mathcal{P}^k|\le 2^k\big[(R_t-|\Za|+1)/2\big]^{\lfloor k/k_0\rfloor}< 2^k(r_tg_tf_t)^{k/k_0} 
\end{align*}
since there are at most $\lfloor k/k_0\rfloor$ additional loops, at most $(R_t-|\Za|+1)/2$ points where such a loop can be created, and at most $2^k$ shapes of the loops. 
\smallskip

Hence, for any $\mathcal{P}\in\mathfrak{N}^t$, on the event $\mathcal{E}_t$
\begin{align*}
\sum_{y\in \mathcal{P}}U(t,y)
&\le U(t,y^{\mathcal{P}}_\mathrm{min})
+\sum _{k=k_0}^\infty\sum_{ y\in \mathcal{P}^k}U(t,y)\\
&< U(t,y^{\mathcal{P}}_\mathrm{min})\Big(1+
\sum_{k=k_0}^\infty 2^k (t r_tg_tf_t)^{k/k_0}(a_tf_t)^{k/k_0-k}\Big)\\&
=U(t,y^{\mathcal{P}}_\mathrm{min})\Big(1
+ 2^{k_0}tr_tg_tf_t (a_t f_t)^{1-k_0}\Big[1-2(tr_tg_tf_t)^{1/k_0}(a_tf_t)^{1/k_0-1}\Big]^{-1}\Big).
\end{align*}
Since $k_0>2\alpha$ this implies 
\[
\sum_{y\in \mathcal{P}}U(t,y) < U(t,y^{\mathcal{P}}_\mathrm{min})(1+o(1))
\]
as $t\to\infty$ uniformly over the equivalence classes. 
To conclude the proof, note that 
\begin{align*}
\sum_{y\in \mathrm{Null}^t}U(t,y)=\sum_{\mathcal{P}\in\mathfrak{N}^t}\sum_{y\in\mathcal{P}}U(t,y) <  (1+o(1))\sum_{\mathcal{P}\in\mathfrak{N}^t}U(t,y^{\mathcal{P}}_\mathrm{min}) < (1+o(1)) \sum_{y\in \mathrm{Null}^t_1}U(t,y)
\end{align*}
on the event $\mathcal{E}_t\cap\mathcal{E}^{[2,\infty)}_t$. 
\end{proof}

\begin{prop}
\label{prop:U0null}
Almost surely,
\[ U_0(t)=(1+o(1))\sum_{y\in \mathrm{Null}^t_1}U(t,y) \]
on the event $\mathcal{E}_t\cap\mathcal{E}^{[2,\infty)}_t$, 
as $t\to\infty$.
\end{prop}

\begin{proof} This is a direct consequence of Lemmas~\ref{L:nullpaths} and~\ref{L:null1}. Indeed, 
\begin{align*}
U_0(t)
&=\sum_{y\in\mathcal{P}^t}U(t,y)
=\sum_{\mathcal{P}\in \mathfrak{P}^t}\sum_{y\in\mathcal{P}}
U(t,y)<(1+o(1))\sum_{\mathcal{P}\in \mathfrak{P}^t}U(t,y^{\mathcal{P}}_\mathrm{null})\\
&=(1+o(1))\sum_{y\in \mathrm{Null}^t}U(t,y)
=(1+o(1))\sum_{y\in \mathrm{Null}^t_1}U(t,y)
\end{align*} 
on the event $\mathcal{E}_t\cap\mathcal{E}^{[2,\infty)}_t$, as $t\to\infty$.
\end{proof}

\subsection{Completion of the proof of Theorem~\ref{main0}}
We are now in a position to prove the localisation statement in Theorem \ref{main0} on the event that $\mathcal{E}_t$ holds; the fact that $\P(\mathcal{E}_t) \to 1$ as $t \to \infty$ will be proven in Proposition \ref{e0}. The second statement of Theorem \ref{main0}, that $\P(\mathfrak{D}_t) \to p/(2-p)$, will be proven in Proposition~\ref{prob}.

\smallskip

By Proposition \ref{e2} we may work on the event $\mathcal{E}_t\cap\mathcal{E}^{[2,\infty)}_t$. Since $U_1$ is negligible with respect to~$U$ by Proposition~\ref{u12}, it remains to show that the contribution to $U_0$ from the paths not ending in~$\Omega_t$ is negligible. For $\alpha \in (1,2)$ this follows from Propositions~\ref{u3}; for $\alpha\ge 2$ this follows from Propositions~\ref{prop:U0null}. In fact, the latter argument works for all $\alpha>1$ but we prefer to use the much simpler argument for $\alpha \in (1,2)$.

\bigskip


\section{Point process analysis}
\label{s:pp}

In this section we develop a point processes approach to analyse the high exceedences of $\xi$ and top order statistics of the penalisation functional $\Psi_t$. We use this analysis to prove that the $\mathcal{E}_t$ holds eventually with overwhelming probability. We also use it to give an explicit construction for the limiting random variable $\Upsilon$ from Theorem~\ref{main1}. Since the proofs in this section are quite technical, we defer some of them to Appendix \ref{A:B}.

\smallskip

Recall that $E = \Z \setminus D$ denotes the set of positive integers whose potential values are exclusive, and abbreviate $q = 1-p$.

\subsection{Point process convergence for the rescaled potential}
The first step is to establish that the potential, properly rescaled, converges to a Poisson point process. The limiting point process will arise as a superposition of two distinct independent Poisson point processes that are, respectively, the limit of the potential restricted to the duplicated and the exclusive sites.

\smallskip

Let us begin by defining the limiting point process. Consider the measure
\begin{align*}
\mu(dx\otimes dy)=dx\otimes \frac{\alpha}{|y|^{\alpha+1}}dy
\end{align*}
on $\R^2$. 
In the sequel, we denote by the same symbol the restriction of $\mu$
to subsets of $\R^2$, and we denote by $(0,\infty]$ the extension of
$(0,\infty)$ by the point $\infty$, equipped with the topology generated
by the topology of $(0,\infty)$ and the sets of the form $(a,\infty]$,
for all $a\in \R$.
Let $\Pi^{\ssup e}$ be a Poisson point process on $\R\times (0,\infty]$ with the intensity measures
$q\mu$. Let $\Pi^{\ssup{d,+}}$ be a Poisson point process on 
$[0,\infty)\times (0,\infty]$ with the intensity measures 
$p\mu$ and independent of $\Pi^{\ssup e}$. Let $\Pi^{\ssup{d,-}}$ be a Poisson point process on 
$(-\infty,0]\times (0,\infty]$ defined by 
$\Pi^{\ssup{d,-}}(A)=\Pi^{\ssup{d,+}}(\hat A)$ for any Borel set $A\subseteq (-\infty,0]\times (0,\infty]$, where 
$\hat A$ is the reflection of the set $A$ with respect to the $y$-axis. Finally, let $\Pi^{\ssup{d}}$ be the point process on $\R\times (0,\infty]$ defined by 
\begin{align*}
\Pi^{\ssup{d}}(A)
=\Pi^{\ssup{d,-}}(A\cap (-\infty,0]\times (0,\infty])
+\Pi^{\ssup{d,+}}(A\cap [0,\infty)\times (0,\infty]),
\end{align*} 
and let 
\begin{align*}
\Pi=\Pi^{\ssup d}+\Pi^{\ssup e}
\end{align*}
be a point process on $\R\times (0,\infty]$. Denote the corresponding probability and expectation by 
$\Prob_{*}$ and $\mathrm{E}_*$. 

\smallskip

We show the convergence of the potential, properly rescaled, to the Poisson point process $\Pi$. Let 
\begin{align*}
\Pi^{\ssup e}_s=\sum_{z\in E}\e\Big(\frac{z}{s}, \frac{\xi(z)}{s^{1/\alpha}}\Big), \quad
\Pi^{\ssup{d,+}}_s=\sum_{z\in D, z\ge 0}\e\Big(\frac{z}{s}, \frac{\xi(z)}{s^{1/\alpha}}\Big) \quad \text{and} \quad
\Pi^{\ssup{d,-}}_s=\sum_{z\in D, z\le 0}\e\Big(\frac{z}{s}, \frac{\xi(z)}{s^{1/\alpha}}\Big), 
\end{align*}
where $\e(x,y)$ denotes the Dirac measure in $(x,y)$. Denote 
\begin{align*}
\Pi^{\ssup d}_s=\Pi^{\ssup{d,+}}_s+\Pi^{\ssup{d,-}}_s
\qquad\text{and}\qquad
\Pi_s=\Pi^{\ssup d}_s+\Pi^{\ssup e}_s.
\end{align*}

The following convergence result is classical and we defer its proof to Appendix \ref{A:B}.
\begin{lemma}
\label{lppp}
As $s\to\infty$, $(\Pi^{\ssup {d,+}}_s, \Pi^{\ssup {d,-}}_s, \Pi^{\ssup e}_s)$ converges in law to $(\Pi^{\ssup {d, +}}, \Pi^{\ssup {d, -}},\Pi^{\ssup e})$, and in particular, $\Pi_s$ converges in law to $\Pi$.
\end{lemma}

\subsection{Asymptotic properties of the top order statistics of the penalisation functional}
We now show how to use the convergence of the potential to extract asymptotic properties of the top order statistics of the penalisation functional $\Psi_t$. We first introduce the limiting versions of $\Za$, $\Zb$ and $\mathfrak{D}_t$ and study their properties, before arguing that we may successfully pass to the limit. 

\smallskip

Given a point measure $\Sigma$, we say that $x\in\Sigma$
if $\Sigma(\{x\})>0$.  Let $\bar \Pi$ be the point process on $[0,\infty)\times (0,\infty]$ defined by 
\begin{align*}
\bar \Pi(A)=\Pi^{\ssup e}(A)+\Pi^{\ssup e}(\hat A)+\Pi^{\ssup {d,+}}(A)
\end{align*}
for any Borel set $A$, where $\hat A$ denotes the reflection of $A$ with respect to the $y$-axis. Remark that the three components of $\bar \Pi$ are independent Poisson point processes with the intensity measures $q\mu$, $q\mu$ and $p\mu$, respectively, and so $\bar \Pi$ is itself a Poisson point process with intensity measure $(2q + p) \mu = (2-p)\mu$. Abbreviate 
\begin{align}
\label{rho}
\rho=\frac{1}{\alpha-1},
\end{align}
and let the positive random variables $X^{\ssup 1}, X^{\ssup 2}, Y^{\ssup 1}$ and $Y^{\ssup 1}$ be defined by the properties that 
\begin{align*}
(X^{\ssup 1}, Y^{\ssup 1}) &\in \bar \Pi, \text{ and if } (x,y) \in \bar \Pi \text{ then }y-\rho x\le Y^{\ssup 1} -\rho X^{\ssup 1}, \\
(X^{\ssup 2}, Y^{\ssup 2}) &\in \bar \Pi,
\text{ and if } (x,y) \in \bar \Pi  \setminus(X^{\ssup 1}, Y^{\ssup 1}) \text{ then }y-\rho|x|\le Y^{\ssup 2}-\rho X^{\ssup 2} .
\end{align*}
In Lemma \ref{xy} we show that these are well-defined. Denote $\mathfrak{D} = \{ (X^{\ssup 1}, Y^{\ssup 1})  \in  \Pi^{\ssup {d,+}}  \}$.

 \smallskip

At the end of this section we shall identify $(X^{\ssup i}, Y^{\ssup i}), i = 1,2$, and $\mathfrak{D}$ as the limiting versions of $(|\Zc|, \xi(\Zc)), i = 1,2$, and $\mathfrak{D}_t$ respectively. For now, we establish some properties of these objects.

\begin{lemma}
\label{xy}
Almost surely, the random variables $X^{\ssup 1}, X^{\ssup 2}, Y^{\ssup 1}$ and $Y^{\ssup 2}$ are well-defined and satisfy $Y^{\ssup 1}-\rho X^{\ssup 1} >Y^{\ssup 2}-\rho X^{\ssup 2} >0$. 
\end{lemma}

\begin{proof}
For any $a>0$ compute 
\begin{align}
\label{mumu}
\mu\big(\{(x,y): y> a+\rho x\}\big)=2\int_0^{\infty}\int_{a+\rho x}^{\infty}\frac{\alpha}{y^{\alpha+1}}dydx=
2 a^{1-\alpha}. 
\end{align}
Since this is finite, almost surely there are finitely many points of $(x,y)\in \bar \Pi$
satisfying $y-\rho x> a$. On the other hand, since~\eqref{mumu} tends to infinity as $a\downarrow 0$, almost surely 
there exist points $(x,y)\in \bar \Pi$ satisfying $y-\rho x> 0$. This implies the result.
\end{proof}

\begin{lemma}
\label{xy2}
The random variable $(X^{\ssup 1},Y^{\ssup 1})$ has density
\begin{align}
\label{density}
p(x,y)=(2-p)\alpha{y^{-\alpha-1}}\exp\{-(2-p) (y-\rho x)^{1-\alpha}\}\one\{y-\rho x>0\}.
\end{align}
\end{lemma}
\begin{proof}
We have
\begin{align*}
\Prob_*&\big(X^{\ssup 1}\in dx,Y^{\ssup 1}\in dy\big)\\
&=\Prob_*\Big(\bar \Pi(dx\times dy)=1,
\bar \Pi\big(\{(u,v): v-\rho u>y-\rho x\}\big)=0\Big)\notag\\
&=\Prob_*\big(\bar \Pi(dx\times dy)=1\big)\times \Prob_*\Big(\bar \Pi\big(\{(u,v): v-\rho u>y-\rho x \}\big)=0 \Big) \notag \\
&= (2-p) \exp\Big\{-(2-p)\mu\big(\{(u,v): u\ge 0,  y-\rho x<v-\rho u\}\big)\Big\}\mu(dx,dy).
\end{align*}
To complete the result, compute
\begin{align*}
\mu\big(\{(u,v): u\ge 0, y-\rho x<v-\rho u\}\big)
&=\int_0^{\infty}\int_{y-\rho x+\rho u}^{\infty}\frac{\alpha}{v^{\alpha+1}}dvdu\notag\\
&=\int_0^{\infty}(y-\rho x+\rho u)^{-\alpha}du
=(y-\rho x)^{1-\alpha} . \qedhere
\end{align*}
\end{proof}

\begin{lemma} 
\label{probb}
$\Prob_*(\mathfrak{D})= p / (2-p)$.
\end{lemma}

\begin{proof} 
Since the components $\Pi^{\ssup e}$ and $\Pi^{\ssup {d,+}}$ appearing in the definition of $\bar \Pi$ are independent Poisson point processes with the intensity measures $q\mu$ and $p\mu$ respectively, we have 
\begin{equation*}
Prob_* \left( (X^{\ssup 1}, Y^{\ssup 1}) \in \Pi^{\ssup{d,+}} \right)=\frac{p}{p+2q}=\frac{p}{2-p} . \qedhere
\end{equation*}
\end{proof}

We now argue that we can successfully pass to the limit. As a consequence, we prove that the event $\mathcal{E}_t$ holds eventually with overwhelming probability. Since the proof of these results are rather technical, we defer them to Appendix \ref{A:B}.

\begin{prop}
\label{sl1}\label{lpsi}
As $t\to\infty$, \begin{enumerate}[(i)]
\item $ \Big(\frac{|\Za|}{r_t},\frac{|\Zb|}{r_t},\frac{\xi(\Za)}{a_t},\frac{\xi(\Zb)}{a_t}\Big)
\Rightarrow (X^{\ssup 1},X^{\ssup 2},Y^{\ssup 1},Y^{\ssup 2}),$
\item  $\Big(\frac{\Psi_t(\Za)}{a_t},
\frac{\Psi_t(\Zb)}{a_t}\Big)
\Rightarrow (Y^{\ssup 1}-\rho X^{\ssup 1},Y^{\ssup 2}-\rho X^{\ssup 2}).$
\end{enumerate}
\end{prop}

\begin{prop}
\label{e0}
$\text{\rm Prob}(\mathcal{E}_t)\to 1$ as $t\to\infty$.
\end{prop}

\begin{prop} 
\label{prob}
$\Prob(\mathfrak{D}_t)\to \frac{p}{2-p}$
as $t\to\infty$. 
\end{prop} 

\subsection{An explicit construction of the limiting random variable}
\label{s:up}
We complete this section by giving an explicit construction of the limiting random variable $\Upsilon$ in Theorem \ref{main1}. For any $\delta>0$, let 
\begin{align*}
S^{\ssup{\delta,+}}
&=\!\!\sum_{\heap{(x,y)\in \Pi^{\ssup e}}{0<x<X^{\ssup 1} ,y\ge \delta Y^{\ssup 1}}}\!\!\log \Big(1-\frac{y}{Y^{\ssup 1}}\Big)
\qquad\text{and}\qquad
S^{\ssup{\delta,-}}
=\!\!\sum_{\heap{(x,y)\in \Pi^{\ssup e}}{-X^{\ssup 1}<x<0,y\ge \delta Y^{\ssup 1}}}\!\!\log \Big(1-\frac{y}{Y^{\ssup 1}}\Big),
\end{align*}
on the event $\mathfrak{D}$, and zero otherwise, and let
\begin{align}
\label{sdel}
S^{\ssup \delta}=- X^{\ssup 1} \big(S^{\ssup{\delta,+}}-S^{\ssup{\delta,-}}\big) .
\end{align}
Observe that these variables are well-defined since, for every $(x,y)\in\Pi^{\ssup e}$
such that $|x|<X^{\ssup 1}$, we have $y<Y^{\ssup 1}-\rho X^{\ssup 1}+\rho|x|<Y^{\ssup 1}$. 

\smallskip

In the next lemma we show that, as $\delta \to 0$,
\[ \mathcal{L} (  S^{\ssup \delta}|\mathfrak{D} )  \Rightarrow \mathcal{L}(\log \Upsilon) \]
 for a certain random variable $\Upsilon$ with positive density on $\R_+$. In Section \ref{s:main1} we identify $\Upsilon$ with the random variable appearing in Theorem \ref{main1}.

\begin{lemma} 
\label{s}
As $\delta\downarrow 0$,
\begin{align*}
\mathcal{L}(S^{\ssup \delta}|\mathfrak{D})\Rightarrow \mathcal{L}( \log \Upsilon ),
\end{align*}
where $\Upsilon$ is a random variable with positive density on $\R_+$ defined as follows. Let $(X, Y) \in \R^2$ be a random variable with density given by \eqref{density}. Then, conditionally on $(X, Y)$, $\log \Upsilon$ is the value at time $X$ of a time-inhomogeneous L\' evy process with zero drift, no Brownian component and the L\' evy measure 
\begin{align}
\label{levy}
L(dx\otimes dz)=
\left\{
\begin{array}{ll}
\displaystyle \frac{q\alpha e^{|z|}}{(Y-\rho X+\rho x)^{\alpha}
Y^{\alpha}(1-e^{|z|})^{\alpha+1}}
dx\otimes dz & \text{ if }
|z|<\log\frac{Y}{\rho( X-x)},\\
0 & \text{ otherwise.}
\end{array}
\right.
\end{align}
\end{lemma}

\begin{proof} 
Denote for brevity $X=X^{\ssup 1}$ and $Y=Y^{\ssup 1}$. Conditionally on $\Pi^{\ssup d}$, $\mathfrak{D}$ and $(X, Y)$, the point process $\Pi^{\ssup e}$ is Poissonian with the intensity measure 
\begin{align*}
\mu^{\ssup e}(dx\otimes dy)=
\left\{
\begin{array}{ll}
q(Y-\rho X+\rho |x|)^{-\alpha}\mu (dx\otimes dy) & \text{ if }
y-\rho |x|<Y-\rho X,\\
0 & \text{ otherwise, }
\end{array}
\right.
\end{align*}
and $S^{\ssup{\delta,+}}$ is the value at time $X$ of a time-inhomogeneous L\' evy process with zero drift, no Brownian 
component and the L\' evy measure 
\begin{align*}
L^{\ssup{\delta,+}}(dx\otimes dz)=
\left\{
\begin{array}{ll}
\displaystyle \frac{q\alpha e^z}{(Y-\rho X+\rho x)^{\alpha}
Y^{\alpha}(1-e^z)^{\alpha+1}}
dx\otimes dz & \text{ if }
\log\frac{\rho(X-x)}{Y}<z\le \log(1-\delta),\\
0 & \text{ otherwise,}
\end{array}
\right.
\end{align*}
where we consider $\delta<Y-\rho X$. Further, conditionally on $\Pi^{\ssup d}$, $\mathfrak{D}$ and $(X, Y)$, the variable $S^{\ssup{\delta,-}}$ is independent and identically distributed with $S^{\ssup{\delta,+}}$.  Due to symmetry, $S^{\ssup \delta}$  is therefore the value of a time $X$ of a time-inhomogeneous L\' evy process with zero drift, no Brownian 
component and the L\' evy measure 
\begin{align*}
L^{\ssup{\delta}}(dx\otimes dz)=
\left\{
\begin{array}{ll}
\displaystyle \frac{q\alpha e^{|z|}}{(Y-\rho X+\rho x)^{\alpha}
Y^{\alpha}(1-e^{|z|})^{\alpha+1}}
dx\otimes dz & \text{ if }
\log\frac{1}{1-\delta}\le |z|< \log\frac{Y}{\rho(X-x)},\\
0 & \text{ otherwise. }
\end{array}
\right.
\end{align*}
As $\delta\downarrow 0$, $S^{\ssup{\delta,-}}$ converges weakly to  
the value at time $X$ of a time-inhomogeneous L\' evy process with zero drift, no Brownian 
component and the L\' evy measure given by \eqref{levy}, where the limiting L\' evy measure is valid because
\begin{align*}
\int_0^{X}&\int_{\R}\min\{1,z^2\}L(dx\otimes dz)
=2\int_0^{X}\int_0^{\infty}\min\{1,z^2\}L(dx\otimes dz)\\
&=2\int_0^{X}\int_0^{Y-\rho X+\rho x}\frac{q\alpha \min\{1,\log^2(1-y/Y)\}}{(Y-\rho X+\rho x)^{\alpha}
y^{\alpha+1}}dydx\\
&\le 2\int_0^{ X}\int_0^{Y}\frac{q\alpha \min\{1,\log^2(1-y/Y)\}}{(Y-\rho X+\rho x)^{\alpha}
y^{\alpha+1}}dydx\\
&=\frac{2q\alpha}{\alpha-1}\Big[(Y-\rho  X)^{1-\alpha}-Y^{1-\alpha}\Big]\int_0^{Y}\frac{ \min\{1,\log^2(1-y/Y)\}}
{y^{\alpha+1}}dy<\infty.
\end{align*}
Since a L\'evy process has positive density on $\R$ at positive times, and since the law of $\log \Upsilon$ is obtained by averaging over the law of L\'{e}vy processes at positive times, $\Upsilon$ also has positive density on $\R_+$.
\end{proof}

\bigskip


\section{Fluctuation theory in the case $\alpha \in (1,2)$}
\label{s:main1}

In this section we study the fluctuations in the ratio $u(t, \Za)/u(t, \Zb)$ in the case $\alpha \in (1, 2)$, building on our analysis in Section \ref{s:paths1}, and hence complete the proof of Theorem \ref{main1}.

\smallskip

Recall that $E$ denotes the set $\Z \setminus D$. For any $t>0$, let
\begin{align*}
S_t=-\sum_{\heap{0<|z|<|\Za|}{z \in E}}\text{sgn}(z/\Za)\log \Big(1-\frac{\xi(z)}{\xi(\Za)}\Big)
\end{align*}
on the event $\mathcal{E}_t$ and zero otherwise. In Section \ref{s:paths} we showed that the ratio $u(t, \Za)/u(t, -\Za)$ was well-approximated by $\exp\{S_t\}$, so it remains to study the convergence of $S_t$. To do this, we first truncate the sum at potential values above a certain threshold and show that this is a good approximation of the full sum; we then study the convergence of the truncated sums.  

\smallskip

For any $\delta>0$, define 
\[  S_{t}^{\ssup{\delta}} =-\sum_{\heap{0<|z|<|\Za|}{z\in E}}
\text{sgn}(z/\Za)\log \Big(1-\frac{\xi(z)}{\xi(\Za)}\Big)\one_{\{\xi(z) \ge \delta\xi(\Za)\}} \]
on the event $\mathcal{E}_t$ and zero otherwise, and let $\hat{S}_{t}^{\ssup{\delta}} = S_t -  S_{t}^{\ssup{\delta}}$. Denote by $\Prob^{\ssup e}$ and $\mathrm{E}^{\ssup e}$ the conditional probability and expectation given $D$ and $\{\xi(z): z\in D\}$. The next lemma shows that the truncated sum $S_{t}^{\ssup{\delta}}$ is a good approximation for the full sum $S_t$.

\begin{lemma} 
\label{epseps}
For any $\e_1,\e_2>0$ there is $\delta_0>0$ such that for each $0<\delta\le \delta_0$ 
\begin{align*}
\Prob\big(\{|\hat S_{t}^{\ssup{\delta}}|>\e_1\}\cap \mathfrak{D}_t\big)<\e_2
\end{align*}
eventually for all $t$.
\end{lemma}

\begin{proof} 
Let 
\begin{align*}
\mathcal{\hat E}_t=\Big\{c_1 <\frac{\xi(\Za)}{a_t}<c_2, c_1 <\frac{\Psi_t(\Za)}{a_t}<c_2, c_1<\frac{\Psi_t(\Za)}{\xi(\Za)}, \frac{|\Za|}{\xi(\Za)^{\alpha}}<c_2\Big\},
\end{align*}
where $c_1,c_2>0$  are chosen according to Proposition~\ref{sl1} so that
\begin{align*}
\Prob(\mathcal{\hat E}_t)>1-\e_2/3
\end{align*}
for all $t\ge t_1$. By Proposition~\ref{e0}, let $t_2$ be such that for all $t\ge t_2$
\begin{align*}
\Prob(\mathcal{E}_t)>1-\e_2/3. 
\end{align*}
For each $\delta>0$, let $t_3$ be such that $\delta c_1a_{t_3}>1$. Let $t_0= \max\{t_1,t_2,t_3\}$. 
Consider the event $\mathfrak{D}_t\cap \mathcal{E}_t\cap \mathcal{\hat E}_t$, $\delta>0$ and $t\ge t_0$. 

\smallskip

By Chebychev's inequality we have  
\begin{align*}
&\text{Prob}^{\ssup e}
\big(|\hat S_t^{\ssup{\delta}}|>\e_1\big)
\le \e_1^{-2}\mathrm{E}^{\ssup e}\big[\hat S_t^{\ssup{\delta}}\big]^2\\
&= \e_1^{-2}\mathrm{E}^{\ssup e}\Big[\!\!\!\!\!\!
\sum_{\heap{0<z<|\Za|}{z\in E}}\!\!\!\!\!\!
\text{sgn}(\Za)
\Big(\log \Big(1-\frac{\xi(-z)}{\xi(\Za)}\Big)\one_{\{\xi(-z)< \delta\xi(\Za)\}}-
\log \Big(1-\frac{\xi(z)}{\xi(\Za)}\Big)\one_{\{\xi(z)< \delta\xi(\Za)\}}\Big)\Big]^2
\end{align*}
Since the summands are independent and consist of the differences of two independent identically distributed terms under $\Prob^{\ssup e}$ we obtain 
\begin{align*}
\text{Prob}^{\ssup e}
&\big(|\hat S_t^{\ssup{\delta}}|>\e_1\big)
\le 4\e_1^{-2}\!\!\sum_{\heap{0<z<|\Za|}{z\in E}}\!\!
\mathrm{E}^{\ssup e}\Big[\log \Big(1-\frac{\xi(z)}{\xi(\Za)}\Big)\one_{\{\xi(z)< \delta\xi(\Za)\}}\Big]^2 .
\end{align*}
For each $0<z<|\Za|$, $z\in E$, the conditional distribution of $\xi(z)$ is the Pareto distribution with parameter $\alpha$
conditioned on $\Psi_t(z)<\Psi_t(\Za)$, that is, on 
\begin{align*}
\xi(z)-\frac{|z|}{t}\log \xi(z)<\Psi_t(\Za).
\end{align*} 
Observe that for all $\delta\le c_1$
\begin{align*}
\Big\{y\in [1,\infty):y-\frac{|z|}{t}\log y<\Psi_t(\Za)\Big\}
\supset [1,\Psi_t(\Za)]\supset [1,\delta\xi(\Za)].
\end{align*}
This implies 
\begin{align*}
\int_{1}^{\infty}\frac{\alpha}{y^{\alpha+1}}\one_{\{y-\frac{|z|}{t}\log y<\Psi_t(\Za)\}}dy
\ge \int_{1}^{\infty}\frac{\alpha}{y^{\alpha+1}}\one_{\{y<\Psi_t(\Za)\}}dy=1-\Psi_t(\Za)^{-\alpha}>1/2.
\end{align*}
Observe that $\delta\xi(\Za)>\delta c_1 a_t\ge 1$. Using 
the change of variables $y=u\xi(\Za)$, we obtain
\begin{align*}
\int_{1}^{\delta\xi(\Za)}
&\frac{\alpha}{y^{\alpha+1}}\log^2 \Big(1-\frac{y}{\xi(\Za)}\Big)
\one_{\{y-\frac{|z|}{t}\log y<\Psi_t(\Za)\}}dy\\
&=\int_{1}^{\delta\xi(\Za)}
\frac{\alpha}{y^{\alpha+1}}\log^2 \Big(1-\frac{y}{\xi(\Za)}\Big)dy\\
&\le\frac{\alpha}{\xi(\Za)^{\alpha}}\int_{0}^{\delta}
u^{-\alpha-1}\log^2 (1-u)du.
\end{align*}
Since 
\begin{align*}
\int_{0}^{\delta}
u^{-\alpha-1}\log^2 (1-u)du\sim \int_0^{\delta}u^{1-\alpha}du
=\frac{\delta^{2-\alpha}}{2-\alpha}
\end{align*}
as $\delta\downarrow 0$, 
we can choose $\delta_1\le c_1$ small enough so that for all $\delta\le \delta_1$ 
\begin{align*}
\mathrm{E}^{\ssup e}\Big[\log \Big(1-\frac{\xi(z)}{\xi(\Za)}\Big)\one_{\{\xi(z)< \delta\xi(\Za)\}}\Big]^2
\le \frac{4\alpha\delta^{2-\alpha}}{(2-\alpha)\xi(\Za)^{\alpha}}
\end{align*}
and 
\begin{align*}
\text{Prob}^{\ssup e}
&\big(|\hat S_t^{\ssup{\delta}}|>\e_1\big)
\le\frac{16\alpha\delta^{2-\alpha}}{\e_1^2(2-\alpha)}\cdot\frac{|\Za|}{\xi(\Za)^{\alpha}}
<\frac{16\alpha c_2\delta^{2-\alpha}}{\e_1^2(2-\alpha)}
<\e_2/3
\end{align*}
for all $\delta\le \delta_0$ with some $\delta_0\le \delta_1$. Hence
\begin{align*}
\Prob\big(\{|\hat S_{t}^{\ssup{\delta}}|>\e_1\big\}\cap \mathfrak{D}_t)
\le \Prob\big(\{|\hat S_{t}^{\ssup{\delta}}|>\e_1\big\}\cap \mathfrak{D}_t\cap\mathcal{E}_t\cap\mathcal{\hat E}_t)
+\Prob(\mathcal{E}_t^c)+\Prob(\mathcal{\hat E}_t^c)
<\e_2
\end{align*}
as required.
\end{proof}

We next show that the truncated sum $S_t^{\ssup \delta}$ converges to the variable $S^{\ssup \delta}$ introduced in \eqref{sdel}; since the proof is similar to those appearing in Appendix \ref{A:B}, we also defer it to the appendix.

\begin{prop}
\label{td}
As $t\to\infty$,
\begin{align*}
\mathcal{L}(S_t^{\ssup \delta}|\mathfrak{D}_t)
\Rightarrow \mathcal{L}(S^{\ssup{\delta}}|\mathfrak{D}).
\end{align*}
\end{prop}

We are now ready to put everything together to complete the proof of Theorem \ref{main1}, in particular showing that the ratio $u(t, \Za)/u(t, -\Za)$ converges in distribution to $\Upsilon$, where $\Upsilon$ is the random variable defined in Lemma \ref{s}.

\subsection{Completion of the proof of Theorem~\ref{main1}}

By Propositions~\ref{u12} and~\ref{u3} on the event $\mathfrak{D}_t$
only the shortest paths to $\Za$ and $-\Za$ are non-negligible and hence 
we have by Proposition~\ref{shortest} 
\begin{align*}
\frac{u(t,\Za)}{u(t,-\Za)}
&=\left\{\prod_{j=0}^{|\Za|-1}\frac{\xi(\Za)-\xi(-j)}{\xi(\Za)-\xi(j)}\right\}^\iota+o(1)\\
&=\exp\Big\{-\text{sgn}(\Za)\!\!\sum_{0\le j<|\Za|}\!\!
\Big[\log\Big(1-\frac{\xi(j)}{\xi(\Za)}\Big)-\log\Big(1-\frac{\xi(-j)}{\xi(\Za)}\Big)\Big]\Big\}+o(1)\\
&=\exp\big\{S_t\big\}+o(1).
\end{align*}
It suffices to show that 
\begin{align*}
\mathcal{L}(S_t|\mathfrak{D}_t)\Rightarrow \mathcal{L}(\log \Upsilon),
\end{align*}
where $\Upsilon$ is the random variable defined in Lemma \ref{s}. Let $x\in \R$, $\e>0$
and choose $\hat\e>0$ so that 
\begin{align}
\label{d1}
\Prob(\log \Upsilon \le x-\hat \e)
&> \Prob(\log \Upsilon \le x)-\e/4,\\
\Prob(\log \Upsilon \le x+\hat \e)
&< \Prob(\log \Upsilon \le x)+\e/4.
\label{d2}
\end{align}
Choose $\delta_0$ according to Lemma~\ref{epseps}
with $\e_1=\hat \e$ and $\e_2<\e\Prob(\mathfrak{D}_t)/4$ for all $t$, which is possible by Proposition~\ref{prob} since 
$\Prob(\mathfrak{D}_t)$ converges to a positive limit.
By Lemma~\ref{s}, choose 
$\delta\le \delta_0$ so that 
\begin{align}
\label{a1}
\Prob\big(S^{\ssup \delta}\le x-\e_1| \mathfrak{D}\big)
&> \Prob(\log \Upsilon \le x-\e_1)-\e/4,\\
\Prob\big(S^{\ssup \delta}\le x+\e_1| \mathfrak{D}\big)
&< \Prob(\log \Upsilon \le x+\e_1)+\e/4,
\label{a2}
\end{align}
and choose $t_0$ such that the statement of Lemma~\ref{epseps} holds for $t\ge t_0$.
We have 
\begin{align*}
\Prob(\{S_t\le x\}\cap \mathfrak{D}_t)
&=\Prob\big(\{S^{\ssup \delta}_t\le x-\hat S^{\ssup \delta}_t\}\cap \{|\hat S^{\ssup \delta}_t|\le\e_1\}\cap \mathfrak{D}_t\big)\\
&\quad +\Prob\big(\{S^{\ssup \delta}_t\le x-\hat S^{\ssup \delta}_t\}\cap \{|\hat S^{\ssup \delta}_t|>\e_1\}\cap \mathfrak{D}_t\big).
\end{align*}
Hence by Lemma~\ref{epseps} for all $t>t_0$,
\begin{align*}
\Prob\big(S^{\ssup \delta}_t\le x-\hat\e| \mathfrak{D}_t\big)-\e/4
< \Prob(S_t\le x|\mathfrak{D}_t)
< \Prob\big(S^{\ssup \delta}_t\le x+\hat\e| \mathfrak{D}_t\big)+\e/4.
\end{align*}
By Proposition~\ref{td} there is $t_1\ge t_0$ such that for all $t\ge t_1$
\begin{align*}
\Prob\big(S^{\ssup \delta}_t\le x-\hat\e| \mathfrak{D}_t\big)
&> \Prob\big(S^{\ssup \delta}\le x-\hat\e| \mathfrak{D}\big)
-\e/4,\\
\Prob\big(S^{\ssup \delta}_t\le x+\hat\e| \mathfrak{D}_t\big)
&< \Prob\big(S^{\ssup \delta}\le x+\hat\e| \mathfrak{D}\big)
+\e/4,
\end{align*}
implying 
\begin{align*}
\Prob\big(S^{\ssup \delta}\le x-\hat\e| \mathfrak{D}\big)-\e/2
< \Prob(S_t\le x|\mathfrak{D}_t)
< \Prob\big(S^{\ssup \delta}\le x+\hat\e| \mathfrak{D}\big)+\e/2.
\end{align*}
Combining this with~\eqref{a1} and~\eqref{a2} we obtain 
\begin{align*}
\Prob(S\le x-\hat\e)-3\e/4
< \Prob(S_t\le x|\mathfrak{D}_t)
< \Prob(\log \Upsilon \le x+\hat\e)+3\e/4.
\end{align*}
Together with~\eqref{d1} and~\eqref{d2} this gives 
\begin{align*}
\Prob(\log \Upsilon \le x)-\e
< \Prob(S_t\le x|\mathfrak{D}_t)
< \Prob(\log \Upsilon \le x)+\e.
\end{align*}
for all $t\ge t_1$. 

\bigskip


\section{Fluctuation theory in the case $\alpha\ge 2$}
\label{s:main2}

In this section we study the fluctuations in the ratio $u(t, \Za)/u(t, \Zb)$ in the case $\alpha \ge 2$, and hence complete the proof of Theorem \ref{main2}. Due to our analysis in Section \ref{s:paths}, we know that it is sufficient to study only the contribution from paths which visit the sites in $\mathcal{N}_t$ at most once.

\smallskip

The first step is to show that, by conditioning on the information not contained in the sites in $\mathcal{N}_t$, we are left with an expression that is amenable to applying standard fluctuation theory; here we use our analysis of the function $I$. The final step is to show that the fluctuations due to $\mathcal{N}_t$ are already enough to imply that $|\log u(t, \Za)/u(t, \Zb) | \to \infty$, regardless of the contributions from all other sites; we achieve this by invoking a central limit argument.

\smallskip

For each $t>0$, let $\mathcal{F}_t$ be the $\sigma$-algebra generated by $D$, $\Za$, $\mathcal{N}_t$, and $\{\xi(z):z\notin \mathcal{N}_t\}$. 
Let $\text{Prob}_{\mathcal{F}_t}$, ${\mathrm E}_{\mathcal{F}_t}$ and $\text{Var}_{\mathcal{F}_t}$ denote, respectively, conditional probability, expectation and variance with respect to~$\mathcal{F}_t$. For each $z\in \mathcal{N}_t$, define
\begin{align*}
Q_t(z)=-\text{sgn}(z/\Za)\log\Big(1-\frac{\xi(z)}{\xi(\Za)}\Big)
\end{align*}
whenever $\xi(z)<\xi(\Za)$ and zero otherwise. 
Let 
\begin{align*}
Q_t=\sum_{z\in\mathcal{N}_t}Q_t(z). 
\end{align*}
Observe that $Q_t(z)$, $z\in\mathcal{N}_t$, are conditionally independent with respect to $\mathcal{F}_t$, which implies that
\begin{align}
\label{varvar}
\text{Var}_{\mathcal{F}_t} Q_t 
=\sum_{z \in \mathcal{N}_t}\text{Var}_{\mathcal{F}_t} Q_t(z).
\end{align}
Further, it is easy to see that for each $z\in\mathcal{N}_t$, the conditional distribution 
of $\xi(z)$ is the Pareto distribution with parameter $\alpha$ conditioned on $\Psi_t(z)<\Psi_t(\Za)$
and $\xi(z)>\delta_t\xi(\Za)$. 

\smallskip

The next lemma establishes that, after conditioning on $\mathcal{F}_t$, the contribution to the ratio $u(t, \Za)/u(t, \Zb)$ due to the sites in $\mathcal{N}_t$ is well-approximated by a product over these sites.

\begin{prop}
\label{P:decomp}
There exists an $\mathcal{F}_t$-measurable random variable $P_t$ such that
\[ \big|  \log u ( t, \Za ) - \log u ( t, -\Za ) - Q_t  + P_t  \big| \one_{\mathfrak{D}_t\cap \mathcal{E}_t\cap \mathcal{E}_t^{[2,\infty)}}  \to 0 \]
almost surely, as $t\to\infty$.
\end{prop}

\begin{proof}
We write $\mathrm{Null}^t_{1+}$, $\mathrm{Null}^t_{1-}$ for the set of null paths in $\mathrm{Null}^t_1$ ending in $\Za$ and 
$-\Za$, respectively. Further, we denote by $\mathcal{N}_t^+$ and $\mathcal{N}_t^{-}$ the subsets of $\mathcal{N}_t$ consisting of the points lying between $0$ and $\Za$, and $0$ and $-\Za$, respectively. Finally, 
we denote by $N_t^{+}$ and $N_t^{-}$ the respective cardinalities of
$\mathcal{N}_t^+$ and $\mathcal{N}_t^{-}$.

\smallskip

By Propositions~\ref{u12} and ~\ref{prop:U0null} we have that on the event $\mathcal{E}_t\cap \mathcal{E}_t^{[2,\infty)}$ almost surely 
\begin{align}
\label{eq:unull+}
u(t,\Za)=(1+o(1))\sum_{y\in \mathrm{Null}^t_{1+}}U(t,y),
\end{align}
as $t\to\infty$.
Fix a path $y\in\mathrm{Null}^t_{1+}$ and note that $y_{\ell(y)}=\Za$. We wish to extract from the path terms involving potential values of sites in $\mathcal{N}_t$. To do this first we recall
that by~\eqref{uy1} 
\begin{align}
\label{eq:UytoI}
U(t,y)=e^{-2t}I_{\ell(y)}(t;\xi({\bf y})),
\end{align}
where $\xi({\bf y})$ denotes the sequence $\xi(y_0),\ldots,\xi(y_{\ell(y)})$. 
Then, for any $z\in\mathcal{N}_t^+$
it follows from the recursion~\eqref{rec2} and the symmetry of $I$ proved in Lemma~\ref{bounds} 
that 
\[
I_{\ell(y)}(t;\xi({\bf y}))= \frac1{\xi(\Za)-\xi(z)}
\Big[I_{\ell(y)-1}\big(t;\xi({\bf y}\backslash\{z\})\big)-I_{\ell(y)-1}\big(t;\xi({\bf y}\backslash\{\Za\})\big)\Big],
\]
where $\xi({\bf y}\backslash \{z\})$ denotes the sequence $\xi(y_0),\ldots,\xi(y_{\ell(y)})$ with the occurrence of $\xi(z)$ removed (note that since $y\in \mathrm{Null}^t_{1+}$ it makes exactly one visit to $z$). 
On the event $\mathcal{E}_t$ we have 
$\xi(z)<\xi(\Za)$ and therefore we can 
use Lemma \ref{L:Ixtoy} and obtain
\begin{align*}
I_{\ell(y)-1}\big(t;\xi({\bf y}\backslash\{\Za\})\big)\le \frac{\ell(y)}{(\xi(\Za)-\xi(z))t}
I_{\ell(y)-1}\big(t;\xi({\bf y}\backslash\{z\})\big).
\end{align*}
Observe that we have $\xi(\Za)-\xi(z)>a_tf_t$ on the event 
$\mathcal{E}_t$ as well as $\ell(y)\le R_t<r_tg_t(1+f_t)$. Plugging this into the above gives
\begin{align*}
I_{\ell(y)-1}\big(t;\xi({\bf y}\backslash\{\Za\})\big)\le\frac{g_t(1+f_t)}{f_t\log t}
I_{\ell(y)-1}\big(t;\xi({\bf y}\backslash\{z\})\big),
\end{align*}
and thus
\begin{align*}
I_{\ell(y)}(t;\xi({\bf y}))
=I_{\ell(y)-1}\big(t;\xi({\bf y}\backslash\{z\})\big)
\Big(1+O\Big(\frac{g_t}{f_t\log t}\Big)\Big)
\frac{1}{\xi(\Za)-\xi(z)}
\end{align*}
on $\mathcal{E}_t$, as $t\to\infty$. Iterating this procedure for all 
$z\in\mathcal{N}_t^+$ and observing that 
\begin{align*}
\frac{N_t^+g_t}{f_t\log t}<
\frac{\delta_t^{-\alpha}\log(1/\delta_t)g_t}{f_t\log t}\to 0
\end{align*}
on the event $\mathcal{E}_t^{[2,\infty)}$, 
we obtain that 
\begin{align*}
I_{\ell(y)}(t;\xi({\bf y}))
=(1+o(1))I_{\ell(y)-N_t^+}\big(t;\xi({\bf y}\backslash\mathcal{N}_t)\big)
\prod_{z\in\mathcal{N}_t^+}\frac{1}{\xi(\Za)-\xi(z)}
\end{align*}
on $\mathcal{E}_t\cap \mathcal{E}_t^{[2,\infty)}$,
as $t\to\infty$,
where $\xi({\bf y}\backslash\mathcal{N}_t)$ denotes the sequence $\xi(y_0),\ldots,\xi(y_{\ell(y)})$ with all occurrences of $\xi(z)$ with $z\in\mathcal{N}_t$ removed.
Combining this with~\eqref{eq:unull+} 
and~\eqref{eq:UytoI} we obtain
\begin{align*}
u(t,\Za)=(1+o(1))e^{-2t}
\prod_{z\in\mathcal{N}_t^+}\frac{1}{\xi(\Za)-\xi(z)}
\sum_{y\in{\rm Null}^t_{1+}}
I_{\ell(y)-N_t^+}\big(t;\xi({\bf y}\backslash\mathcal{N}_t)\big)
\end{align*}
and hence 
\begin{align}
\log u(t,\Za)
&=
-\sum_{
z\in \mathcal{N}_t^+}\log \Big(1-\frac{\xi(z)}{\xi(\Za)}\Big)
-N_t^+\log\xi(\Za)\notag\\
&+\log\sum_{y\in{\rm Null}^t_{1+}}
I_{\ell(y)-N_t^+}\big(t;\xi({\bf y}\backslash\mathcal{N}_t)\big)
-2t+o(1)
\label{fff1}
\end{align}
on $\mathcal{E}_t\cap \mathcal{E}_t^{[2,\infty)}$,
as $t\to\infty$. Similarly, 
\begin{align}
\log u(t,-\Za)
&=
-\sum_{z\in \mathcal{N}_t^-}\log \Big(1-\frac{\xi(z)}{\xi(\Za)}\Big)
-N_t^-\log\xi(\Za)\notag\\
&+\log\sum_{y\in{\rm Null}^t_{1-}}
I_{\ell(y)-N_t^-}\big(t;\xi({\bf y}\backslash\mathcal{N}_t)\big)
-2t+o(1)
\label{fff2}
\end{align}
on $\mathfrak{D}\cap\mathcal{E}_t\cap \mathcal{E}_t^{[2,\infty)}$,
as $t\to\infty$. Combining~\eqref{fff1} and~\eqref{fff2}
we obtain the desired result with 
\begin{align*}
P_t
&=N_t^+\log\xi(\Za)+\log\sum_{y\in{\rm Null}^t_{1+}}
I_{\ell(y)-N_t^+}\big(t;\xi({\bf y}\backslash\mathcal{N}_t)\big)\\
&-N_t^-\log\xi(\Za)-\log\sum_{y\in{\rm Null}^t_{1-}}
I_{\ell(y)-N_t^-}\big(t;\xi({\bf y}\backslash\mathcal{N}_t)\big),
\end{align*}
which is obviously $\mathcal{F}_t$-measurable.
\end{proof}

We now study the scale of the fluctuations due to the sites in $\mathcal{N}_t$, showing in particular that these fluctuations are unbounded.

\begin{prop}
\label{P:var}
As $t \to \infty$,
\[  \left[\text{\rm Var}_{\mathcal{F}_t}   Q_t  \right]^{-1} \id_{\mathcal{E}_t \cap \mathcal{E}_t^{[2,\infty)}}  \to 0 \quad 
\text{almost surely} \]
and
\[   \max_{z \in \mathcal{N}_t} 
\big|{\mathrm E}_{\mathcal{F}_t} Q_t(z) \big| \id_{\mathcal{E}_t }  \to  0 \quad \text{almost surely.} \]
\end{prop}

\begin{proof}
We work throughout on the event $\mathcal{E}_t$. 
Using $\delta_t\xi(\Za)>\delta_tf_ta_t>1$ for the upper bound
and $\delta_t\xi(\Za)/\Psi_t(\Za)<\delta_t/f_t\to 0$
for the lower bound 
we obtain   
\begin{align*}
\int_1^{\infty}\frac{\alpha dy}{y^{\alpha+1}}
\one_{\{y-\frac{|z|}{t}\log y<\Psi_t(\Za),\xi(z)>\delta_t\xi(\Za)\}}
\le \int_{\delta_t\xi(\Za)}^{\infty}\frac{\alpha dy}{y^{\alpha+1}}=\big(\delta_t\xi(\Za)\big)^{-\alpha}
\end{align*}
and 
\begin{align*}
\int_1^{\infty}\frac{\alpha dy}{y^{\alpha+1}}
\one_{\{y-\frac{|z|}{t}\log y<\Psi_t(\Za),\xi(z)>\delta_t\xi(\Za)\}}
\ge \int_{\delta_t\xi(\Za)}^{\Psi_t(\Za)}\frac{\alpha dy}{y^{\alpha+1}}=(1+o(1))\big(\delta_t\xi(\Za)\big)^{-\alpha}
\end{align*}
implying 
\begin{align}
\label{pb}
\int_1^{\infty}\frac{\alpha dy}{y^{\alpha+1}}
\one_{\{y-\frac{|z|}{t}\log y<\Psi_t(\Za),\xi(z)>\delta_t\xi(\Za)\}}=(1+o(1))\big(\delta_t\xi(\Za)\big)^{-\alpha}
\end{align}
as $t\to\infty$ uniformly for all $z$ almost surely.  

\smallskip

Further, using $\delta_t\xi(\Za)>1$, the change of variables $y=u\xi(\Za)$, and $\Psi_t(\Za)/\xi(\Za)>f_t$ we get 
\begin{align}
\int_1^{\infty}&\frac{\alpha dy}{y^{\alpha+1}}
\log^2\Big(1-\frac{y}{\xi(\Za)}\Big)
\one_{\{y-\frac{|z|}{t}\log y<\Psi_t(\Za),\delta_t\xi(\Za)<y<\xi(\Za)\}}\notag\\
&\ge \int_{\delta_t\xi(\Za)}^{\Psi_t(\Za)}\frac{\alpha dy}{y^{\alpha+1}}
\log^2\Big(1-\frac{y}{\xi(\Za)}\Big)
\ge \frac{\alpha}{\xi(\Za)^{\alpha}} \int_{\delta_t}^{f_t} u^{-\alpha-1}
\log^2(1-u)du.
\label{www1}
\end{align}
Since $\delta_t/f_t\to 0$ we have 
\begin{align}
\label{www2}
 \int_{\delta_t}^{f_t} u^{-\alpha-1}
\log^2(1-u)du=(1+o(1)) \times
\left\{\begin{array}{ll}
\frac{1}{\alpha-2}\delta_t^{2-\alpha}, & \alpha>2,\\ 
\log(1/\delta_t), & \alpha=2.
\end{array}\right.
\end{align}
Combining~\eqref{pb}, ~\eqref{www1}, and~\eqref{www2} 
we obtain 
\begin{align}
\label{www5}
{\mathrm E}_{\mathcal{F}_t}Q_t^2(z)
\ge (1+o(1)) \times
\left\{\begin{array}{ll}
\frac{\alpha}{\alpha-2}\delta_t^{2}, & \alpha>2,\\ 
\alpha\delta_t^2\log(1/\delta_t), & \alpha=2.
\end{array}\right.
\end{align}
Now, using $\delta_t\xi(\Za)>1$ and the change of variables $y=u\xi(\Za)$ we compute 
\begin{align}
-\int_1^{\infty}&\frac{\alpha dy}{y^{\alpha+1}}
\log\Big(1-\frac{y}{\xi(\Za)}\Big)
\one_{\{y-\frac{|z|}{t}\log y<\Psi_t(\Za),\delta_t\xi(\Za)<y<\xi(\Za)\}}\notag\\
&\le -\int_{\delta_t\xi(\Za)}^{\xi(\Za)}\frac{\alpha dy}{y^{\alpha+1}}
\log\Big(1-\frac{y}{\xi(\Za)}\Big)
= -\frac{\alpha}{\xi(\Za)^{\alpha}} \int_{\delta_t}^{1} u^{-\alpha-1}
\log(1-u)du.
\label{www3}
\end{align}
Observe that 
\begin{align}
\label{www4}
-\int_{\delta_t}^{1} u^{-\alpha-1}
\log(1-u)du=(1+o(1))\frac{1}{\alpha-1}\delta_t^{1-\alpha}.
\end{align}
Combining~\eqref{pb}, \eqref{www3} and~\eqref{www4} we get 
\begin{align}
\label{www6} 
\big |{\mathrm E}_{\mathcal{F}_t}Q_t(z)\big|
\le (1+o(1))
\frac{\alpha}{\alpha-1}\delta_t.
\end{align}
Combining~\eqref{www5} and~\eqref{www6} we obtain
\begin{align}
\label{varz}
\text{Var}_{\mathcal{F}_t} Q_t(z)
\ge (1+o(1)) \times
\left\{\begin{array}{ll}
c(\alpha)\delta_t^{2}, & \alpha>2,\\ 
\alpha\delta_t^2\log(1/\delta_t), & \alpha=2.
\end{array}\right.,
\end{align} 
where $c(\alpha)=\frac{\alpha}{(\alpha-2)(\alpha-1)^2}$.
To prove the first result, it remains to notice that $|\mathcal{N}_t|>\delta_t^{-\alpha}/\log\log(1/\delta_t)$ on 
event $\mathcal{E}_t^{[2,\infty)}$ and that  
$\delta_t^{2-\alpha}/\log\log(1/\delta_t)\to \infty$
if $\alpha>2$ and $\delta_t^{2-\alpha}\log(1/\delta_t)/\log\log(1/\delta_t)\to \infty$ if $\alpha=2$. The second result follows immediately from~\eqref{www6}. 
\end{proof}

The final step is to apply a central limit theorem to show that the fluctuations due to the sites in $\mathcal{N}_t$ are in the Gaussian universality class. For each $z\in\mathcal{N}_t$, denote 
\begin{align}
\label{defv}
V_t(z)=\frac{Q_t(z)-{\mathrm E}_{\mathcal{F}_t}Q_t(z)}
{\sqrt{\text{Var}_{\mathcal{F}_t}Q_t}} \, , \quad V_t=\sum_{z\in \mathcal{N}_t}V_t(z),
\end{align}
and denote by 
\begin{align*}
F_{V_t}(x)=\Prob_{\mathcal{F}_t}\big(V_t\le x\big)
\end{align*}
the conditional distribution function of $V_t$. \

\begin{prop}
\label{P:clt}
As $t \to \infty$,
\[  \sup_{x \in \mathbb{R}} |F_{V_t} (x) - \Phi(x) |  \id_{\mathcal{E}_t \cap \mathcal{E}_t^{[2,\infty)}} \to 0 \quad \text{almost surely} , \]
where $\Phi$ denotes the distribution function of a standard normal random variable.
\end{prop}

\begin{proof}
This result follows from an application of the central limit theorem that we state and prove in Appendix \ref{clt}. It remains to verify that the conditions of the theorem are satisfied. 

\smallskip

First note that, conditionally on $\mathcal{F}_t$, the random variables $V_t(z)$, $z \in \mathcal{N}_t$, are independent. Moreover, by construction, for each $t>0$ and 
$z \in \mathcal{N}_t$,
\[ {\mathrm E}_{\mathcal{F}_t}V_t(z) = 0  \quad \text{and} \quad   \sum_{z \in \mathcal{N}_t} {\mathrm E}_{\mathcal{F}_t} V_t^2(z)  = 1   \quad  \text{almost surely.}  
\] 
Hence it remains to verify that, for each $\varepsilon > 0$,
\begin{align}
\label{tt3}
\sum_{z \in \mathcal{N}_t} 
{\mathrm E}_{\mathcal{F}_t}\big[ V_t^2(z) 
\id_{\{|V_t(z)| \ge \varepsilon \}}  \big]  
\id_{\mathcal{E}_t \cap \mathcal{E}_t^{[2,\infty)}}  \to 0   \quad \text{almost surely}.
\end{align}

For the rest of the proof assume the event 
$\mathcal{E}_t \cap \mathcal{E}_t^{[2,\infty)}$ holds, and remark that, according to~\eqref{defv} 
\begin{align*}
Q_t(z)={\mathrm E}_{\mathcal{F}_t}Q_t(z)+V_t(z)\sqrt{\text{Var}_{\mathcal{F}_t}Q_t}.
\end{align*}
Since $Q_t(z)$ and ${\mathrm E}_{\mathcal{F}_t}Q_t(z)$
are either both non-negative or non-positive almost surely, we obtain, using that $\text{Var}_{\mathcal{F}_t} Q_t$ diverges  and that $E_{\mathcal{F}_t}Q_t(z)$ tends to zero 
by Proposition~\ref{P:var},  
\begin{align*}
\big\{|V_t(z)| \ge \varepsilon\big\} \subseteq 
\big\{|Q_t(z)|\ge \e\sqrt{\text{Var}_{\mathcal{F}_t}Q_t}\big\}.
\end{align*}
Hence 
\begin{align}
\sum_{z \in \mathcal{N}_t} {\mathrm E}_{\mathcal{F}_t} 
\big[V_t^2(z)\id_{\{ |V_t(z)| \ge \varepsilon \}} \big]  
& \le  \frac{1}{\text{Var}_{\mathcal{F}_t}Q_t}
\sum_{z \in \mathcal{N}_t} {\mathrm E}_{\mathcal{F}_t} 
\Big[\big(Q_t(z)-{\mathrm E}_{\mathcal{F}_t}Q_t(z)\big)^2\id_{\big\{ |Q_t(z)| \ge \varepsilon \sqrt{\text{Var}_{\mathcal{F}_t}Q_t}\big\}} \Big] 
\label{tt1}
\end{align}
Combining this with~\eqref{varvar} we observe that in order to prove~\eqref{tt3} it suffices to show that 
\begin{align}
\label{tt4}
\frac{{\mathrm E}_{\mathcal{F}_t} 
\Big[\big(Q_t(z)-{\mathrm E}_{\mathcal{F}_t}Q_t(z)\big)^2\id_{\big\{ |Q_t(z)| \ge \varepsilon \sqrt{\text{Var}_{\mathcal{F}_t}Q_t}\big\}} \Big] }
{\text{Var}_{\mathcal{F}_t}Q_t(z)}\to 0
\end{align}
uniformly in $z$ almost surely.

\smallskip

Denote 
\begin{align*}
\nu_t^{\e}=\exp\big\{-\varepsilon \sqrt{\text{Var}_{\mathcal{F}_t}Q_t}\big\}.
\end{align*}
Then
\begin{align*}
\big\{ |Q_t(z)| \ge \varepsilon \sqrt{\text{Var}_{\mathcal{F}_t}Q_t}\big\}
=\big\{\xi(z)>(1-\nu_t^{\e})\xi(\Za)\big\}
\subset\big\{\xi(z)>\delta_t\xi(\Za)\big\}
\end{align*}
since $\nu_t^{\e}\to 0$ almost surely on the event $\mathcal{E}_t \cap \mathcal{E}_t^{[2,\infty)}$ by Proposition~\ref{P:var}. Similarly to the proof of Proposition~\ref{P:var} we use the change of 
variables $y=u\xi(\Za)$ to compute, for $k\in\{0,1,2\}$,
\begin{align*}
&\int_1^{\infty}\frac{\alpha}{y^{\alpha+1}}
\Big|\log\Big(1-\frac{y}{\xi(\Za)}\Big)\Big|^k
\one_{\big\{y-\frac{|z|}{t}\log y<\Psi_t(\Za),\delta_t\xi(\Za)<y<\xi(\Za), y>(1-\nu_t^{\e})\xi(\Za)\big\}}dy\notag\\
&\le \int_{(1-\nu_t^{\e})\xi(\Za)}^{\infty}\frac{\alpha}{y^{\alpha+1}}
\Big|\log\Big(1-\frac{y}{\xi(\Za)}\Big)\Big|^k dy
= \frac{\alpha}{\xi(\Za)^{\alpha}} \int_{1-\nu_t^{\e}}^{\infty} u^{-\alpha-1}
|\log(1-u)|^kdu
\sim\frac{c_k}{\xi(\Za)^{\alpha}}
\end{align*}
with some $c_k>0$. By the second part of Proposition~\ref{P:var} and using~\eqref{pb} we obtain
\begin{align*}
{\mathrm E}_{\mathcal{F}_t} 
\Big[\big(Q_t(z)-{\mathrm E}_{\mathcal{F}_t}Q_t(z)\big)^2\id_{\big\{ |Q_t(z)| \ge \varepsilon \sqrt{\text{Var}_{\mathcal{F}_t}Q_t}\big\}} \Big] 
\sim c_2\delta_t^{\alpha}
\end{align*}
uniformly in $z$ almost surely. Combining this with~\eqref{varz} we arrive at~\eqref{tt4}. 
\end{proof}

We are now ready to complete the proof of Theorem \ref{main2}. The point is that, since we have shown that the fluctuations due to $\mathcal{N}_t$ are unbounded and in the Gaussian universality class, they place negligible probability mass on any bounded scale. Hence we have the result.

\smallskip

\subsection{Completion of the proof of Theorem~\ref{main2}}
Let $c>0$. As a direct corollary of Theorem \ref{main0}, 
\begin{align*}
\Prob\Big(\Big\{\Big|\log\frac{u(t,\Za)}{u(t,-\Za)}\Big|<c\Big\}\cap \mathfrak{D}^c_t \Big)\to 0,
\end{align*}
so it remains to show the convergence on the event $\mathfrak{D}_t$. \smallskip

Let $c>0$. Since $\text{Prob}(\mathfrak{D}_t)\not\to 0$ by Theorem~\ref{main0}
and $\text{Prob}(\mathcal{E}_t\cap \mathcal{E}_t^{[2,\infty)})\to 1$ by Propositions~\ref{e2} and~\ref{e0}, it suffices to show that, as $t\to\infty$, 
\begin{align*}
\Prob\Big(\Big\{\Big|\log\frac{u(t,\Za)}{u(t,-\Za)}\Big|<c\Big\}\cap \mathfrak{D}_t\cap\mathcal{E}_t\cap \mathcal{E}_t^{[2,\infty)}\Big)\to 0,
\end{align*}
By Proposition~\ref{P:decomp} it is then enough to prove that, as $t\to\infty$, 
\begin{align*}
\Prob\big(\big\{|Q_t-P_t|<2c\big\}\cap \mathfrak{D}_t\cap\mathcal{E}_t\cap \mathcal{E}_t^{[2,\infty)}\big)\to 0,
\end{align*}
for which, in turn, it suffices to show that 
\begin{align*}
\mathrm{E}\Big[\Prob_{\mathcal{F}_t}
\big\{|Q_t-P_t|<2c\big\}
\one_{\mathcal{E}_t\cap \mathcal{E}_t^{[2,\infty)}}\Big]\to 0.
\end{align*}
Observe that even though the event $\mathcal{E}_t$ does not belong to $\mathcal{F}_t$ we can take it out of $\Prob_{\mathcal{F}_t}$ by Proposition~\ref{e0} and since the function under $\mathrm{E}$ is bounded. 
Now, by the dominated convergence theorem, it remains to prove that 
\begin{align}
\label{tt6}
\Prob_{\mathcal{F}_t}
\big\{|Q_t-P_t|<2c\big\}\to 0
\end{align}
almost surely on $\mathcal{E}_t\cap \mathcal{E}_t^{[2,\infty)}$.
To do so, assume that $\mathcal{E}_t\cap \mathcal{E}_t^{[2,\infty)}$ holds and observe that  
\begin{align*}
Q_t=V_t\sqrt{\text{Var}_{\mathcal{F}_t}Q_t}+\mathrm{E}_{\mathcal{F}_t}Q_t.
\end{align*}
Hence~\eqref{tt6} is equivalent to showing that, almost surely, 
\begin{align}
\label{tt7}
\Prob_{\mathcal{F}_t}
\Big\{V_t\in 
\big[\text{Var}_{\mathcal{F}_t}Q_t\big]^{-\frac 1 2}\big(P_t-\mathrm{E}_{\mathcal{F}_t}Q_t-2c,P_t-\mathrm{E}_{\mathcal{F}_t}Q_t+2c\big)\Big\}\to 0
\end{align} 
Since $P_t$, $\mathrm{E}_{\mathcal{F}_t}Q_t$, and 
$\text{Var}_{\mathcal{F}_t}Q_t$ are $\mathcal{F}_t$-measurable, and the length of the interval on the right-hand side of $\in$ tends to zero by Proposition~\ref{P:var}, 
\eqref{tt7} now follows from~Proposition~\ref{P:clt}.

\bigskip


\appendix

\section{Point process arguments}
\label{A:B}
In this appendix we give the details of our point process arguments in Section \ref{s:pp}, and in particular provide the proofs of Lemma \ref{lppp} and Propositions \ref{sl1}--\ref{prob} and \ref{td}. All notation in this appendix is carried over from the main part of the paper.

\smallskip

\subsection{Point process convergence for the potential}
We first give the proof of the point process convergence for the rescaled potential in Lemma \ref{lppp}.

\begin{proof}[Proof of Lemma \ref{lppp}]
Let $(\sigma_n)_{n\in\N_0}$ be a collection of random variables defined by $\sigma_n=1$ if $n \in D$ and $\sigma_n=-1$ if $n \in E$. Define the point process
\begin{align*}
\Sigma_s=\sum_{z\in \Z}\e\Big(\frac{z}{s},\frac{\sigma_{|n|}\xi_0(z)}{s^{1/\alpha}}\Big)
\end{align*}
and let $\Sigma$ be a Poisson point process on $G=\big(\R\times [-\infty,0)\big) \cup \big(\R \times (0,\infty]\big)$
with the intensity measure $\hat\mu$ which equals $p\mu$ on the lower half-plane 
and $q\mu$ on the upper half-plane. It suffices to show that $\Sigma_s$ converges in law to $\Sigma$ on the state space $G$, 
as $\Pi_s^{\ssup{d,+}}$ can be represented 
by the restriction of $\Sigma_s$ to the right half of the lower half plane (reflected with respect to the $x$-axes)
and $\Pi_s^{\ssup e}$ by the restriction of $\Sigma_s$ to the upper half plane. Clearly the corresponding restrictions of $\Sigma$ have the same law as the pair $(\Pi^{\ssup{d,+}}, \Pi^{\ssup e})$. 

\smallskip
Let $\mathcal{C}_K^+$ be the set of continuous functions $h:G\to\mathbb{R}_+$ with compact support.
For any $s$, denote by   
\begin{align*}
\Phi_s(h)=\mathrm{E} \exp\Big\{-\int hd\Sigma_s\Big\}
\qquad\text{and}\qquad
\Phi(h)=\mathrm{E} \exp\Big\{-\int hd\Sigma\Big\},
\end{align*}
the Laplace transforms of $\Sigma_s$ and $\Sigma$, where $h\in\mathcal{C}_K^+$. 
Recall from~\cite[Prop.\ 3.6]{resnick} that since $\Sigma$ is a Poisson point process its Laplace transform is given by 
\begin{align}
\label{lt}
\log\Phi(h)=-\iint_{\R^2}(1-e^{-h(x,y)})\hat\mu(dx,dy).
\end{align}
By~\cite[Prop.\ 3.19]{resnick} it suffices to show that $\Phi_t(h)\to\Phi(h)$ for all $h\in\mathcal{C}_K^+$. We have 
\begin{align*}
\log \Phi_s(h)
&=\log\mathrm{E} \exp\Big\{-\int hd\Sigma_s\Big\}
=\log \mathrm{E}\exp\Big\{-\sum_{z\in\Z}h
\Big(\frac{z}{s},\frac{\sigma_{|z|}\xi_0(z)}{s^{1/\alpha}}\Big)\Big\}\\
&=\sum_{z=0}^{\infty}\log \mathrm{E}\exp\Big\{-\delta_{z}\Big[h
\Big(\frac{z}{s},\frac{\sigma_{z}\xi_0(z)}{s^{1/\alpha}}\Big)
+h\Big(-\frac{z}{s},\frac{\sigma_{|z|}\xi_0(-z)}{s^{1/\alpha}}\Big)\Big]\Big\}, 
\end{align*}
where $\delta_0=1/2$ and $\delta_z=1$ otherwise. 
Computing the expectation with respect to $(\sigma_n)$ first, we have 
\begin{align*}
\log \Phi_s(h)
&=\sum_{z=0}^{\infty}\log \Big(p\mathrm{E}\exp\Big\{-\delta_{z}\Big[h
\Big(\frac{z}{s},-\frac{\xi_0(z)}{s^{1/\alpha}}\Big)
+h\Big(-\frac{z}{s},-\frac{\xi_0(-z)}{s^{1/\alpha}}\Big)\Big]\Big\}\\
&\phantom{aaaaaaa}+q\mathrm{E}\exp\Big\{-\delta_{z}\Big[h
\Big(\frac{z}{s},\frac{\xi_0(z)}{s^{1/\alpha}}\Big)
+h\Big(-\frac{z}{s},\frac{\xi_0(-z)}{s^{1/\alpha}}\Big)\Big]\Big\}\Big). 
\end{align*}
Further, using the independence of $\xi(z)$ and $\xi(-z)$
we get 
\begin{align}
\log \Phi_s(h)
&=\log\Big(p\mathrm{E}\exp\Big\{-h
\Big(\frac{0}{s},-\frac{\xi_0(0)}{s^{1/\alpha}}\Big)\Big\}
+q\mathrm{E}\exp\Big\{-h
\Big(\frac{0}{s},\frac{\xi_0(0)}{s^{1/\alpha}}\Big)\Big\}\Big)\notag\\
&\quad +\sum_{z=1}^{\infty}\log \Big(p\mathrm{E}\exp\Big\{-h
\Big(\frac{z}{s},-\frac{\xi_0(z)}{s^{1/\alpha}}\Big)\Big\}
\mathrm{E}\exp\Big\{-h\Big(-\frac{z}{s},-\frac{\xi_0(-z)}{s^{1/\alpha}}\Big)\Big\}\notag\\
&\phantom{aaaaaaa}+q\mathrm{E}\exp\Big\{-h
\Big(\frac{z}{s},\frac{\xi_0(z)}{s^{1/\alpha}}\Big)\Big\}
\mathrm{E}\exp\Big\{-h\Big(-\frac{z}{s},\frac{\xi_0(-z)}{s^{1/\alpha}}\Big)\Big\}\Big). 
\label{ee3}
\end{align}
Using the substitution $y=us^{1/\alpha}$, compute  
\begin{align}
\mathrm{E}\exp\Big\{-h
\Big(\frac{z}{s},\frac{\xi_0(z)}{s^{1/\alpha}}\Big)\Big\}
&=1-\mathrm{E}\Big[1-\exp\Big\{-h
\Big(\frac{z}{s},\frac{\xi_0(z)}{s^{1/\alpha}}\Big)\Big\}\Big]\notag\\
&=1-\int_1^{\infty} \Big[1-\exp\Big\{-h
\Big(\frac{z}{s},\frac{y}{s^{1/\alpha}}\Big)\Big\}\Big]
\frac{\alpha dy}{y^{\alpha+1}}\notag\\
&=1-\frac{1}{s}\int_{0}^{\infty} \Big[1-\exp\Big\{-h
\Big(\frac{z}{s},u\Big)\Big\}\Big]
\frac{\alpha du}{u^{\alpha+1}}
\label{ee1}
\end{align}
for all $s$ such that $s^{-1/\alpha}\le c$, where $c$ is the  distance from the set $A$ to the $y$-axes, which is positive since $A$ is compactly supported in $G$. 
Observe that 
\begin{align}
\int_{0}^{\infty} \Big[1-\exp\Big\{-h
\Big(\frac{z}{s},u\Big)\Big\}\Big]
\frac{\alpha du}{u^{\alpha+1}}\le 
\int_c^{\infty}\frac{\alpha du}{u^{\alpha+1}}<\infty
\label{ee2}
\end{align}
uniformly in $z$. Repeating the calculations~\eqref{ee1} and~\eqref{ee2} for all expectations in~\eqref{ee3} and doing the Taylor expansion of the logarithm, which is possible by~\eqref{ee2}, we obtain  
\begin{align*}
&\log \Phi_s(h)\\
&=\sum_{z=1}^{\infty}\log \Big(1-\Big[\frac{p}{s}
\int_{0}^{\infty} \Big[1-\exp\Big\{-h
\Big(\frac{z}{s},-u\Big)\Big\}\Big]
\frac{\alpha du}{u^{\alpha+1}}
+\frac{p}{s}\int_{0}^{\infty} \Big[1-\exp\Big\{-h
\Big(-\frac{z}{s},-u\Big)\Big\}\Big]
\frac{\alpha du}{u^{\alpha+1}}\\
&\quad+\frac{q}{s}\int_{0}^{\infty} \Big[1-\exp\Big\{-h
\Big(\frac{z}{s},u\Big)\Big\}\Big]
\frac{\alpha du}{u^{\alpha+1}}+\frac{q}{s}\int_{0}^{\infty} \Big[1-\exp\Big\{-h
\Big(-\frac{z}{s},u\Big)\Big\}\Big]
\frac{\alpha du}{u^{\alpha+1}}\Big](1+o(1))\Big)+o(1)\\
&=-\sum_{z=1}^{\infty}\Big[\frac{p}{s}
\int_{-\infty}^{0} \Big[1-\exp\Big\{-h
\Big(\frac{z}{s},u\Big)\Big\}\Big]
\frac{\alpha du}{|u|^{\alpha+1}}
+\frac{p}{s}\int_{-\infty}^{0} \Big[1-\exp\Big\{-h
\Big(-\frac{z}{s},u\Big)\Big\}\Big]
\frac{\alpha du}{|u|^{\alpha+1}}\\
&\quad +\frac{q}{s}\int_{0}^{\infty} \Big[1-\exp\Big\{-h
\Big(\frac{z}{s},u\Big)\Big\}\Big]
\frac{\alpha du}{u^{\alpha+1}}+\frac{q}{s}\int_{0}^{\infty} \Big[1-\exp\Big\{-h
\Big(-\frac{z}{s},u\Big)\Big\}\Big]
\frac{\alpha du}{u^{\alpha+1}}\Big](1+o(1))\\ 
&=-\sum_{z\in \Z}\Big[\frac{p}{s}
\int_{-\infty}^{0} \Big[1-\exp\Big\{-h
\Big(\frac{z}{s},u\Big)\Big\}\Big]
\frac{\alpha du}{|u|^{\alpha+1}}
+\frac{q}{s}\int_{0}^{\infty} \Big[1-\exp\Big\{-h
\Big(\frac{z}{s},u\Big)\Big\}\Big]
\frac{\alpha du}{u^{\alpha+1}}\Big](1+o(1))\\ 
&=-p\iint_{\R\times (-\infty,0)}
\big(1-e^{-h(s,u)}\big)\mu(ds,du)
-q\iint_{\R\times (0,\infty)}
\big(1-e^{-h(s,u)}\big)\mu(ds,du)+o(1). 
\end{align*}
According to~\eqref{lt} we obtain  
$\log\Phi_s(h)\to \log\Phi(h)$ and so $(\Pi^{\ssup{d,+}}_s, \Pi^{\ssup{e}}_s)$ converges in law to $(\Pi^{\ssup{d,+}}, \Pi^{\ssup e})$. This implies that $(\Pi^{\ssup{d,+}}_s,\Pi^{\ssup{d,-}}_s, \Pi^{\ssup{e}}_s)$ converges in law to $(\Pi^{\ssup{d,+}}, \Pi^{\ssup{d,-}},\Pi^{\ssup e})$ since 
$\Pi_s^{\ssup{d,-}}$ and $\Pi^{\ssup{d,-}}$ are deterministic reflections of $\Pi_s^{\ssup{d,+}}$ and $\Pi^{\ssup{d,+}}$ respectively, and so we have the result.
\end{proof}

\subsection{Convergence for functionals of the potential}
We now show how to use the point process convergence of the rescaled potential to pass certain functionals of $\xi$ to the limit. In particular, we prove Propositions \ref{sl1}--\ref{prob} and \ref{td}. 

\smallskip

For technical reasons, instead of working directly with the functional $\Psi_t$ it will be convenient to work with a simpler version $\hat{\Psi}_t$ in which the penalty term is not random (and depends on $|z|$ and $t$ only). To this end, for any $t>0$ and $z\in\Z$, let 
\begin{align*}
\hat\Psi_t(z)=\xi(z)-\frac{|z|}{t}\rho\log t. 
\end{align*}
Denote by $\Zaa$ a maximiser of $\hat \Psi_t$
and by $\Zbb$ a maximiser of $\hat\Psi_t$ on the set $\Z\backslash\{\hat Z_t^{\ssup 1},-\hat Z_t^{\ssup 1}\}$
if $\hat Z_t^{\ssup 1}\in D$ and on the set $\Z\backslash\{\hat Z_t^{\ssup 1}\}$
if $\hat Z_t^{\ssup 1}\in E$.  By the same argument as in Lemma \ref{exi}, for each $i = 1,2$, there are at most two choices for $\Zcc$ and, moreover, there are two if and only if both are in $D$ and symmetrical about the origin. Denote 
\begin{align*}
\mathcal{\hat D}_t=\{\Zaa\in D\} .
\end{align*}
In Lemma \ref{LB1} we establish convergence for the maximisers of $\hat{\Psi}_t$; in Lemma \ref{coincide} we prove that the maximisers of $\Psi_t$ and $\hat{\Psi}_t$ are the same with overwhelming probability.

\smallskip

For all the arguments in this section, we shall make use of the point process $\bar \Pi_s$ defined by 
\begin{align*}
\bar \Pi_s(A)=\Pi_s^{\ssup e}(A)+\Pi_s^{\ssup e}(\hat A)+\Pi_s^{\ssup {d,+}}(A)
\end{align*}
for any Borel set $A$, where $\hat A$ denotes the reflection of $A$ with respect to the $y$-axis. This is the prelimit version of the Poisson point process $\bar \Pi$. 

\begin{lemma}\label{LB1}
As $t\to\infty$, \[\Big(\frac{|\Zaa|}{r_t},\frac{|\Zbb|}{r_t},\frac{\xi(\Zaa)}{a_t},\frac{\xi(\Zbb)}{a_t}\Big)
\Rightarrow (X^{\ssup 1},X^{\ssup 2},Y^{\ssup 1},Y^{\ssup 2}).\]
\end{lemma}
\begin{proof}
Observe that for all $z\in\Z$
\begin{align*}
\frac{\hat\Psi_t(z)}{a_t}=\frac{\xi(z)}{a_t}-\frac{|z|}{t a_t}\rho\log t
=\frac{\xi(z)}{a_t}-\rho \frac{|z|}{r_t}.
\end{align*}
Let $A$ be a compact Borel subset of 
$\{(x_1,x_2, y_1, y_2): y_1>  \rho x_2, y_2>\rho x_2 \} \subseteq ([0, \infty) \times (0,\infty])^2$ such that its boundary $\partial A$ has zero Lebesgue 
measure $\text{Leb}(\partial A)$. Hence 
$A\subseteq \{(x_1,x_2, y_1, y_2): y_1> a+ \rho x_2, y_2>a+\rho x_2\}$ with some $a>0$, and this set has finite measure $\mu$ by~\eqref{mumu}. 
Since $a_t=r_t^{1/\alpha}$ we obtain by Lemma~\ref{lppp} that 
\begin{align}
\Prob\Big(&\Big(\frac{|\Zaa|}{r_t},\frac{|\Zbb|}{r_t},\frac{\xi(\Zaa)}{a_t},\frac{\xi(\Zbb)}{a_t}\Big)\in A\Big)\notag\\
&=\int_A\Prob\Big(\bar \Pi_{r_t}(dx_1\times dy_1)=\bar \Pi_{r_t}(dx_2\times dy_2)=1,\notag\\
&\phantom{aaaaaaaallll} \bar \Pi_{r_t}\big(\big\{(u,v): v-\rho u>y_1-\rho x_1 \big\}\big)=0\notag\\
&\phantom{aaaaaaaallll} \bar \Pi_{r_t}\big(\big\{(u,v): y_2-\rho x_2<v-\rho u<y_1-\rho x_1\big\}\big)=0\Big)\notag\\
&\to\int_A\Prob_*\Big(\bar \Pi(dx_1\times dy_1)=\bar \Pi(dx_2\times dy_2)=1,\notag\\
&\phantom{aaaaaaaallllll}\bar \Pi\big(\{(u,v): v-\rho u>y_1-\rho x_1 \}\big)=0\notag\\
&\phantom{aaaaaaaallllll}\bar \Pi\big(\{(u,v): y_2-\rho|x_2|<v-\rho u<y_1-\rho x_1 \}\big)=0\Big).
\label{e11}
\end{align}
The probability under the integral equals 
\[ p(x_1, x_2, y_1, y_2) = \Prob_*\big(X^{\ssup 1}\in dx_1, X^{\ssup 2}\in dx_2, 
Y^{\ssup 1}\in dy_1, Y^{\ssup 2}\in dy_2\big) \]
by definition. To complete the proof, we need to show that the collection of test sets $A$ is big enough. To do so, it suffices to see that 
\begin{align*}
\int\limits_{ ([0, \infty) \times (0,\infty])^2} 
\!\!\!\!\!\!\!\!\!\!\!\!
& p(x_1,x_2, y_1, y_2)\one\big\{y_1>\rho x_1,y_2>\rho x_2 \big\}dx_1dx_2dy_1dy_2\\
&=\Prob_*\big(Y^{\ssup 1}>\rho X^{\ssup 1}
,Y^{\ssup 2}>\rho X^{\ssup 2})=1
\end{align*}
by Lemma~\ref{xy}.
\end{proof}

\begin{lemma} 
\label{coincide}
$\Prob\big( |\Za| = |\Zaa|, |\Zb| = |\Zbb| \big)\to 1$ as $t\to\infty$. 
\end{lemma}

\begin{proof}
Fix $\e>0$ and let $a>0$ be sufficiently small that 
\begin{align*}
A=\big\{(x_1,x_2,y_1,y_2): x_1 \le 1/a,  x_2 \le 1/a, y_1\ge a+\rho x_1 ,y_2\ge a+\rho x_2 \big\}  \subseteq ( [0, \infty) \times (0,\infty])^2
\end{align*}
is such that 
\begin{align}
\label{g66}
\Prob_*\big((X^{\ssup 1},X^{\ssup 2},Y^{\ssup 1}, Y^{\ssup 2})\in A\big)&>1-\e/2,
\end{align}
which is possible according to Lemma~\ref{xy}. Observe that 
\begin{align*}
\frac{\Psi_t(z)}{a_t}=\frac{\xi(z)}{a_t}-\frac{|z|}{t a_t}\log\xi(z)
=F_t\Big(\frac{|z|}{r_t}, \frac{\xi(z)}{a_t}\Big),
\end{align*}
where 
\begin{align*}
F_t(x,y)=y- x \cdot \frac{\log a_t}{\log t}- x \cdot\frac{\log y}{\log t}.
\end{align*}
It is easy to see that $F_t(x,y)\to y-\rho  x$ as $t\to\infty$. 
Using $\Psi_t(\Za)\ge \Psi_t(\Zb)> 1$ eventually almost surely, proved in
Lemma~\ref{exi}, we obtain 
\begin{align}
\Prob\big(& |\Za| = |\Zaa|, |\Zb| = |\Zbb| \big)\notag\\
&\ge \int_A\Prob\Big(\bar \Pi_{r_t}(dx_1\times dy_1)
= \bar \Pi_{r_t}(dx_2\times dy_2)=1,\notag\\
&\phantom{aaaaaaaallll} \bar \Pi_{r_t}\big(\big\{(u,v): v-\rho u>y_1-\rho x_1 \big\}\big)=0,\notag\\
&\phantom{aaaaaaaallll} \bar \Pi_{r_t}\big(\big\{(u,v): y_2-\rho x_2 <v-\rho u <y_1-\rho x_1 \big\}\big)=0,\notag\\
&\phantom{aaaaaaaallll} \bar \Pi_{r_t}\big(\{(u,v): F_t(u,v)>F_t(x_1,y_1)\}\big)=0\notag\\
&\phantom{aaaaaaaallll} \bar \Pi_{r_t}\big(\{(u,v): F_t(x_2,y_2)<F_t(u,v)<F_t(x_1,y_1)\}\big)=0\Big)\notag\\
&\phantom{aaaaalll}\times \one\{F_t(x_1,y_1)>F_t(x_2,y_2)>1/a_t\}.
\label{g2}
\end{align}
Let $b>0$ and denote $K_b=[0,b]\times (0,\infty)$. Observe that since $\xi(z)>1$ for all $z$ the point process $\bar \Pi_t$
has no points below the level $1/a_t$.  

\smallskip

First, by examining the graphs of $F_t$ one can see that  
the first set in~\eqref{g2} is close to the third and the second is close to the fourth if we restrict them to $K_b$, that is,
\begin{align*}
&\{(u,v): v>1/a_t, F_t(u,v)>F_t(x_1,y_1)\}\cap K_b+E^{\ssup{t,b}}_1(x_1,y_1)\\
& \qquad =\{(u,v): v-\rho u>y_1-\rho x_1\}\cap K_b+E^{\ssup{t,b}}_2(x_1,y_1)
\end{align*}
and 
\begin{align*}
&\{(u,v): v>1/a_t, F_t(x_2,y_2)<F_t(u,v)<F_t(x_1,y_1)\}\cap K_b+E^{\ssup{t,b}}_1(x_1,x_2,y_1,y_2)\\
&\qquad =\{(u,v): v-\rho u >y_1-\rho x_1 \}\cap K_b+E^{\ssup{t,b}}_2(x_1,x_2,y_1,y_2),
\end{align*}
where
\begin{align*}
\text{Leb}\big(E^{\ssup{t,b}}_i(x_1,y_1)\big)\to 0\qquad\text{and}\qquad \text{Leb}\big(E^{\ssup{t,b}}_i(x_1,x_2,y_1,y_2)\big)\to 0, \qquad i=1,2,
\end{align*}
as $t\to\infty$ uniformly on $A$. Moreover, we have 
\begin{align}
\label{g3}
\{(u,v): v>1/a_t, F_t(u,v)>F_t(x_1,y_1)\}\cap K_b\subseteq \big\{(u,v): v\ge  a/2+\rho u \big\}
\end{align}
and 
\begin{align}
\label{g4}
\{(u,v): v>1/a_t, F_t(x_2,y_2)<F_t(u,v)<F_t(x_1,y_1)\}\cap K_b\subseteq \big\{(u,v): v\ge  a/2+\rho u\big\}
\end{align}
eventually for all $t$ uniformly on $A$, and 
the set on the right hand side of~\eqref{g3} and~\eqref{g4}
has finite measure $\mu$ according to~\eqref{mumu}.

\smallskip

Second, the portion of the third and fourth set
in~\eqref{g2} belonging to the complement of $K_b$ is negligible since 
\begin{align*}
&\big\{(u,v): v>1/a_t, F_t(u,v)>F_t(x_1,y_1)\big\}\cap K_b^c
\subseteq \Big\{(u,v): v>  a+\frac{\rho}{2} u, u > b\Big\}.
\end{align*}
and 
\begin{align*}
\big\{(u,v): v>1/a_t, F_t(x_2,y_2)<F_t(u,v)<F_t(x_1,y_1)\big\}\cap K_b^c
\subseteq \Big\{(u,v): v>  a+\frac{\rho}{2} u ,  u >b\Big\},
\end{align*}
where, as $b \to \infty$,
\begin{align*}
\mu\Big(\Big\{(u,v): v>  a+\frac{\rho}{2} u, u>b\Big\}\Big)
=2\int_b^{\infty}\int_{a+\frac{\rho}{2}u}^{\infty}\frac{\alpha}{v^{\alpha+1}}dvdu
=2\Big(a+\frac{\rho b}{2}\Big)^{1-\alpha}\to 0 .
\end{align*} 

Third, we have 
\begin{align*}
\{F_t(x_1,y_1)>F_t(x_2,y_2)\}+E_1^{\ssup t}= \{y_1-\rho x_1 >y_2-\rho x_2 \} +E_2^{\ssup t},
\end{align*}
where, as $t \to \infty$,
\begin{align*}
\text{Leb}(E^{\ssup{t}}_i)\to 0, \qquad i=1,2.
\end{align*} 

By the three arguments above, and since $\bar \Pi_{r_t}\to \bar \Pi$ as $t\to\infty$ by Lemma~\ref{lppp}, it follows from~\eqref{g2} that 
\begin{align*}
\Prob\big(& |\Za| = |\Zaa|, |\Zb| = |\Zbb| \big)\notag\\
&\ge \int_A\Prob_*\Big(\bar \Pi_{r_t}(dx_1\times dy_1)
= \bar \Pi_{r_t}(dx_2\times dy_2)=1,\Big)\notag\\
&\phantom{aaaaaaaallll} \bar \Pi_{r_t}\big(\big\{(u,v): v-\rho u >y_1-\rho x_1 \big\}\big)=0,\notag\\
&\phantom{aaaaaaaallll} \bar \Pi_{r_t}\big(\big\{(u,v): y_2-\rho |x_2|<v-\rho u <y_1-\rho x_1 \big\}\big)=0\Big)-\e/2\\
&=\Prob_*\big((X^{\ssup 1},X^{\ssup 2},Y^{\ssup 1}, Y^{\ssup 2})\in A\big)-\e/2>1-\e
\end{align*}
eventually for all $t$ by~\eqref{g66}. 
\end{proof}

We are now ready to give the proof of Propositions \ref{sl1} and \ref{e0}. 

\begin{proof}[Proof of Proposition \ref{sl1}]
Part (i) follows from Lemmas \ref{LB1} and \ref{coincide}, so it remains to prove part (ii). On the event $\big\{ |\Za| = |\Zaa|, |\Zb| = |\Zbb| \big\}$ we have, for $i = 1,2$,
\begin{align*}
\frac{\hat\Psi(\hat Z_t^{\ssup i})}{a_t}-\frac{\Psi(Z_t^{\ssup i})}{a_t}
=\frac{|\hat Z_t^{\ssup i}|}{ta_t}\log\frac{\xi(\hat Z_t^{\ssup i})}{t^{\rho}}
=\frac{|\hat Z_t^{\ssup i}|}{r_t}\frac{\log(\xi(\hat Z_t^{\ssup i})/t^{\rho})}{\log t} .
\end{align*}
By part (i), as $t\to\infty$,
\begin{align*}
\frac{\log(\xi(\hat Z_t^{\ssup i})/t^{\rho})}{\log t}
=\frac{\log(\xi(\hat Z_t^{\ssup i})/a_t)-\rho\log\log t}{\log t}\to 0
\end{align*}
in probability and hence 
\begin{align*}
\frac{\hat\Psi(\hat Z_t^{\ssup i})}{a_t}-\frac{\Psi(Z_t^{\ssup i})}{a_t}
\to 0
\end{align*}
in probability. Observe that 
\begin{align*}
\frac{\hat\Psi(\hat Z_t^{\ssup i})}{a_t}=
\frac{\xi(\hat Z_t^{\ssup i})}{a_t}-\frac{|\hat Z_t^{\ssup i}|}{ta_t}\rho\log t
=\frac{\xi(\hat Z_t^{\ssup i})}{a_t}-\rho\frac{|\hat Z_t^{\ssup i}|}{r_t}
\end{align*}
and so 
\begin{align*}
\Big(\frac{\hat\Psi_t(\Zaa)}{a_t},
\frac{\hat\Psi_t(\Zbb)}{a_t}\Big)
\Rightarrow (Y^{\ssup 1}-\rho X^{\ssup 1},Y^{\ssup 2}-\rho X^{\ssup 2})
\end{align*}
by part (i). The result now follows from Lemma~\ref{coincide}.
\end{proof}

\begin{proof}[Proof of Proposition \ref{e0}]
We first show that \[
\Prob\left(\xi(z)<\frac{|z|}{t}\log\frac{|z|}{2et}\,\,\forall\, |z|>r_tg_t\right)\to1.
\]
Indeed there exists $c>0$ such that for each $z$ with $|z|>r_tg_t$, we have 
\[
\Prob\left(\xi(z) \ge \frac{|z|}{t}\log\frac{|z|}{2et}\right)<c\frac{t^\alpha}{(\log t)^\alpha}|z|^{-\alpha},
\]
and hence by a union bound there exists $c'>0$ such that
\[
\Prob\left(\xi(z) \ge \frac{|z|}{t}\log\frac{|z|}{2et}\,\,\mbox{ for some }\, |z|>r_tg_t\right)<c'g_t^{1-\alpha}\to0.
\] 
Next, it follows from Proposition~\ref{sl1} that 
\begin{align*}
\Prob\big(r_tf_t<|\Za|<r_tg_t, a_tf_t<\xi(\Za)<a_tg_t\big) \to 1
\end{align*}
and 
\begin{align*}
\Prob\big(\Psi_t(\Za)>f_t\xi(\Za)\big) \to 1
\end{align*}
since $f_t\to 0$ and $g_t\to\infty$.
Thus, it remains to show that
\begin{align}
\Prob\Big(\xi(\Za)=\xi_{R_t}^{\ssup 1}, \xi_{R_t}^{\ssup 1}-\xi_{R_t}^{\ssup 2}>a_tf_t\Big)
\to 1.
\end{align}
Let $\e>0$. 
By Proposition~\ref{xy2}, choose $a>0$ and the set 
\begin{align*}
A=\big\{(x,y): x \le 1/a,y>a+\rho x \big\} \subseteq ([0, \infty) \cup (0, \infty])^2
\end{align*}
so that
\begin{align}
\label{e6}
\Prob_*((X^{\ssup 1},Y^{\ssup 1})\in A)>1-\e
\end{align}
eventually. Let $0<\delta<\min\{\frac{a}{2},\frac{a^2}{2\rho }\}$. 
Observe that by Lemma~\ref{coincide}
\begin{align}
\label{e8}
\Prob&\Big(\xi(\Za)=\xi_{R_t}^{\ssup 1}, \xi_{R_t}^{\ssup 1}-\xi_{R_t}^{\ssup 2}>a_tf_t\Big)\\
&\ge 
\Prob\Big(\xi(\Zaa)=\xi^{\ssup{1}}_{(1+\delta)|\Zaa|}, 
\xi(\Zaa)-\xi^{\ssup{2}}_{(1+\delta)|\Zaa|}>\delta a_t\Big)-\e/4
\end{align}
eventually as $f_t\to 0$ and $R_t/|\Za|=1+f_t\to 1$. 
This implies 
\begin{align}
\Prob&\Big(\xi(\Za)=\xi_{R_t}^{\ssup 1}, \xi_{R_t}^{\ssup 1}-\xi_{R_t}^{\ssup 2}>a_tf_t\Big)\notag\\
&\ge \int_A\Prob\big(\bar \Pi_{r_t}(dx\times dy)=1, \bar \Pi_{r_t}(M_{\delta}(x,y))=0\big)-\e,
\label{e10}
\end{align}
where 
\begin{align*}
M_{\delta}(x,y)=\big\{(u,v):v-\rho u >y-\rho x \big\}
\cup\big\{(u,v): u <(1+\delta)x,y>y-\delta\big\}.
\end{align*} 
It is easy to see that 
\begin{align*}
M_{\delta}(x,y)\subseteq \{(u,v):v>a/2+\rho u \}
\end{align*}
for all $(x,y)\in A$, and this set has finite measure 
$\mu$ by~\eqref{mumu}. This implies 
\begin{align}
\int_A&\Prob\big(\bar \Pi_{r_t}(dx\times dy)=1, \bar \Pi_{r_t}(M_{\delta}(x,y))=0\big)\notag\\
&\to \int_A\Prob_*\big(\bar \Pi(dx\times dy)=1\big)
\Prob_*\big(\bar \Pi(M_{\delta}(x,y))=0\big)\notag\\
&=\int_A
(2-p) \exp\big\{- (2-p) \mu\big(M_{\delta}(x,y)\big)\big\}\mu(dx,dy).
\label{e5}
\end{align}
Since
\begin{align*}
\lim_{\delta\downarrow 0}\mu\big(M_{\delta}(x,y)\big)\to 
\mu\big(\big\{(u,v):v-\rho u>y-\rho x\big\}\big) 
\end{align*}
we obtain by the dominated convergence theorem that 
\begin{align*}
\lim_{\delta\downarrow 0}\int_A (2-p)
\exp\big\{-(2-p) \mu\big(M_{\delta}(x,y)\big)\big\}\mu(dx,dy)
\to \Prob_*((X^{\ssup 1},Y^{\ssup 2})\in A)>1-\e/4
\end{align*}
by~\eqref{e6}.
Hence for a sufficiently small $\delta$ we have 
\begin{align*}
\int_A
(2-p) \exp\big\{-(2-p) \mu\big(M_{\delta}(x,y)\big)\big\}\mu(dx,dy)>1-\e/2
\end{align*}
and using~\eqref{e5}
\begin{align*}
\int_A&\Prob\big(\bar \Pi_{r_t}(dx\times dy)=1, \bar \Pi_{r_t}(M_{\delta}(x,y))=0\big)>1-3\e/4
\end{align*}
eventually for all $t$. Combining this with~\eqref{e10} we obtain, eventually for all $t$,
\begin{equation*}
\Prob \Big(\xi(\Za)=\xi_{R_t}^{\ssup 1}, \xi_{R_t}^{\ssup 1}-\xi_{R_t}^{\ssup 2}>a_tf_t\Big)>1-\e . \qedhere
\end{equation*}
\end{proof}

We finish this section by proving Propositions \ref{prob} and \ref{td}. For this we need a few more definitions. For any $c>0$, denote by $M_c$ the space of point measures on $[-c,c]\times (0,\infty]$ equipped with the vague convergence, see~\cite[Ch.3]{resnick}. For any point measure $\Sigma$ on $\R\times (0,\infty]$, we denote by $\Sigma|_c$ the restriction of $\Sigma$ to $[-c,c]\times (0,\infty]$. Let 
$\Xi: M_c \times M_c \to [-c,c] \times (0,\infty)$ be defined by 
the property that $\Xi(0,0)=(0,0)$ and 
\begin{align*}
\Xi(\Sigma_1,\Sigma_2)=(x,y) \text{ if } 
(x,y)\in\Sigma_1+\Sigma_2\text{ and if }
(u,v)\in\Sigma_1+\Sigma_2\text{ then }v-\rho u <y-\rho x, 
\end{align*}
which is well-defined for any non-zero $(\Sigma_1,\Sigma_2)$
since each set $\{(u,v):|u|\le c, v>a+\rho |u|\}$, $a>0$, is compact in $[-c,c]\times (0,\infty]$ and hence $\Sigma_1+\Sigma_2$ has an most finitely many points there. By~\cite[Prop.~3.13]{resnick}, which states that vague convergence corresponds to the convergence of points carrying the point measures in each compact set, we obtain that $\Xi$ is continuous at all non-zero points. Let $\Xi_1$ and $\Xi_2$ be the projection of $\Xi$ onto its first and second coordinates respectively. Let 
$\Theta:M_c\times M_c\to \{0,1\}$ be defined by 
\begin{align*}
\Theta(\Sigma_1,\Sigma_2)=\left\{\begin{array}{ll}
1&\text{ if }\Xi (\Sigma_1,\Sigma_2)\in\Sigma_1,\\
0&\text{ otherwise.}
\end{array}\right.
\end{align*}
Now~\cite[Prop.~3.13]{resnick} implies that $\Theta$ is continuous at all non-zero points whenever $\Sigma_1$ and $\Sigma_2$ 
do not share points, that is, $\Sigma_1(\{x\})\Sigma_2(\{x\})=0$ for all $x$.

\begin{proof}[Proof of Proposition \ref{prob}]
Let $\e>0$ and let $c>0$ be such that
\begin{align*}
\Prob\big(|\Zaa|>cr_t\big)<\e/8 \qquad\text{and}\qquad \Prob_*(|X^{\ssup 1}|>c)<\e/8,
\end{align*}
which is possible by Proposition~\ref{sl1}. This implies  
\begin{align}
\label{pp1}
\big|\Prob(\mathcal{\hat D}_t)-\Prob\big(\Theta\big(\Pi_{r_t}^{\ssup d}|_c,\Pi_{r_t}^{\ssup e}|_c\big)=1\big)\big|<\e/4,
\end{align}
and 
\begin{align}
\label{pp2}
\big|\Prob_*(\mathfrak{D})-\Prob_*\big(\Theta\big(\Pi^{\ssup d}|_c,\Pi^{\ssup e}|_c\big)=1\big)\big|<\e/4.
\end{align}
By Lemma~\ref{lppp}
$(\Pi_{r_t}^{\ssup d}|_c,\Pi_{r_t}^{\ssup e}|_c)$ converges in law to $(\Pi^{\ssup d}|_c,\Pi^{\ssup e}|_c)$.
The point measures $\Pi^{\ssup d}|_c$ and $\Pi^{\ssup e}|_c$ almost surely do not share points and hence by the continuous mapping theorem~\cite[p.30]{B}
\begin{align*}
\Prob\big(\Theta\big(\Pi_{r_t}^{\ssup d}|_c,\Pi_{r_t}^{\ssup e}|_c\big)=1\big)
\to \Prob_*\big(\Theta\big(\Pi^{\ssup d}|_c,\Pi^{\ssup e}|_c\big)=1\big)
\end{align*}
as $t\to\infty$. Combining this with~\eqref{pp1} and~\eqref{pp2} we obtain that 
\begin{align*}
\big|\Prob(\mathcal{\hat D}_t)-\Prob_*(\mathfrak{D})\big|<\e
\end{align*}  
eventually, which together with Lemma~\ref{probb} completes the proof. 
\end{proof}

\begin{proof}[Proof of Proposition \ref{td}]
By Lemma~\ref{probb} and Proposition~\ref{prob} it suffices to show that
\begin{align*}
\Prob(\{S_t^{\ssup{\delta}}<x\}\cap \hat{\mathfrak{D}}_t)\to \Prob_*(\{S^{\ssup{\delta}}<x\}\cap \mathfrak{D})
\end{align*}
for all $x$ as $t\to\infty$. 

\smallskip 

Let $\e>0$ and let $c>0$ be such that
\begin{align}
\label{fff}
\Prob\big(|\Zaa|>cr_t\big)<\e/8 \qquad\text{and}\qquad \Prob_*(|X^{\ssup 1}|>c)<\e/8,
\end{align}
which is possible by Proposition~\ref{sl1}. 
Let $\Gamma:M_c\times M_c$ be defined by 
\begin{align*}
\Gamma(\Sigma_1,\Sigma_2)=
-\text{sgn}(\Xi_1(\Sigma_1,\Sigma_2))\Big[\!\!\!\!\!\!\!\!\!\!\!\!\!\!\!\!\!\!\!\!\!\!\!\!\sum_{\heap{(x,y)\in \Pi^{\ssup e}}{0<x<|\Xi_1(\Sigma_1,\Sigma_2)|,y\ge 
\delta \Xi_2(\Sigma_1,\Sigma_2)}}\!\!\!\!\!\!\!\!\!\!\!\!\!\!\!\!\!\!\!\!\!\!\!\!\log \Big(1-\frac{y}{\Xi_2(\Sigma_1,\Sigma_2)}\Big)
-\!\!\!\!\!\!\!\!\!\!\!\!\!\!\!\!\!\!\!\!\!\!\!\!\sum_{\heap{(x,y)\in \Pi^{\ssup e}}{-|\Xi_1(\Sigma_1,\Sigma_2)|<x<0,y\ge 
\delta \Xi_2(\Sigma_1,\Sigma_2)}}\!\!\!\!\!\!\!\!\!\!\!\!\!\!\!\!\!\!\!\!\!\!\!\!\log \Big(1-\frac{y}{\Xi_2(\Sigma_1,\Sigma_2)}\Big)\Big]
\end{align*}
Observe that $\Gamma$ is continuous at all points $(\Sigma_1,\Sigma_2)$ satisfying $\Xi_2(\Sigma_1,\Sigma_2)\neq 0$. 

\smallskip

Further,
$S_t^{\ssup{\delta}}=\Gamma(\Pi^{\ssup d}_{r_t}|_c,\Pi^{\ssup e}_{r_t}|_c)$ and $S^{\ssup{\delta}}=\Gamma(\Pi^{\ssup d}|_c,\Pi^{\ssup e}|_c)$,  provided that $|\Za|\le cr_t$ and $|X^{\ssup 1}|\le c$, respectively. 
Hence the inequalities~\eqref{fff} imply
\begin{align}
\label{qq1}
\big|\Prob(\{S_t^{\ssup{\delta}}<x\}\cap \hat{\mathfrak{D}}_t)-\Prob\big(\Gamma(\Pi^{\ssup d}_{r_t}|_c,\Pi^{\ssup e}_{r_t}|_c)<x,\Theta\big(\Pi_{r_t}^{\ssup d}|_c,\Pi_{r_t}^{\ssup e}|_c\big)=1\big)\big|<\e/4,
\end{align}
and 
\begin{align}
\label{qq2}
\big|\Prob_*(\{S^{\ssup{\delta}}<x\}\cap \mathfrak{D})-\Prob_*\big(\Gamma(\Pi^{\ssup d}|_{c},\Pi^{\ssup e}|_{c})<x, \Theta\big(\Pi^{\ssup d}|_c,\Pi^{\ssup e}|_c\big)=1\big)\big|<\e/4.
\end{align}
By Lemma~\ref{lppp}
$(\Pi_{r_t}^{\ssup d}|_c,\Pi_{r_t}^{\ssup e}|_c)$ converges in law to $(\Pi^{\ssup d}|_c,\Pi^{\ssup e}|_c)$.
Almost surely the point measures $\Pi^{\ssup d}|_c$ and $\Pi^{\ssup e}|_c$ do not share points and $\Xi_2(\Pi^{\ssup d}|_c,\Pi^{\ssup e}|_c)>0$. Hence by the continuous mapping theorem~\cite[p.30]{B}
\begin{align*}
\Prob\big(\Gamma(\Pi^{\ssup d}_{r_t}|_c,\Pi^{\ssup e}_{r_t}|_c)<x,\Theta\big(\Pi_{r_t}^{\ssup d}|_c,\Pi_{r_t}^{\ssup e}|_c\big)=1\big)
\to \Prob_*\big(\Gamma(\Pi^{\ssup d}|_c,\Pi^{\ssup e}|_c)<x, \Theta\big(\Pi^{\ssup d}|_c,\Pi^{\ssup e}|_c \big)=1\big)
\end{align*}
as $t\to\infty$. Combining this with~\eqref{qq1} and~\eqref{qq2} we obtain that, eventually,
\begin{equation*}
\big|\Prob(\{S_t^{\ssup{\delta}}<x\}\cap \hat{\mathfrak{D}}_t)-\Prob_*(\{S^{\ssup{\delta}}<x\}\cap \mathfrak{D})\big|<\e . \qedhere
\end{equation*}   
\end{proof}

\bigskip


\section{Central limit theorem}
\label{clt}

In this appendix we state and prove the central limit theorem that we apply in Section~\ref{s:main2}. This theorem is similar in spirit to the Lindeberg--Feller central limit theorem for triangular arrays, albeit in a slightly non-classical set-up. The notation used in this appendix is independent of the rest of the paper. 

\smallskip

For each $t>0$, let $\mathcal{F}_t$ be a $\sigma$-algebra, and $N_t$
be an $\mathcal{F}_t$-measurable $\N$-valued random variable. Denote by ${\mathrm E}_{\mathcal{F}_t}$
conditional expectation with respect to $\mathcal{F}_t$. Let $\{V_{t,i}: 1\le i\le N_t\}$, $t>0$, be a triangular array of random variables. For each $t>0$, denote 
\begin{align*}
V_t=\sum_{i=1}^{N_t}V_{t,i},
\end{align*} 
and let $F_{V_t}(x)={\mathrm E}_{\mathcal{F}_t}\one\{V_t\le x\}$ be the conditional distribution function of $V_t$.

\begin{theorem}
Let $\{\mathcal{E}_t:t>0\}$ be a family of events.  
Suppose that the following conditions hold:
\begin{enumerate}
\item 
For each $t>0$, conditionally on $\mathcal{F}_t$, the random variables $\{V_{t, i}:1 \le i \le N_t\}$ are independent;
\item 
For each $t>0$ and $ 1 \le i \le N_t$,
\begin{align*}
{\mathrm E}_{\mathcal{F}_t} V_{t, i} = 0  
\quad \text{and} \quad 
{\mathrm E}_{\mathcal{F}_t} V^2_t  = 1  
\quad \text{almost surely}  ;
\end{align*}
\item Lindeberg condition: 
For each $\varepsilon > 0$, as $t \to \infty$, 
\begin{align*}
\sum_{i=1}^{N_t}  
{\mathrm E}_{\mathcal{F}_t}\big[ V^2_{t, i} 
\id_{\{|V_{t, i}| \ge \varepsilon\}} \big] \id_{\mathcal{E}_t}  \to 0   \quad\text{almost surely.}
\end{align*}
\end{enumerate}
Then, as $t \to \infty$,
\[ \sup_{x \in \mathbb{R}} | F_{V_t}(x) - \Phi(x) | \id_{\mathcal{E}_t}   \to 0   \quad  \text{almost surely,}  \]
where $\Phi$ denotes the distribution function of a standard normal random variable.
\end{theorem}

\begin{proof}
Our proof is adapted from \cite[Th.~27.1]{B2}. First note that it is sufficient to prove that, for each 
$u \in \mathbb{R}$,
\[   
\big|  {\mathrm E}_{\mathcal{F}_t} 
\exp \left\{  i u V_t \right\}  
- \exp \{-u^2/2\} \big|    \id_{\mathcal{E}_t}   \to  0 \quad  \text{almost surely,}  \]
since by Levy's theorem this implies that, for each $x \in \mathbb{R}$, 
\[      | F_{V_t}(x) - \Phi(x) | \id_{\mathcal{E}_t}   \to 0   \quad\text{almost surely,}  \]
and moreover, by the continuity of $\Phi$, the pointwise convergence of the cumulative distribution functions necessarily takes place uniformly.

\smallskip

To proceed, abbreviate $\sigma^2_{t, i} 
= {\mathrm E}_{\mathcal{F}_t}V^2_{t, i}$, and use conditions $(1)$ and $(2)$ to write 
\begin{align*}
\big|  {\mathrm E}_{\mathcal{F}_t} 
\exp \left\{  i u V_t \right\}   
- \exp \left\{-u^2/2\right\} \big|   
& = \Big|{\mathrm E}_{\mathcal{F}_t} 
\exp \Big\{\sum_{i = 1}^{N_t}  i u V_{t,i} \Big\} 
-  \exp \Big\{ - \sum_{i = 1}^{N_t} u^2 \sigma_{t, i}^2 /2 \Big\} \Big|       \\
& = \Big|  \prod_{i = 1}^{N_t} {\mathrm E}_{\mathcal{F}_t} 
\exp \left\{  i u V_{t,i} \right\}  - \prod_{i = 1}^{N_t} \exp \left\{  -u^2 \sigma_{t, i}^2 /2 \right\} \Big|       \\
& \le \sum_{i = 1}^{N_t} \left|  {\mathrm E}_{\mathcal{F}_t} \exp \left\{  i u V_{t,i} \right\}   
-  \exp \left\{  -u^2 \sigma_{t, i}^2  /2 \right\} \right|  ,
\end{align*}
where in the last step we used the fact that, for complex numbers $\{z_i\}_{1 \le i \le n}$ and  $\{z'_i\}_{1 \le i \le n}$ of modulus at most $1$,
\[ \Big| \prod_{i = 1}^n z_i  - \prod_{i=1}^n z'_i \Big| \le \sum_{i = 1}^n |z_i - z'_i|. \]
Hence, applying the triangle inequality,
\[
\big|  {\mathrm E}_{\mathcal{F}_t} \exp \left\{  i u V_{t,i} \right\}   
-  \exp \big\{  -u^2 \sigma_{t, i}^2  /2 \big\} \big| \id_{\mathcal{E}_t}   \le A_t + B_t, \]
where 
\[A_t =  \sum_{i = 1}^{N_t} \left|  {\mathrm E}_{\mathcal{F}_t} \left[ \exp \left\{  i u V_{t,i} \right\}  -  \left( 1 - u^2 V^2_{t, i} /2 \right) \right]   \right|  \id_{\mathcal{E}_t}\]
and
\[B_t = \sum_{i = 1}^{N_t} \left|   \exp \left\{  -u^2 \sigma_{t, i}^2 /2 \right\} - \left( 1 - u^2 \sigma^2_{t, i}/2 \right)   \right| \id_{\mathcal{E}_t}  . \]
It remains to show that each of $A_t$ and $B_t$ converge to zero almost surely.

\smallskip

Proceeding first with $A_t$, we apply Taylor's inequality \cite[Eq.~($26.4_2$)]{B2}
\[   \left|  \exp \{ ix \} - \left( 1 + ix - x^2/2 \right) \right| \le \min \{ |x^2|, |x^3| \}  , \quad x \in \R, \]
and condition $(2)$ to bound
\[ A_t \le  \sum_{i = 1}^{N_t}{\mathrm E}_{\mathcal{F}_t} \left[  \min\big\{ |u V_{t, i}|^2, |u V_{t, i}|^3  \big\}\right]\id_{\mathcal{E}_t}  . \]
Fixing $\varepsilon > 0$, we then have
\begin{align*}
A_t
&\le \sum_{i = 1}^{N_t}\Big\{
{\mathrm E}_{\mathcal{F}_t} \left[ |u V_{t, i}|^2 
\id_{\{|V_{t, i}| \ge \varepsilon\}} \right] \id_{\mathcal{E}_t} +  {\mathrm E}_{\mathcal{F}_t} \left[  |u V_{t, i}|^3 \id_{\{|V_{t, i}|  < \varepsilon\}}   \right] \id_{\mathcal{E}_t} \Big\}  \\
& \le u^2\sum_{i = 1}^{N_t}{\mathrm E}_{\mathcal{F}_t}[ V^2_{t, i} \id_{\{|V_{t, i}| \ge \varepsilon\}} ] \id_{\mathcal{E}_t} 
+\varepsilon |u|^3 \sum_{i = 1}^{N_t}\sigma_{t, i}^2   .
 \end{align*}
The second term equals $\varepsilon |u|^3$ by $(2)$, 
and the first term tends to zero by  $(3)$, which proves that $A_t\to 0$ almost surely as $t\to\infty$ as $\e$ was arbitrary.

\smallskip

Turning then to $B_t$, we use the fact that, for $x \ge 0$,
\[ \left| e^{-x} - 1 +x\right| \le x^2/2, \]
to bound
\[ B_t \le  \frac{u^4}{8}  \sum_{i = 1}^{N_t} 
\sigma^4_{t, i}     
\le  \frac{u^4}{8}  \max_{1 \le i \le N_t} \sigma^2_{t, i}  \sum_{i = 1}^{N_t} 
 \sigma^2_{t, i}    . \]
The sum equals one according to $(2)$. 
Fixing $\varepsilon > 0$, we also have that
\begin{align*}
\max_{1 \le i \le N_t} \sigma^2_{t, i}  
&\le \max_{1 \le i \le N_t} \Big\{
{\mathrm E}_{\mathcal{F}_t}[ V^2_{t, i} \id_{\{|V_{t, i}| < \varepsilon\}} ]  
+ {\mathrm E}_{\mathcal{F}_t}[ V^2_{t, i} 
\id_{\{|V_{t, i}| \ge \varepsilon\}} ]    
\Big\} \id_{\mathcal{E}_t} \\
& \le \varepsilon^2 +  \sum_{1 \le i \le N_t}  
{\mathrm E}_{\mathcal{F}_t}[ V^2_{t, i} \id_{\{|V_{t, i}| \ge \varepsilon\}}  ]\id_{\mathcal{E}_t}   .
   \end{align*}
Applying condition $(3)$ we have the result, since $\varepsilon > 0$ was arbitrary.
\end{proof}

\bigskip

\printbibliography

\end{document}